\documentclass[11pt]{article}
\usepackage[a4paper]{geometry}
\usepackage{amsthm}
\usepackage{amsmath}
\usepackage{amssymb}
\usepackage{amsfonts}
\usepackage{graphicx}
\usepackage{color}
\usepackage[bf,SL,BF]{subfigure}
\usepackage{url}
\usepackage{epstopdf}
\usepackage{pdfsync}
\usepackage[colorlinks]{hyperref}
\usepackage{comment}

\usepackage{accents}
\newcommand{\ubar}[1]{\underaccent{\bar}{#1}}

\hypersetup{
    linkcolor=blue,
}

\usepackage{epsfig,amsbsy,graphicx,multirow}

\usepackage[all]{xy}

\usepackage{ algorithm, algorithmic}
\renewcommand{\algorithmiccomment}[1]{\bgroup\hfill//~#1\egroup}

 \numberwithin{equation}{section}

\newlength{\fixboxwidth}
\def\V{\mathfrak{V}}

\def\R{\mathbb{R}}

\def\N{\mathcal{N}}

\def\E{\mathbb{E}}

\def\F{\mathcal{F}}

\def\I{\mathcal{I}}
\def\J{\mathcal{J}}
\def\L{\mathcal{L}}

\def\W{\mathfrak{W}}

\def\T{\mathcal{T}}

\def\<{\big\langle}
\def\>{\big\rangle}
\def\Img{\operatorname{Im}}
\def\Inv{\operatorname{Inv}}
\def\Ker{\operatorname{Ker}}

\def\Cond{\operatorname{Cond}}
\def\Span{\operatorname{span}}

\def\diiv{\operatorname{div}}
\def\dist{\operatorname{dist}}

\def\Tr{\operatorname{Trace}}

\def\diam{\operatorname{diam}}

\def\supp{{\operatorname{support}}}
\def\dim{{\operatorname{dim}}}

\def\dx{\,{\rm d}x}

\def\loc{{\rm loc}}
\def\homo{{\rm hom}}
\def\er{\mathcal{E}}

\def\app{{\mathrm{app}}}

\definecolor{red}{rgb}{0.9, 0, 0}

\newtheorem{Theorem}{Theorem}[section]
\newtheorem{Proposition}[Theorem]{Proposition}
\newtheorem{Lemma}[Theorem]{Lemma}
\newtheorem{Corollary}[Theorem]{Corollary}
\newtheorem{Remark}[Theorem]{Remark}
\newtheorem{Example}{Example}[section]
\newtheorem{Definition}[Theorem]{Definition}
\newtheorem{Construction}[Theorem]{Construction}

\begin{document}
\title{Multigrid with rough coefficients\\
and Multiresolution operator decomposition\\ from Hierarchical Information Games}

\date{\today}

\author{Houman Owhadi\footnote{California Institute of Technology, Computing \& Mathematical Sciences , MC 9-94 Pasadena, CA 91125, owhadi@caltech.edu}}

\maketitle

\begin{abstract}
We introduce a near-linear complexity (geometric and meshless/algebraic) multigrid/multiresolution method for PDEs with rough ($L^\infty$) coefficients with rigorous a-priori accuracy and performance estimates.
The method is discovered through a decision/game theory formulation of the problems of (1) identifying restriction and interpolation operators (2) recovering a signal from incomplete measurements based on norm constraints on its image under a linear operator (3) gambling on the value of the solution of the PDE based on a hierarchy of nested measurements of its solution or source term.
The resulting elementary gambles form a hierarchy of (deterministic)  basis functions of $H^1_0(\Omega)$  (gamblets) that (1) are orthogonal  across subscales/subbands with respect to the scalar product induced by the energy norm of the PDE (2) enable sparse compression of the solution space in $H^1_0(\Omega)$ (3) induce an orthogonal multiresolution operator decomposition.
The operating diagram of the multigrid method is that of an inverted pyramid in which gamblets are computed locally (by virtue of their exponential decay), hierarchically (from fine to coarse scales) and the PDE is decomposed into a hierarchy of independent linear systems with uniformly bounded condition numbers. The
resulting algorithm is  parallelizable both in space (via localization) and in bandwith/subscale (subscales can be computed independently from each other). Although the method is deterministic it has a natural Bayesian interpretation under the measure of probability emerging (as a mixed strategy) from the information game formulation and multiresolution approximations form a martingale with respect to the filtration induced by the hierarchy of nested measurements.
\end{abstract}

\section{Introduction}\label{secdt}
\subsection{Scientific discovery as a decision theory problem}
The process of scientific discovery is oftentimes based on intuition, trial and error and plain guesswork. This paper is motivated by the question of the existence of a rational decision framework that could be used to facilitate/guide this process, or turn it, to some degree, into an algorithm.
In exploring this question, we will  consider the problem of finding a method for solving (up to a pre-specified level of accuracy)
PDEs with rough ($L^\infty$) coefficients as fast as possible with the following prototypical  PDE (and its possible discretization over a fine mesh) as an  example
\begin{equation}\label{eqn:scalar}
\begin{cases}
    -\diiv \big(a(x)  \nabla u(x)\big)=g(x) \quad  x \in \Omega;\, g \in L^2(\Omega),\text{ or }g \in H^{-1}(\Omega) \\
    u=0 \quad \text{on}\quad \partial \Omega,
    \end{cases}
\end{equation}
where $\Omega$ is a bounded subset of $\R^d$ (of arbitrary dimension $d\in \mathbb{N}^*$) with piecewise Lipschitz boundary, $a$ is a symmetric, uniformly elliptic $d\times d$ matrix with entries in $L^\infty(\Omega)$ and such that for all $x\in \Omega$ and $l\in \R^d$,
\begin{equation}
\lambda_{\min}(a) |l|^2 \leq l^T a(x) l \leq \lambda_{\max}(a)|l|^2.
\end{equation}
Although multigrid methods \cite{Fedorenko:1961, Brandt:1973, Hackbusch:1978, Hackbusch:1985, Stuben:2001} are now well known as the fastest for solving elliptic boundary-problems and have successfully been generalized to other types of PDEs and computational problems \cite{Yavneh:2006}, their convergence rate can be severely affected by the lack of regularity of the coefficients  \cite{EngquistLuo:1997, wan2000}. Furthermore, although significant  progress  has been achieved in the development of multigrid methods that are, to some degree, robust with respect to meshsize and  lack of smoothness
(we refer in particular to algebraic multigrid \cite{RugeStuben1987}, multilevel finite element splitting \cite{Yserentant1986}, hierarchical basis multigrid \cite{BankYserentant88, ChowVassilevski2003},  multilevel preconditioning \cite{Vassilevski89}, stabilized hierarchical basis methods \cite{Vassilevski1997sirev, VassilevskiWang1997a, VassilevskiWang1998},   energy minimization  \cite{Mandel1999, wan2000, Xu2004, XuZhu2008, Vassilevski2010} and homogenization based methods \cite{EngquistLuo:1997, Efendiev2011}),  the design of multigrid methods that are provably robust with respect to rough ($L^\infty$) coefficients has remained an open problem of practical importance \cite{BraWu09}.

Alternative hierarchical strategies for the resolution of \eqref{eqn:scalar} are (1) wavelet based methods \cite{Beylkin:1995,Beylkin:1998,Beylkin:1998b, DorobantuEngquist1988, EngquistRunborg02}  (2)
the Fast Multipole Method \cite{GreengardRokhlin:1987} and (3) Hierarchical matrices \cite{HackbuschGrasedyck:2002, Bebendorf:2008}. Although methods based on (classical) wavelets  achieve a multiresolution compression of the solution space of \eqref{eqn:scalar} in $L^2$ and although approximate wavelets and approximate $L^2$  projections can stabilize  hierarchical basis methods \cite{VassilevskiWang1997a, VassilevskiWang1998}, their applications to \eqref{eqn:scalar} are limited by the facts that (a) the  underlying wavelets  can perform arbitrarily badly  \cite{BaOs:2000} in their $H^1_0(\Omega)$ approximation of the solution space and (b) the operator \eqref{eqn:scalar} does not preserve the orthogonality between subscales/subbands with classical wavelets.
The Fast Multipole Method and hierarchical matrices exploit the property that sub-matrices of the inverse discrete operator are low rank away from the diagonal. This low rank property can be rigorously proven   for \eqref{eqn:scalar} (based on the approximation of its Green's function by sums of products of harmonic functions \cite{Bebendorf2005}) and leads to  provable convergence (with rough coefficients), up to the pre-specified level of accuracy $\epsilon$ in $L^2$-norm,  in $\mathcal{O}(N \ln^{6} N \ln^{2d+2} \frac{1}{\epsilon})$ operations (\cite{Bebendorf2005} and \cite[Thm.~2.33 and Thm.~4.28]{Bebendorf:2008}).
Can the problem of finding a fast solver for \eqref{eqn:scalar} be, to some degree, reformulated as an Uncertainty Quantification/Decision Theory problem that could, to some degree, be solved as such in an automated fashion? Can discovery be computed?
Although these questions may seem unorthodox  their answer appears to be positive: this paper shows that
this reformulation is possible and leads to a  multigrid/multiresolution method/algorithm  solving  \eqref{eqn:scalar}, up to the pre-specified level of accuracy $\epsilon$ in $H^1$-norm (i.e. finding $u^{\app}$ such that $\|u-u^{\app}\|_{H^1_0(\Omega)}\leq \epsilon \|g\|_{H^{-1}(\Omega)}$ for an arbitrary $g$ decomposed over $N$ degrees of freedom), in $\mathcal{O}\big(N \ln^{3d}\big( \max(\frac{1}{\epsilon},N^{1/d})\big)\big)$ operations (for $\epsilon\sim N^{-1/d}$, the hierarchical matrix method achieves $\epsilon$-accuracy in $L^2$ norm in $\mathcal{O}(N \ln^{2d+8} N)$ operations and the proposed multiresolution method achieves  $\epsilon$-accuracy in $H^1$ norm in $\mathcal{O}(N \ln^{3d} N)$ operations). For subsequent solves (i.e. if \eqref{eqn:scalar} needs to be solved for more than one $g$) then the proposed multiresolution method achieves accuracy $\epsilon \approx N^{-\frac{1}{d}}$ in $H^1$-norm in $\mathcal{O}(N \ln^{d+1} N)$ operations (we refer to Subsection \ref{subseccomplexity} and in particular to Table \ref{tabcomplexity} for a detailed complexity analysis of the proposed method, which can also achieve sublinear complexity if one only requires $L^2$-approximations).

The core mechanism supporting the  complexity of the  method presented here is the fast decomposition of $H^1_0(\Omega)$ into a direct sum of linear subspaces that are orthogonal (or near-orthogonal)  with respect to the energy scalar product and over which \eqref{eqn:scalar} has uniformly bounded condition numbers.  It is, to some degree, surprising that  this decomposition can be achieved in near linear complexity and not in the complexity of an eigenspace decomposition. Naturally \cite{OwZh:2016}, this  decomposition can be applied to the fast simulation of the wave and parabolic equations associated to \eqref{eqn:scalar}  or to its fast diagonalization.

The essential step behind the automation of the discovery/design of scalable numerical  solvers is the observation that fast computation requires repeated computation with partial information (and limited  resources) over hierarchies of levels of complexity and the reformulation of this process as that of playing underlying hierarchies of adversarial information games \cite{VNeumann28, VonNeumann:1944}.

Although the problem of finding a fast solver for \eqref{eqn:scalar} may appear disconnected from that of finding statistical estimators or making decisions from data sampled from an underlying unknown probability distribution, the proposed  game theoretic reformulation is, to some degree, analogous to the one developed in Wald's Decision Theory \cite{Wald:1945}, evidently influenced by Von Neumann's Game Theory \cite{VNeumann28, VonNeumann:1944} (the generalization of worst case Uncertainty Quantification analysis \cite{ouq2010} to sample data/model uncertainty requires an analogous game theoretic formulation \cite{OwhadiScovel:2015w}, see also  \cite{OwhadiScovel:2013} for how the underlying calculus  could be used to guide the discovery of new Selberg identities).
We also refer to subsection \ref{subsecone3} for a review of the correspondence between statistical inference and numerical approximation.

\subsection{Outline of the paper}
The essential difficulty in generalizing the multigrid concept to PDEs with rough coefficients lies in the fact that the interpolation (downscaling) and restriction (upscaling) operators are, a priori, unknown. Indeed, in this situation,  piecewise linear finite-elements  can perform arbitrarily badly  \cite{BaOs:2000} and the design of the interpolation operator requires the  identification of accurate basis elements adapted to the microstructure $a(x)$.

This identification problem has also been the essential difficulty in numerical homogenization
 \cite{WhHo87, BaOs:1983, BaCaOs:1994, Beylkin:1995, HoWu:1997, EEngquist:2003,   OwZh:2007a, BraWu09}.
 Although inspired by classical homogenization ideas and concepts (such as oscillating test functions \cite{Mur78, EfHo:2007, EfGiHouEw:2006}, cell problems/correctors and effective coefficients \cite{BeLiPa78,  PapanicolaouVaradhan:1981, Abdulle:2004,NoPaPi:2008, EnSou08, AnGlo06}, harmonic coordinates \cite{Kozlov:1979, BaOs:1983, BaCaOs:1994, Owhadi:2003, BenarousOwhadi:2003, AllBri05, OwZh:2007a}, compactness by compensation \cite{Spagnolo:1968,Gio75, Muratb:1978, BeOw:2010}) an essential goal of numerical homogenization has been the numerical approximation of the solution space of \eqref{eqn:scalar} with arbitrary rough coefficients \cite{OwZh:2007a}, i.e., in particular, without the assumptions found in classical homogenization, such as scale separation, ergodicity at fine scales and $\epsilon$-sequences of operators (otherwise the resulting method could lack robustness to rough coefficients, even under the assumption that coefficients are stationary \cite{BalJing10}).
Furthermore, to envisage applications to multigrid methods, the computation of these basis functions must also be provably localized  \cite{ BaLip10, OwZh:2011,  MaPe:2012, GrGrSa2012, OwhadiZhangBerlyand:2014, HouLiu2015} and compatible with nesting strategies \cite{OwhadiZhangBerlyand:2014}.
In \cite{Owhadi:2014}, it has been shown that this process of identification (of accurate basis elements for numerical homogenization), could, in principle, be guided through its reformulation as a Bayesian Inference problem in which the source term $g$ in \eqref{eqn:scalar} is replaced by noise $\xi$ and one tries to estimate the value of the solution at a given point based on a finite number of observations. In particular it was found that Rough Polyharmonic Splines \cite{OwhadiZhangBerlyand:2014} and Polyharmonic Splines  \cite{Harder:1972, Duchon:1976,Duchon:1977,Duchon:1978} can be re-discovered as  solutions of  Gaussian filtering problems. This paper is inspired by the suggestion that this link between numerical homogenization and Bayesian Inference (and the link between Numerical Quadrature  and Bayesian Inference \cite{Poincare:1896, Diaconis:1988, Shaw:1988, Hagan:1991, Hagan:1992}) are not coincidences but particular instances of mixed strategies for underlying information games and that optimal or near optimal methods could be obtained by identifying such games and their optimal strategies.

The process of identification of these games starts with the (Information Based Complexity \cite{Woniakowski2009}) notion that computation can only be done with partial information.
For instance, since the operator \eqref{eqn:scalar} is infinite dimensional, one cannot directly compute with $u\in H^1_0(\Omega)$ but only with finite-dimensional \emph{features} of $u$. An example of such finite-dimensional features is the $m$-dimensional vector $u_m:=(\int_{\Omega}u\phi_1,\ldots,\int_{\Omega}u\phi_m)$ obtained by integrating  the solution $u$ of \eqref{eqn:scalar} against $m$ test/measurement functions $\phi_i \in L^2(\Omega)$. However to achieve an accurate approximation of $u$ through computation with $u_m$ one must fill the information gap between $u_m$ and $u$ (i.e. construct an interpolation operator giving  $u$ as a function of $u_m$). We will, therefore, reformulate the identification of this interpolation operator  as a non-cooperative (min max) game where Player I chooses the source term $g$ \eqref{eqn:scalar} in an admissible set/class (e.g. the unit ball of $L^2(\Omega)$) and Player II is shown $u_m$ and must approximate $u$ from these incomplete measurements. Using the energy norm
\begin{equation}\label{eqenergynorm}
\|u\|_a^2:=\int_{\Omega}\nabla u^T(x) a(x)\nabla u(x)\,dx,
\end{equation}
to quantify the accuracy of the recovery and calling $u^*$ Player I's bet (on the value of $u$),
the objective of Player I is to maximize the approximation error $\|u-u^*\|_{a}$, while the objective of Player II is to minimize it.
A remarkable result from Game Theory (as developed by Von Neumann \cite{VNeumann28}, Von Neumann and Morgenstern \cite{VonNeumann:1944} and Nash \cite{Nash:1951}) is that optimal strategies for deterministic zero sum finite games are mixed (i.e. randomized) strategies. Although the information game described above is zero sum, it is not finite. Nevertheless, as in Wald's Decision Theory \cite{Wald:1945}, under sufficient regularity conditions it can be made compact and therefore approximable by a finite game.
Therefore although the information game described above is purely deterministic (and has no a priori connection to statistical estimation), under compactness (and continuity of the loss function), the best strategy for Player I is to play at random by placing a  probability distribution $\pi_{I}$ on the set of candidates for $g$ (and select $g$ as a sample from $\pi_{I}$) and the optimal strategy for Player II is to place a  probability distribution $\pi_{II}$ on the set of candidates for $g$ and approximate the solution of \eqref{eqn:scalar} by the expectation of $u$ (under $\pi_{II}$ used as a prior distribution)  conditioned on the measurements $\int_{\Omega} u\phi_i$.

Although the estimator employed by Player II may be called Bayesian, the game described here is not (i.e. the choice of Player I might be distinct from that of Player II) and Player II must solve a min max optimization problem over $\pi_{I}$ and $\pi_{II}$ to identify an optimal prior distribution for the Bayesian estimator (a careful choice of the prior also appears to be important due to the possible high sensitivity of posterior distributions \cite{OSS:2013, OwhadiScovel:2013, owhadiBayesiansirev2013}).
Although solving the min max problem over $\pi_{I}$ and $\pi_{II}$ may be one way of determining the strategy of Player II, it will not be the method employed here. We will instead analyze the  error of Player II's approximation as a function of Player II's prior and the source term $g$ picked by Player I. Furthermore, to preserve the linearity of the calculations we will restrict Player II's decision space (the set of possible priors $\pi_{II}$) to Gaussian priors on the source term $g$.
Since the resulting analysis is independent of the structure of \eqref{eqn:scalar} and solely depends on its linearity we will first perform this investigation, in Section \ref{sec:problem}, in the algebraic framework of linear systems of equations,  identify Player II's  optimal mixed strategy  and show that it is characterized by deterministic optimal recovery and accuracy properties.
The mixed strategy identified in Section \ref{sec:problem} will then be applied in
 \ref{sec:contcase} to the numerical homogenization of \eqref{eqn:scalar} and the discovery of interpolation interpolators. In particular, it will be shown that the resulting elementary gambles form a set of  deterministic basis functions (gamblets) characterized by (1)  optimal recovery and accuracy properties (2)  exponential decay (enabling their localized computation) (3) robustness to high contrast.

To compute fast, the game presented above must not be limited to filling the information gap between $u_m\in \R^m$ and $u\in H^1_0(\Omega)$. This game must be played (and repeated) over hierarchies of levels of complexity (e.g.  one must fill  information gaps between $\R^{4}$ and $\R^{16}$, then $\R^{16}$ and $\R^{64}$, etc...). We will therefore, in Section \ref{sechnh}, consider the (hierarchical) game where Player I chooses the r.h.s of  \eqref{eqn:scalar} and Player II must (iteratively) gamble on the value of its solution  based on a hierarchy of nested measurements of $u$ (from coarse to fine measurements).
 Under Player II's mixed strategy (identified in Section \ref{sec:problem}  and used in Section \ref{sec:contcase}), the resulting
sequence of multi-resolution approximations  forms a martingale. Conditioning and the independence of martingale increments lead to the hierarchy of nested interpolation operators and to the multiresolution orthogonal decomposition of \eqref{eqn:scalar} into independent linear systems of uniformly bounded condition numbers.
The resulting elementary gambles (gamblets) (1) form a hierarchy of nested basis functions leading to the orthogonal decomposition (in the scalar product of the energy norm) of $H^1_0(\Omega)$ (2) enable the sparse compression of the solution space of  \eqref{eqn:scalar} (3) can be computed and stored in near-linear complexity by solving a nesting of linear systems with uniformly bounded condition numbers (4) enable the computation of the solution of \eqref{eqn:scalar} (or its hyperbolic or parabolic analogues) in near-linear complexity. The  implementation and complexity of the algorithm are discussed in Section \ref{secnumimple} with numerical illustrations.

\subsection{On the correspondence between statistical inference and numerical approximation}\label{subsecone3}
As exposed by Diaconis \cite{Diaconis:1988}, the investigation of the correspondence between statistical inference and numerical approximation can be traced back to  Poincar{\'e}'s course in Probability Theory  \cite{Poincare:1896}. It is useful to recall Diaconis' compelling example \cite{Diaconis:1988} as an illustration of this conection. Let $f:[0,1]\rightarrow \R$ be a given function and assume that we are interested in the numerical approximation of $\int_0^1 f(t)\,dt$. The Bayesian approach to this quadrature problem  is to (1) Put a prior (probability distribution) on continuous functions $\mathcal{C}[0,1]$ (2) Calculate $f$ at $x_1,x_2,\ldots,x_n$ (to obtain the data $(f(x_1),\ldots,f(x_n))$) (3) Compute a posterior (4) Estimate $\int_0^1 f(t)\,dt$ by the Bayes rule. If the prior on $\mathcal{C}[0,1]$ is that of a Brownian Motion (i.e. $f(t)=B_t$ where $B_t$ is a Brownian motion and $B_0$ is normal), then $\E\big[f(x)\big|f(x_1),\ldots,f(x_n)\big]$ is the piecewise linear interpolation of $f$ between the points $x_1,\ldots,x_n$ and one re-discovers the trapezoidal quadrature rule. If the prior on $\mathcal{C}[0,1]$ is that of the first integral of a Brownian Motion (i.e. $f(t)\sim \int_0^t B_s\,ds$) then the posterior $\E\big[f(x)\big|f(x_1),\ldots,f(x_n)\big]$ is the cubic spline interpolant and integrating $k$ times yields splines of order $2k+1$.

Subsequent to Poincar{\'e}'s early discovery  \cite{Poincare:1896},  Sul'din \cite{Suldin1959} and (in particular)  Larkin \cite{Larkin1972} initiated the systematic investigation of the correspondence between conditioning Gaussian measures/processes and numerical approximation. As noted by Larkin \cite{Larkin1972}, despite Sard's introduction of probabilistic concepts in the theory of linear approximation \cite{Sard1963}, and Kimeldorf and Wahba's exposition \cite{Kimeldorf70} of the correspondence between Bayesian estimation and spline smoothing/interpolation, the application of probabilistic concepts and techniques to numerical integration/approximation ``attracted little attention among numerical analysts'' (perhaps due to the counterintuitive nature of the process of randomizing a known function). However, a natural framework for understanding this  process of randomization can be found
in the pioneering works of Wo{\'z}niakowski \cite{Woniakowski1986}, Packel \cite{Packel1987}, and Traub, Wasilkowski,
and Wo\'{z}niakowski \cite{Traub1988} on Information Based Complexity  \cite{Nemirovsky1992, Woniakowski2009}, the branch of computational complexity that studies the complexity of approximating continuous mathematical operations with discrete and finite ones up a to specified level of accuracy.
Indeed the concept  that numerical implementation requires computation with partial information and limited resources emerges naturally from Information Based Complexity, where it is also augmented by concepts of contaminated and priced information associated with, for example, truncation errors and the cost of numerical operations. In this framework, the performance of an algorithm operating on incomplete information can be analysed in the usual worst case setting or the average case (randomized) setting \cite{Ritter2000, Novak2010} with respect to the missing information.
Although the measure of probability (on the solution space) employed in the average case setting  may be arbitrary,
as observed by Packel \cite{Packel1987}, if that measure is chosen carefully (as the solution of a game theoretic problem) then
the average case setting can be interpreted as lifting a (worst case) min max problem (where saddle points of pure strategies do not, in general, exist) to a min max problem over mixed (randomized) strategies (where saddle points do exist \cite{VNeumann28, VonNeumann:1944}).
As exposed by Diaconis \cite{Diaconis:1988} (see also Shaw \cite{Shaw:1988}) the randomized setting  also establishes a correspondence
between Numerical Analysis and Bayesian Inference  providing a natural framework for the statistical description of numerical errors (in which confidence intervals can be derived from posterior distributions). Furthermore \cite{PalastiRenyi1956, Diaconis:1988}, classical  min max numerical quadrature rules  can be formulated as solutions of  Bayesian inference problems with carefully chosen priors \cite{Diaconis:1988} and, as shown by O'Hagan \cite{Hagan:1991, Hagan:1992},  this correspondence can be exploited to discover new and useful numerical quadratures rules.
As envisioned by Skilling \cite{Skilling1992}, by placing a (carefully chosen) probability distribution  on the solution space of an ODE  and conditioning on quadrature points, one obtains a posterior distribution on the solution whose mean may coincide with classical numerical integrators such as Runge-Kutta methods \cite{schober2014nips}. As shown in \cite{ChkrebtiiCampbell2015} the statistical approach is particularly well suited for chaotic dynamical systems for which deterministic worst case error bounds may provide little information.
 While in \cite{Skilling1992, schober2014nips, ChkrebtiiCampbell2015} the probability distribution is directly placed on the solutions space, for PDEs  \cite{Owhadi:2014} argues that the prior distribution must be placed on source terms
(or on the image space of an integro-differential operator) and propagated/filtered through the inverse operator to reflect the structure of the solution space. In particular \cite{Owhadi:2014} shows that this process of filtering noise with the inverse operator, when combined with conditioning, produces accurate finite-element basis functions for the solution space whose  deterministic  worst case errors can be bounded by standard deviation errors using the reproducing kernel structure of the covariance function of the filtered Gaussian field.  As already witnessed in \cite{ChkrebtiiCampbell2015, schober2014nips, Owhadi:2014,Hennig2015, Hennig2015b, Briol2015, Conrad2015}, it is natural to expect that the possibilities offered by combining numerical uncertainties/errors with  model uncertainties/errors in a unified framework  will stimulate a resurgence of  the statistical inference approach to numerical analysis.

\section{Linear Algebra with incomplete information}\label{sec:problem}
\subsection{The recovery problem}
The problem of identifying interpolation operators for \eqref{eqn:scalar} is equivalent (after discretization or in the algebraic setting) to that of recovering or approximating the
solution of a linear system of equations from an incomplete set of measurements (coarse variables) given known norm constraints on the image of the solution.

Let $n\geq 2$ and $A$ be a known real invertible $n\times n$ matrix.   Let $b$ be an unknown element of $\R^n$.
Our purpose is to approximate the solution $x$ of
\begin{equation}\label{eqn:scalargeneral}
A x=b
\end{equation}
based  on the information that (1) $x$ solves
\begin{equation}\label{eq:ekiudiu3d}
\Phi x =y,
\end{equation}
where $\Phi$ (the measurement matrix) is a known, rank $m$, $m\times n$ real matrix  such that $m<n$ and  $y$ (the measurement vector) is a known vector of $\R^m$, and (2) the norm $b^T T^{-1} b$ of $b$ is known or bounded by a known constant (e.g., $b^T T^{-1} b \leq 1$), where $T^{-1}$ is a known positive definite  $n\times n$ matrix (with $T^{-1}$ being the identity matrix as a prototypical example).
Observe that since $m<n$, the measurements \eqref{eq:ekiudiu3d} are, a priori, not sufficient to recover the exact value $x$.

As described in Section \ref{secdt}, by formulating this recovery problem as a (non-cooperative) information game (where Player I chooses $b$ and
 Player II chooses an approximation $x^*$ of $x$ based on the observation $\Phi x$), one (Player II) is naturally lead to search for mixed strategy in the Bayesian class by placing a prior distribution on $b$. The purpose of this section is to analyze the resulting approximation error and select the prior distribution accordingly. To preserve the linearity (i.e. simplicity and computational efficiency) of calculations we will restrict Player II's decision space to Gaussian priors.

\subsection{Player I's mixed strategy}
We will therefore, in the first step of the analysis, replace $b$ in \eqref{eqn:scalargeneral} by $\xi$, a centered Gaussian vector of $\R^n$ with covariance matrix $Q$ (which may be distinct from $T$) and consider the following stochastic linear system
\begin{equation}\label{eqn:noisy}
A X=\xi\,.
\end{equation}
The Bayesian answer (a mixed strategy for Player II) to the recovery problem of Section \ref{sec:problem} is to approximate $x$ by the conditional expectation
$\E[X|\Phi X=y]$.

\begin{Theorem}\label{thm:Gaussian}
 The solution $X$  of \eqref{eqn:noisy} is a centered Gaussian vector of $\R^n$ with covariance matrix
\begin{equation}\label{eqhbddjdhh}
K=A^{-1} Q (A^{-1})^T\,.
\end{equation}
Furthermore,  $X$ conditioned on the value $\Phi X=y$ is a Gaussian vector of $\R^n$ with mean $\E[X|\Phi X=y]= \Psi y$,
and of covariance matrix $K^{\Phi}$, where $\Psi$ is the $n \times m$ matrix
\begin{equation}\label{eq:phiidef}
\Psi:=K \Phi^T (\Phi K \Phi^T)^{-1},
\end{equation}
and $K^{\Phi}$ is the rank $n-m$  positive  $n\times n$ symmetric matrix defined by $K^{\Phi}:=  K  -   \Psi \Phi K$.
\end{Theorem}
\begin{proof}
\eqref{eqhbddjdhh} simply follows from $X=A^{-1}\xi$. Since $X$ is a Gaussian vector,
 $\E[X|\Phi X=y]= \Psi y$ where $\Psi$ is a $n\times m$ matrix minimizing the mean squared error
$\E\big[|X- M\Phi X|^2\big]$
over all $n\times m$ matrices $M$. We have
$\E\big[|X- M\Phi X|^2\big]= \Tr[K]+ \Tr[M \Phi K \Phi^T M^T] -2 \Tr[\Phi K M]$
whose minimum is achieved for $M=\Psi$ as defined by \eqref{eq:phiidef}. The covariance matrix of $X$ given $\Phi X=y$ is then obtained by observing that for $v\in \R^n$, $v^T K^{\Phi} v=\E\big[|v^T X-  v^T \Psi \Phi X|^2\big]= v^T K v -  v^T \Psi \Phi K v$.
\end{proof}

\subsection{Variational/optimal recovery properties and approximation error}\label{subsecvarprop}
For a $n\times n$ symmetric positive definite  matrix $M$ let $\<\cdot,\cdot\>_{M}$ be the (scalar) product on $\R^n$ defined by: for $u, v\in \R^n$,
$ \<u, v\>_{M} := u^T M v$ and write $\|v\|_{M}:=\<v, v\>_M^\frac{1}{2}$ the corresponding norm. When $M$ is the identity matrix then we write $\<u, v\>$ and $ \|v\|$ the corresponding scalar product and norm.
For a linear subspace $V$ of $\R^n$ we write $P_{V, M}$ for the orthogonal projections onto $V$ with respect to the scalar product $\<\cdot,\cdot\>_{M}$.
For a (possibly rectangular) matrix $B$ we write $\Img(B)$ the image (range) of $B$ and $\Ker(B)$ the null space of $B$.
For an integer $n$, let $I_n$ be the $n\times n$ identity matrix.

\begin{Theorem}\label{lem:minimizingproperty}
For $w\in \R^m$, $\Psi w$ is the unique minimizer of the following quadratic problem
\begin{equation}\label{eq:dueihdbis}
\begin{cases}
\text{Minimize }  &\<v, v\>_{K^{-1}}\\
\text{Subject to } &\Phi v=w\text{ and } v \in \R^n\,.
\end{cases}
\end{equation}
In particular, $v=\Psi y$, the Bayesian approximation of the solution of \eqref{eqn:scalargeneral}, is the unique minimizer of $\|A v\|_{Q^{-1}}$ under the measurement constraints $\Phi v=y$.  Furthermore, it also holds true that (1) $\Phi \Psi=I_m$
(2) $\Img(\Psi)$ is the orthogonal complement of $\Ker(\Phi)$ with respect to the product $\<\cdot,\cdot\>_{K^{-1}}$ and (3) $\Psi\Phi=P_{\Img(K \Phi^T), K^{-1}}$ and  $I_n-\Psi\Phi=P_{\Ker(\Phi), K^{-1}}$.
\end{Theorem}
\begin{proof}
First observe that \eqref{eq:phiidef} implies that $\Phi \Psi=I_m$ where $I_m$ is the identity $m\times m$ matrix. Therefore $\Phi (\Psi w)=w$.
Note that \eqref{eq:phiidef} implies that for all $z\in \R^m$, $\<\Psi z,v\>_{K^{-1}} =z^T \big(\Phi K \Phi^T\big)^{-1} \Phi v$.
Therefore if $v\in \Ker(\Phi)$ then $\<\Psi z,v\>_{K^{-1}}=0$  for all $z\in \R^m$. Conversely if $\<\Psi z,v\>_{K^{-1}}=0$ for all $z\in \R^m$ then $v$ must belong to $\Ker(\Phi)$.
Since the dimension of  $\Img(\Psi)$ is $m$ and that of $\Ker(\Phi)$ is $n-m$ we conclude that $\Img(\Psi)$ is the orthogonal complement $\Ker(\Phi)$ with respect to the product $\<\cdot,\cdot\>_{K^{-1}}$ and in particular,
\begin{equation}\label{eqn:orthogonalitybis}
 \<\Psi  w,v\>_{K^{-1}} =0, \quad \forall w\in \R^m \text{ and } \forall v\in \R^n \text{ such that }\Phi v=0\,.
\end{equation}
Let $w\in \R^m$ and   $v\in \R^n$ such that $\Phi v=w$. Since $\Psi w -v \in \Ker(\Phi)$, it follows from \eqref{eqn:orthogonalitybis} that
$ \<v,v\>_{K^{-1}} = \<\Psi w,\Psi w\>_{K^{-1}} +\<v-\Psi w,v-\Psi w\>_{K^{-1}}$.
Therefore $\Psi w$ is the unique minimizer of $\<v,v\>_{K^{-1}}$
over all $v\in \R^n$ such that $\Phi v=w$.
Now consider $f\in \R^n$, since $\Img(\Psi)=\Img(K\Phi^T)$ and $\Img(\Psi)$ is the orthogonal complement of $\Ker(\Phi)$ with respect to the product $\<\cdot,\cdot\>_{K^{-1}}$,
 there exists a unique $z\in \R^m$ and a unique $g\in \Ker(\Phi)$ such that $f=K \Phi^T  z +g$.
Since $\Psi \Phi =K \Phi^T (\Phi K \Phi^T)^{-1} \Phi$,
it follows that $\Psi \Phi f= K \Phi^T  z \text{ and }(I_n-\Psi \Phi)f=g$.
We conclude by observing that $g=P_{\Ker(\Phi), K^{-1}} f$.
\end{proof}

\begin{Theorem}\label{thm:gedgdgjdh}
For $v\in \R^n$, $w^*=\Phi v$ is the unique minimizer of $\|v- \Psi w\|_{K^{-1}}$
over all $w\in \R^m$. In particular, $\|v-\Psi \Phi v\|_{K^{-1}}=\min_{z\in \R^m} \|v-K\Phi^T z\|_{K^{-1}}$
and  if  $x$ is the solution of the original equation \eqref{eqn:scalargeneral}, then
$\|x-\Psi y\|_{K^{-1}} =\min_{w\in \R^m} \|x-\Psi w\|_{K^{-1}} =\min_{z\in \R^m} \|x-K\Phi^T z\|_{K^{-1}}$.
\end{Theorem}
\begin{proof}
The proof follows by observing that  $v-\Psi \Phi v$ belongs to the null space of $\Phi$ which,  from  Theorem \ref{lem:minimizingproperty},  is the orthogonal complement of the image of $\Psi$ with respect to the scalar product defining the norm  $\|\cdot\|_{K^{-1}}$. Observe also that the image of $\Psi$ is equal to that of $K\Phi^T$.
\end{proof}

\begin{Remark}\label{rmk:accurkm1}
Observe that,  from Theorem \ref{lem:minimizingproperty}, $v-\Psi \Phi v$ spans the null space of $\Phi$, and
$\|v\|_{K^{-1}}^2= \big\|v-\Psi \Phi v\big\|_{K^{-1}}^2+\big\|\Psi \Phi v\big\|_{K^{-1}}^2$. Therefore if
 $D$ is a symmetric positive definite $n\times n$ matrix then
$\sup_{v \in \R^n} \big\|v-\Psi \Phi v\big\|_{D}/\|v\|_{K^{-1}}=\sup_{v\in \R^n,\, \Phi v=0}\|v\|_{D}/\| v\|_{K^{-1}}$.
In particular, if  $x$ is the solution of  \eqref{eqn:scalargeneral} and $y$ the vector in \eqref{eq:ekiudiu3d}, then
$\big\|x-\Psi y\big\|_{D}/\|b\|_{Q^{-1}}\leq \sup_{v\in \R^n,\, \Phi v=0} \|v\|_{D}/\| v\|_{K^{-1}}$
and the right hand side is the smallest constant for which the inequality holds (for all $b$).
\end{Remark}

\begin{Remark}
A simple calculation (based on the reproducing Kernel property $\<v,K_{\cdot,i}\>_{K^{-1}}=v_i$) shows that if $x$ is the solution of  \eqref{eqn:scalargeneral} and $y$ the vector in \eqref{eq:ekiudiu3d}, then\\
$\Big|(x-\Psi y)_i\Big| \leq \big(K^\Phi_{i,i}\big)^\frac{1}{2}  \|b\|_{Q^{-1}}$, i.e.
 the variance of  the $i$th entry of the solution of the stochastic system \eqref{eqn:noisy}  conditioned on   $\Phi X=y$,  controls the accuracy of the approximation of the $i$th entry of the solution of the deterministic system \eqref{eqn:scalargeneral}. In that sense, the role of
$K^\Phi$  is analogous to that of the power function in radial basis function interpolation \cite{Wendland:2005, Fasshauer:2005}
 and that of the  Kriging function \cite{WuSchback:93} in geostatistics \cite{Myers:1992}.
\end{Remark}

\subsection{Energy norm estimates and selection of the prior}\label{subsecpriorselection}
We will from now on assume that  $A$ is symmetric positive definite. Observe that in this situation
 the energy norm $\|\cdot\|_A$ is of practical significance
 for quantifying the approximation error and  Theorem \ref{thm:gedgdgjdh} leads to the estimate
$\|x-\Psi y\|_{K^{-1}} =\min_{z\in \R^m} \|Q^{-\frac{1}{2}}b- Q^{-\frac{1}{2}} A^{\frac{1}{2}} K^\frac{1}{2}\Phi^T z\|$ which simplifies to the energy norm estimate expressed by Corollary \eqref{cor:geaaedgjdh} under the choice  $Q=A$ (note that $K^{-1}=A$ under that choice).

\begin{Corollary}\label{cor:geaaedgjdh}
If $A$ is symmetric positive definite  and $Q=A$, then for $v\in \R^n$, $\|v-\Psi \Phi v\|_{A}=\min_{z\in \R^m} \|v-A^{-1}\Phi^T z\|_{A}$.
Therefore, if   $x$ is the solution of  \eqref{eqn:scalargeneral} and $y$ the vector in \eqref{eq:ekiudiu3d}, then
$\|x-\Psi y\|_{A} =\min_{w\in \R^m} \|x-\Psi w\|_{A} =\min_{z\in \R^m} \|x-A^{-1}\Phi^T z\|_{A}$. In particular
\begin{equation}\label{eq:gdugeueydd}
\|x-\Psi y\|_{A} =\min_{z\in \R^m} \|A^{-\frac{1}{2}} b-A^{-\frac{1}{2}}\Phi^T z\|\,.
\end{equation}
\end{Corollary}
\begin{Remark}\label{rmk:galerkinproj}
Therefore, according to Corollary \ref{cor:geaaedgjdh}, if $Q=A$, then $\Psi y$ is the Galerkin approximation of $x$, i.e.~the best approximation of $x$ in $\|\cdot\|_A$-norm in the image of $\Psi$ (which is equal to the image of $A^{-1}\Phi^T$). This is interesting because $\Psi y$ is obtained without the prior knowledge of $b$.
\end{Remark}
Corollary \ref{cor:geaaedgjdh} and Remark \ref{rmk:galerkinproj} motivate us to select $Q=A$ as the covariance matrix of the Gaussian prior distribution (mixed strategy of Player II).

\subsection{Impact and selection of the measurement matrix $\Phi$}
It is natural to wonder how good this recovery strategy is (under the choice $Q=A$) compared to the best possible function of  $y$ and how the approximation error is impacted by the measurement matrix $\Phi$. If the energy norm is used to quantify accuracy, then the recovery problem can be expressed as finding the function $\theta$ of the measurements $y$ minimizing the (worst case) approximation error $\inf_{\theta}\sup_{\|b\|\leq 1} \|x - \theta(y)\|_A/\|b\|$
with $x=A^{-1}b$ and $y=\Phi A^{-1} b$. Writing $0<\lambda_1(A)\leq \cdots \leq \lambda_n(A)$, the eigenvalues of $A$ in increasing order, and $a_1,\ldots,a_n$, the corresponding eigenvectors, it is easy to obtain that (1) the best choice for $\Phi$ would correspond to measuring the projection of $x$ on $\operatorname{span}\{a_1,\ldots,a_m\}$ and would lead to the worst approximation error $1/\sqrt{\lambda_{m+1}}$
and (2) the worst choice would correspond to measuring the projection of $x$ on a subspace orthogonal to $a_1$  and would lead to the worst approximation error $1/\sqrt{\lambda_{1}}$.
Under the decision $Q=A$  the minimal value of \eqref{eq:gdugeueydd} is also $1/\sqrt{\lambda_{m+1}}$ and achieved for  $\Img(\Phi^T)=\operatorname{span}\{a_1,\ldots,a_m\}$ and the maximal value of \eqref{eq:gdugeueydd} is $1/\sqrt{\lambda_{1}}$ and achieved when  $\Img(\Phi^T)$ is orthogonal to $a_1$. The following theorem, which is a direct application of \eqref{eq:gdugeueydd} and the estimate derived in \cite[p.~10]{HalkoMartinssonTropp:2011} (see also \cite{MartinssonRokhlinTygert:2011}), shows that, the subset of  measurement matrices that are not \emph{nearly optimal} is of small measure if the rows of $\Phi^T$ are sampled independently on the unit sphere of $\R^n$.

\begin{Theorem}\label{thm:duihdi3dw}
If $\Phi$ is a $n\times m$ matrix with i.i.d. $\mathcal{N}(0,1)$ (Gaussian) entries, $Q=A$,  $x$ is the solution of the original equation \eqref{eqn:scalargeneral}, and $2\leq p$ then
with probability at least $1-3 p^{-p}$, $\|x-\Psi y\|_{A}/\|b\|\leq  (1+9 \sqrt{m+p}\sqrt{n})/\sqrt{\lambda_{m+1}}.$
\end{Theorem}

Although the randomization of the measurement matrix \cite{Frieze98fastmontecarlo, LeParker99, Gilbert2002,OwhadiCandes2004} can be an efficient strategy in  compressed sensing \cite{Tropp2005, CandesTao2006, CandesRombergTao2006, Donoho:2006, Gilbert:2007, Chandrasekaran2011} and in
Singular Value Decomposition/Low Rank approximation \cite{HalkoMartinssonTropp:2011}, we will not use this strategy here because the design of the interpolation operator presents the  (added) difficulty of approximating the eigenvectors associated with the smallest eigenvalues of $A$ rather than those associated with the largest ones. Furthermore, $\Psi$  has to be computed efficiently and the dependence of the approximation constant in Theorem \ref{thm:duihdi3dw} on $n$ and $m$ can be problematic if sharp convergence estimates are to be obtained.
We will instead select the measurement matrix based  on the transfer property introduced in   \cite{BeOw:2010}
and given in a discrete context in the following theorem.

\begin{Theorem}\label{thm:geaaejhdgjvvggdh}
If $A$ is symmetric positive definite, $Q=A$ and $x$ is the solution of the original equation \eqref{eqn:scalargeneral}, then for any symmetric positive definite matrix $B$, we have
\begin{equation}
\inf_{v\in \R^n}\sqrt{\frac{v^T B v}{v^T A v}}\min_{z\in \R^m} \|b-\Phi^T z\|_{B^{-1}} \leq \|x-\Psi y\|_{A} \leq \sup_{v\in \R^n}\sqrt{\frac{v^T B v}{v^T A v}}\min_{z\in \R^m} \|b-\Phi^T z\|_{B^{-1}}
\end{equation}
\end{Theorem}
\begin{proof}
Corollary \ref{cor:geaaedgjdh} implies that  if  $x$ is the solution of the original equation \eqref{eqn:scalargeneral}, then
$\|x-\Psi y\|_{A} =\min_{z\in \R^m} \|b-\Phi^T z\|_{A^{-1}}$. We finish the proof by observing that
if $A$ and $B$ are symmetric positive definite matrices such that $\alpha_1 B \leq A \leq \alpha_2 B$ for some constants $\alpha_1,\alpha_2>0$ then
$\alpha_2^{-1} B^{-1} \leq A^{-1} \leq \alpha_1^{-1} B^{-1}$.
\end{proof}

Therefore according to Theorem \ref{thm:geaaejhdgjvvggdh}, once a good measurement matrix $\Phi$ has been identified for a symmetric positive definite matrix $B$ such that $\alpha_1 B \leq A $, the same measurement matrix can be used for $A$ at the cost of an  increase of the  bound on the error by the multiplicative factor $\alpha_1^{-1/2}$.
As a prototypical example, one may consider a (stiffness) matrix $A$ obtained from a finite element discretization of the PDE \eqref{eqn:scalar} and
 $B$ may be the stiffness matrix of the finite element discretization of the Laplace Dirichlet PDE
\begin{equation}\label{eqn:LapDir}
    -\Delta u'(x)=g(x) \text{ on } \Omega \text{ with } u'=0  \text{ on }\partial \Omega,
\end{equation}
obtained from the same finite-elements (e.g. piecewise-linear nodal basis functions over the same fine mesh $\T_h$). Using the energy norm \eqref{eqenergynorm}, Theorem \ref{thm:geaaejhdgjvvggdh} and Remark \ref{rmk:galerkinproj} imply the following proposition
\begin{Proposition}\label{prop:guguyuyh}
Let $u_h$ (resp. $u_h'$)  be the finite element approximation of the solution $u$ of \eqref{eqn:scalar} (resp. the solution $u'$ of \eqref{eqn:LapDir}) over the finite nodal elements of $\T_h$. Let $u_H$ (resp. $u_H'$) be the finite element approximation of the solution $u$ of \eqref{eqn:scalar} (resp. the solution $u'$ of \eqref{eqn:LapDir}) over linear space spanned by the rows of $A^{-1} \Phi^T$ (resp. over the linear space spanned by the rows of $B^{-1} \Phi^T$). It holds true that
\begin{equation}\label{eq:dgiegduygd}
\frac{1}{\sqrt{\lambda_{\max}(a)}} \|u_h'-u_H'\|_{H^1_0(\Omega)} \leq \|u_h-u_H\|_a\leq \frac{1}{\sqrt{\lambda_{\min}(a)}} \|u_h'-u_H'\|_{H^1_0(\Omega)}
\end{equation}
\end{Proposition}
Observe that the right hand side of \eqref{eq:dgiegduygd} does not depend on $\lambda_{\max}(a)$, therefore if $\lambda_{\min}(a)=1$, then the error bound on $\|u_h-u_H\|_a$ does not depend on the contrast of $a$ (i.e. $\lambda_{\max}(a)/\lambda_{\min}(a)$).

\section{Numerical homogenization and design of the interpolation operator in the continuous case}\label{sec:contcase}

We will now generalize the results and continue the analysis of Section \ref{sec:problem} in the continuous  case and design the interpolation operator for \eqref{eqn:scalar} in  the context of numerical homogenization.

\subsection{Information Game and Gamblets}
As in Section \ref{sec:problem} we will identify the interpolation operator (that will be used for the multigrid algorithm) through a non cooperative
game formulation where Player I chooses the source term $g$ \eqref{eqn:scalar}  and Player II tries to approximate the solution $u$ of \eqref{eqn:scalar} based on a finite number of measurements $(\int_{\Omega} u\phi_i)_{1\leq i \leq m}$ obtained from linearly independent test functions $\phi_i\in L^2(\Omega)$.
As in Section \ref{sec:problem}, this game formulation, motivates the search for a mixed strategy for Player II that can be expressed by replacing the source term $g$ with noise $\xi$. We will therefore consider the following SPDE

\begin{equation}\label{eqn:scalarspde}
\begin{cases}
    -\diiv \Big(a(x)  \nabla v(x)\Big)=\xi(x) \quad  x \in \Omega;\\
    v=0 \quad \text{on}\quad \partial \Omega,
    \end{cases}
\end{equation}
where $\Omega$ and $a$ are the domain and conductivity of  \eqref{eqn:scalar}. As in Section \ref{sec:problem}, to preserve the computational efficiency of the interpolation operator we will assume that $\xi$ is a centered Gaussian field on $\Omega$. The decision space of Player II is therefore the covariance function of $\xi$.
Write $\L$ the differential operator $-\diiv(a\nabla)$ with zero Dirichlet boundary condition mapping $H^1_0(\Omega)$ onto $H^{-1}(\Omega)$. Motivated by the analysis (Remark \ref{rmk:galerkinproj}) of Subsection \ref{subsecpriorselection} (which can be reproduced in the continuous case) we will select the covariance function of $\xi$ (Player II's decision) to be $\L$.  Therefore, under that choice,
for all $f\in H^1_0(\Omega)$, $\int_{\Omega}f(x)\xi(x)\,dx$ is a Gaussian random variable with mean 0 and variance $\int_{\Omega} f\L f=\|f\|_a^2$  where $\|f\|_a$ is the energy norm of $f$  defined in \eqref{eqenergynorm}. Introducing  the scalar product on $H^1_0(\Omega)$ defined by
\begin{equation}\label{eqsp}
\<v,w\>_a:=\int_{\Omega} (\nabla v)^T a \nabla w\,,
\end{equation}
recall that if $(e_1,e_2,\ldots)$ is an orthonormal basis of $(H^1_0(\Omega),\|\cdot\|_a)$ diagonalizing $\L$, then
$\xi$ can formally be represented as $\xi=\sum_{i=1}^\infty (\L e_i) X_i$ (where the $X_i$ are i.i.d. $\mathcal{N}(0,1)$ random variables) and, therefore, $\xi$ can
 also be identified as the linear isometry mapping $H^1_0(\Omega)$ onto a Gaussian space and $f=\sum_{i=1}^\infty \<f,e_i\>_a e_i$ onto $\int_{\Omega}f(x)\xi(x)\,dx=\sum_{i=1}^\infty \<f, e_i\>_a X_i$.

Observe also that  \cite{Owhadi:2014}, if $\xi'$ is White Noise on $\Omega$ (i.e. a Gaussian field with covariance function $\delta(x-y)$) then $\xi$ can be represented as $\xi=\L^{-\frac{1}{2}}\xi'$. Furthermore \cite[Prop.~3.1]{Owhadi:2014} the solution of \eqref{eqn:scalarspde} is Gaussian field with covariance function $G(x,y)$ (where $G$ is  the Green's function of the PDE \eqref{eqn:scalar}, i.e. $\L G(x,y)=\delta(x-y)$ with $G(x,y)=0$ for $y\in \partial \Omega$).

Let $\F$ be the $\sigma$-algebra generated by the random variables $\int_{\Omega}v(x)\phi_i$ for $i\in \{1,\ldots,m\}$ (with $v$ solution of \eqref{eqn:scalarspde}). We will identify the interpolation basis elements by conditioning the solution of \eqref{eqn:scalarspde} on $\F$. Observe that the covariance matrix of the measurement vector $(\int_{\Omega}v(x)\phi_i)_{1\leq i \leq m}$ is  the $m\times m$  symmetric matrix $\Theta$ defined by
\begin{equation}\label{eqtheta0}
\Theta_{i,j}:=\int_{\Omega^2}\phi_i(x)G(x,y) \phi_j(y)\,\dx\,dy
\end{equation}
Note that for $l\in \R^m$, $l^T \Theta l=\|w\|_a^2$ where $w$ is the solution of \eqref{eqn:scalar} with right hand side $g=\sum_{i=1}^m l_i \phi_i$.
Therefore (since the test functions $\phi_i$ are linearly independent) $\Theta$ is positive definite and we will write $\Theta^{-1}$ its inverse.
Write $\delta_{i,j}$ the Kronecker's delta ($\delta_{i,i}=1$ and $\delta_{i,j}=0$ for $i\not=j$).

\begin{Theorem}\label{thmgufufg0}
Let $v$ be the solution of \eqref{eqn:scalarspde}. It holds true that
\begin{equation}\label{eqmeanvspde0}
\E\big[v(x)\big| \F \big]=\sum_{i=1}^m \psi_i(x) \int_{\Omega} v(y)\phi_i (y)\,dy
\end{equation}
where the functions $\psi_i \in H^1_0(\Omega)$ are defined by
\begin{equation}\label{eqgamblet}
\psi_i(x):=\E\Big[v(x)\Big| \int_{\Omega} v(y)\phi_j(y)\,dy=\delta_{i,j},\,j\in \{1,\ldots,m\}\Big]
\end{equation}
and admit the following representation formula
\begin{equation}\label{eq:doehdd}
\psi_i(x)=\sum_{j=1}^m \Theta_{i,j}^{-1} \int_{\Omega}G(x,y)\phi_j(y)\,dy\,.
\end{equation}
Furthermore, the distribution of $v$ conditioned on $\F$ is that of a Gaussian field with mean \eqref{eqmeanvspde0} and covariance function
$\Gamma(x,y)=G(x,y)+ \sum_{i,j=1}^m \psi_i(x) \psi_j(y) \Theta_{i,j}$\\ $- \sum_{i=1}^m \psi_i(x) \int_{\Omega} G(y,z)\phi_i (z)\,dz- \sum_{i=1}^m \psi_i(y) \int_{\Omega} G(x,z)\phi_i (z)\,dz\,.$
\end{Theorem}
\begin{proof}
The proof is similar to that of \cite[Thm.~3.5]{Owhadi:2014}. The identification of the covariance function follows from the expansion of
$\Gamma(x,y)=\E\Big[\big(v(x)-\E\big[v(x)\big| \F \big]\big)\big(v(y)-\E\big[v(y)\big| \F \big]\big) \Big]$. Note that \eqref{eq:doehdd} proves that $\psi_i \in H^1_0(\Omega)$.
\end{proof}

Since, according to \eqref{eqgamblet} and the discussion preceding \eqref{eqn:scalarspde}, each $\psi_i$ is an elementary gamble (bet) on value of the solution
of \eqref{eqn:scalar} given the information $\int_{\Omega}\phi_j u=\delta_{i,j}$ for  $j=1,\ldots,m$ we will refer to
the basis functions $(\psi_i)_{1\leq i \leq m}$, as \emph{gamblets}. According to \eqref{eqmeanvspde0}, once gamblets have been identified, they form a basis for betting on the value of the solution of \eqref{eqn:scalar} given the measurements $(\int_{\Omega}\phi_j u)_{1\leq i \leq m}$.

\subsection{Optimal recovery properties}
Although gamblets admit the representation formula \eqref{eq:doehdd}, we will not use it for their practical (numerical) computation. Instead we will work with variational properties inherited from the conditioning of the Gaussian field $v$. To guide our intuition, note that
since  $\L$ is the precision function (inverse of the covariance function) of  $v$,
the conditional expectation of $v$ can be identified by minimizing  $\int_{\Omega}\psi\L \psi $ given measurements constraints.
 This observation motivates us to consider, for $i\in \{1,\ldots, m\}$, the following quadratic optimization problem
\begin{equation}\label{eq:dueihdbewdaisq}
\begin{cases}
\text{Minimize }  &\|\psi\|_a\\
\text{Subject to } &\psi \in H^1_0(\Omega)\text{ and }\int_{\Omega}\phi_j \psi=\delta_{i,j}\text{ for } j=1,\ldots,m
\end{cases}
\end{equation}
where $\|\psi\|_a$ is the energy norm of $\psi$  defined in \eqref{eqenergynorm}.

The following theorem shows that \eqref{eq:dueihdbewdaisq} can be used to identify $\psi_i$ and that gamblets are characterized by optimal (variational) recovery properties.
\begin{Theorem}\label{thm:dkdehgjdhdgh0}
It holds true that (1) The optimization problem \eqref{eq:dueihdbewdaisq} admits a unique minimizer $\psi_i$ defined by \eqref{eqgamblet} and \eqref{eq:doehdd}
(2) For $w\in \R^m$, $\sum_{i=1}^m w_i \psi_i$ is the unique minimizer of $\|\psi\|_a$ subject to $\int_{\Omega} \psi(x) \phi_j(x)=w_j$ for $j\in \{1,\ldots,m\}$ and  (3) (using the scalar product defined in \eqref{eqsp})  $\<\psi_i,\psi_j\>_a=\Theta_{i,j}^{-1}$.
\end{Theorem}
\begin{proof}
 Let $w\in \R^m$ and $\psi_w=\sum_{i=1}^m w_i \psi_i$ with $\psi_i$ defined as in \eqref{eq:doehdd}. The definition of $\Theta$ implies that $\int_{\Omega} \psi_w(x) \phi_j(x)=w_j$ for $j\in \{1,\ldots,m\}$. Furthermore we obtain by integration by parts that
 for all $\varphi\in H^1_0(\Omega)$, $\<\psi_w,\varphi\>_a=\sum_{i,j=1}^m  w_i \Theta_{i,j}^{-1} \int_{\Omega} \phi_j \varphi$.
Therefore, if $\psi \in H^1_0(\Omega)$  is such that $\int_{\Omega} \psi(x) \phi_j(x)=w_j$ for $j\in \{1,\ldots,m\}$ then
$\<\psi_w,\psi-\psi_w\>_a=0$ and
\begin{equation}\label{eq:hdkhdkjh33e}
\|\psi\|_a^2=\|\psi_w\|_a^2+\|\psi-\psi_w\|_a^2
\end{equation}
which finishes the proof of optimality of $\psi_i$ and $\psi_w$.
\end{proof}

\subsection{Optimal accuracy of the recovery}\label{subsecoptaccuracy}
 Define
\begin{equation}\label{eqmawybdysedsd}
u^*(x):= \sum_{i=1}^m \psi_i(x) \int_{\Omega} u(y) \phi_i(y)\,dy
\end{equation}
where $u$ is the solution of \eqref{eqn:scalar} and $\psi_i$ are the gamblets defined by \eqref{eqgamblet} and \eqref{eq:doehdd}. Note $u^*$ corresponds to  Player II's bet on the value of $u$ given the measurements $(\int_{\Omega} u(y) \phi_i(y)\,dy)_{1\leq i \leq m}$. In particular,  if
$v$ is the solution of \eqref{eqn:scalarspde} then
\begin{equation}\label{eqmawybdyd}
u^*(x)=\E\big[v(x)\big| \int_{\Omega} v(y)\phi_i (y)\,dy=\int_{\Omega} u(y)\phi_i (y)\,dy \big]
\end{equation}

For  $\phi \in H^{-1}(\Omega)$ write $\L^{-1}\phi$ the solution of \eqref{eqn:scalar} with $g=\phi$.
The following Theorem shows that $u^*$ is the best approximation (in energy norm) of $u$ in $\operatorname{span} \{\L^{-1} \phi_i :  i\in \{1,\ldots,m\}\}$.
\begin{Theorem}\label{thm:dkdehgjdhdgh02}
Let $u$ be the solution of \eqref{eqn:scalar}, $u^*$ defined in \eqref{eqmawybdysedsd} and \eqref{eqmawybdyd}. It holds true that
\begin{equation}\label{sidasaeweddaud02}
    \|u - u^*\|_{a}=\inf_{\psi \in \operatorname{span} \{\L^{-1} \phi_i :  i\in \{1,\ldots,m\}\}}\|u-\psi\|_a
\end{equation}
\end{Theorem}
\begin{proof}
By Theorem \ref{thmgufufg0} $\operatorname{span} \{\L^{-1} \phi_i :  i\in \{1,\ldots,m\}\}=\operatorname{span}\{\psi_1,\ldots,\psi_m\}$ and \eqref{sidasaeweddaud02} follows from the fact that
 $\int_{\Omega}(u-u^*)\phi_j=0$ for all $j$ implies that $u-u^*$ is orthogonal to $\operatorname{span}\{\psi_1,\ldots,\psi_m\}$ with respect to the scalar product $\<\cdot,\cdot\>_a$.
\end{proof}

\subsection{Transfer property and selection of the measurement functions}
We will now select the measurement (test) functions $\phi_i$ by
extending the  result of Proposition \ref{prop:guguyuyh}  to
  the continuous case. For $V$, a finite dimensional linear subspace of $H^{-1}(\Omega)$, define
\begin{equation}\label{ksjjseddesel3}
(\diiv a \nabla)^{-1} V:=\operatorname{span} \{(\diiv a\nabla)^{-1} \phi \,:\,  \phi\in V\}.
\end{equation}
where $(\diiv a\nabla)^{-1} \phi$ is the solution of \eqref{eqn:scalar} with $g=-\phi$.
 Similarly define $\Delta^{-1}V:=\operatorname{span} \{\Delta^{-1} \phi \,:\,  \phi\in V\}$ where $\Delta^{-1}\phi$ is the solution of \eqref{eqn:LapDir} with $g=-\phi$.
\begin{Proposition}\label{prop:deihdidu}
If $u$ and $u'$ are the solutions of \eqref{eqn:scalar} and \eqref{eqn:LapDir} (with the same r.h.s. $g$) and $V$ is a finite dimensional linear subspace of $H^{-1}(\Omega)$, then
\begin{equation}\label{sidasasedsssddaud}
\frac{1}{\sqrt{\lambda_{\max}(a)}}\inf_{v\in \Delta^{-1}V}   \|u' - v\|_{H^1_0(\Omega)} \leq  \inf_{v\in (\diiv a\nabla)^{-1} V}   \|u - v\|_{a}\leq \frac{1}{\sqrt{\lambda_{\min}(a)}}\inf_{v\in \Delta^{-1}V}   \|u' - v\|_{H^1_0(\Omega)}
\end{equation}
\end{Proposition}
\begin{proof}
Write $G$ the Green's function of \eqref{eqn:scalar} and $G^*$ the Green's function of  \eqref{eqn:LapDir}. Observe that for $f\in V$ and $v=(\diiv a\nabla)^{-1}f$,
$\|u - v\|_{a}^2=\int_{\Omega^2}(g(x)-f(x))G(x,y) (g(y)-f(y))\,dx\,dy$.  The monotonicity of Green's function as a quadratic form  (see for instance \cite[Lemma~4.13]{BenarousOwhadi:2003}), implies $ \int_{\Omega^2}(g(x)-f(x))G(x,y) (g(y)-f(y))\,dx\,dy \leq \frac{1}{\lambda_{\min}(a)} \int_{\Omega^2}(g(x)-f(x))G^*(x,y) (g(y)-f(y))\,dx\,dy$ (with a similar inequality on the l.h.s.) which concludes the proof.
\end{proof}
 This extension, which is also directly related to the transfer property of the flux-norm (introduced in  \cite{BeOw:2010} and generalized in \cite{Sym12}, see also \cite{Wang:2012}), allows us to select accurate finite dimensional bases for the approximation of the solution space of \eqref{eqn:scalar}.

\begin{Construction}\label{cons:odoehdduhd}
Let $(\tau_i)_{1\leq i \leq m}$ be a partition of $\Omega$ such that each $\tau_i$ is Lipschitz, convex and of diameter at most $H$. Let $(\phi_i)_{1\leq i \leq m}$ be elements of $L^2(\Omega)$ such that for each $i$, the support of
$\phi_i$ is contained in the closure of $\tau_i$ and $\int_{\tau_i} \phi_i\not=0$.
\end{Construction}

\begin{Proposition}\label{prop:gegddgdjdef}
Let $(\phi_i)_{1\leq i \leq m}$ be the elements of Construction \ref{cons:odoehdduhd} and let $u$ be the solution of \eqref{eqn:scalar}. If $V=\operatorname{span} \{\phi_i\,:\, 1\leq i\leq m\}$ then
\begin{equation}\label{sidasaeweddaud2}
 \inf_{v\in (\diiv a\nabla)^{-1} V}   \|u - v\|_{a}\leq C H \|g\|_{L^2(\Omega)}
\end{equation}
with $C=\big(\pi \sqrt{\lambda_{\min}(a)}\big)^{-1} \Big(1+\max_{1 \leq i \leq m} \big(\frac{ \frac{1}{|\tau_i|}\int_{\tau_i} \phi_i^2}{(\frac{1}{|\tau_i|}\int_{\tau_i} \phi_i)^2})^\frac{1}{2}\Big)$ (writing $|\tau_i|$ the volume of $\tau_i$).
\end{Proposition}
\begin{proof}
Using Proposition \ref{prop:deihdidu} it is sufficient to complete proof when $a$ is the constant identity matrix. Let $u'$ be the solution of \eqref{eqn:LapDir} and $v\in \Delta^{-1} V$. Note that $\Delta v=\sum_{i=1}^m c_i \phi_i$, therefore $\|u' - v\|_{H^1_0(\Omega)}^2= -\int_{\Omega} (u'-v) (g-\sum_{i=1}^m c_i \phi_i)$.
Taking  $c_i=\int_{\tau_i} g/ \int_{\tau_i} \phi_i $
 we obtain that $\int_{\tau_i} (g- \sum_{j=1}^m c_j \phi_j)=0$ and, writing $|\tau_i|$ the volume of $\tau_i$,
$ \|u' - v\|_{H^1_0(\Omega)}^2=-\sum_{i=1}^m \int_{\tau_i} (u'-v-\frac{1}{|\tau_i|}\int_{\tau_i}(u'-v)) (g- \sum_{j=1}^m c_j \phi_j)$
which by Poincar\'{e}'s inequality (see \cite{PayneWeinberger:1960} for the optimal constant $1/\pi$ used here) lead to
$ \|u' - v\|_{H^1_0(\Omega)}^2 \leq \frac{H}{\pi} \sum_{i=1}^m \big(\int_{\tau_i} |\nabla (u'-v)|^2\big)^\frac{1}{2}
 \big(\int_{\tau_i} (g-\sum_{j=1}^m c_j \phi_j)^2\big)^\frac{1}{2}$.
Therefore, by using Cauchy-Schwartz inequality and simplifying,
$\|u' - v\|_{H^1_0(\Omega)} \leq \frac{H}{\pi} \|g-\sum_{i=1}^m c_i \phi_i\|_{L^2(\Omega)}\,.$
Now, since each $\phi_i$ has support in $\tau_i$ we have
$ \|\sum_{i=1}^m c_i \phi_i\|_{L^2(\Omega)}^2=\sum_{i=1}^m (\int_{\tau_i} g)^2 \frac{ \int_{\tau_i} \phi_i^2}{(\int_{\tau_i} \phi_i)^2}\leq \|g\|_{L^2(\Omega)}^2 \max_{1 \leq i \leq m} \frac{ \frac{1}{|\tau_i|}\int_{\tau_i} \phi_i^2}{(\frac{1}{\tau_i}\int_{\tau_i} \phi_i)^2}$,
which concludes the proof.
\end{proof}

The value of the constant $C$ in Proposition \ref{prop:gegddgdjdef} motivates us to modify Construction \ref{cons:odoehdduhd} as follows.
\begin{Construction}\label{cons:odoehdduhdI}
Let  $(\phi_i)_{1\leq i \leq m}$ be the elements constructed in \ref{cons:odoehdduhd} under the additional assumptions that (a) each $\phi_i$ is equal to one  on $\tau_i$ and zero elsewhere (b) there exists $\delta \in (0,1)$ such that for each $i\in \{1,\ldots,m\}$, $\tau_i$ contains a ball of diameter $\delta H$.
\end{Construction}

Let $(\phi_i)_{1\leq i \leq m}$ be as in Construction \ref{cons:odoehdduhdI}. Note that the additional assumption (a) implies that the constant $C$ in Proposition \ref{prop:gegddgdjdef} is equal to $2/(\pi \sqrt{\lambda_{\min}(a)})$. Assumption (b)  will be used for localization purposes in subsections \ref{subsec:expdecay} and \ref{subsec:localization} (and is not required for Theorem \ref{thm:dkdehgjdhdgh}). The following theorem is  a direct consequence of  Proposition \ref{prop:gegddgdjdef} and Theorem \ref{thm:dkdehgjdhdgh02}.

\begin{Theorem}\label{thm:dkdehgjdhdgh}
 If $u$ is the solution of \eqref{eqn:scalar} and $(\psi_i)_{i=1}^m$ are the gamblets identified in \eqref{eqgamblet}, \eqref{eq:dueihdbewdaisq} and \eqref{eq:doehdd}  then
\begin{equation}\label{sidasaeweddaud}
 \inf_{v\in \operatorname{span}\{\psi_1,\ldots,\psi_m\}}   \|u - v\|_{a}\leq \frac{2}{\pi \sqrt{\lambda_{\min}(a)}} H \|g\|_{L^2(\Omega)}
\end{equation}
and the minimum in the l.h.s of \eqref{sidasaeweddaud} is achieved for $v=u^*$ defined in \eqref{eqmawybdysedsd} and \eqref{eqmawybdyd}.
\end{Theorem}
\begin{Remark}\label{rmkiegudyg3}
The assumption of convexity of the subdomains $\tau_i$ is only used  to derive sharper constants via Poincar\'{e}'s inequality for convex domains  (without it, approximation error bounds remain valid after multiplication by $\pi$). Similarly,  the transfer property can be used to derive constructions that are distinct from \ref{cons:odoehdduhd} and \ref{cons:odoehdduhdI}.
\end{Remark}

\begin{Remark}
Gamblets defined via the constrained energy minimization problems \eqref{eq:dueihdbewdaisq} are analogous to the \emph{energy minimizing bases} of  \cite{Mandel1999, wan2000, Xu2004, XuZhu2008} and in particular \cite{Vassilevski2010}. However
they form a different set of basis functions when global constraints are added: the (total) \emph{energy minimizing bases} of  \cite{Mandel1999, wan2000, Xu2004, XuZhu2008, Vassilevski2010} are defined by minimizing the total energy $\sum_i \|\psi_i\|_a^2$ subject to the constraint $\sum_i \psi_i(x)=1$ related to the local preservation of constants.
Numerical experiments \cite{ Xu2004} suggest that  total energy minimizing basis functions could lead to a $\mathcal{O}(\sqrt{H})$ convergence rate (with rough coefficients). Note that \eqref{eq:dueihdbewdaisq} is also analogous to the constrained minimization problems associated  with Polyharmonic Splines  \cite{Harder:1972, Duchon:1976,Duchon:1977,Duchon:1978, OwhadiZhangBerlyand:2014}, which can be recovered with a Gaussian prior (on $\xi$) with covariance function $\delta(x-y)$ (corresponding to exciting \eqref{eqn:scalarspde} with white noise). We suspect that the basis functions obtained in the orthogonal decomposition of \cite{MaPe:2012} can also  be recovered via the variational formulation \eqref{eq:dueihdbewdaisq} by identifying the null space of the Clement quasi-interpolation operator with that of appropriately chosen measurement functions $\phi_i$.
\end{Remark}

\subsection{Exponential decay of gamblets}\label{subsec:expdecay}
Theorems \ref{thm:dkdehgjdhdgh0} and \ref{thm:dkdehgjdhdgh02} show that the gamblets $\psi_i$ have optimal recovery properties analogous to the discrete case of Theorem \ref{lem:minimizingproperty} and Corollary \ref{cor:geaaedgjdh}. However one may wonder why one should compute these gamblets rather than the elements $(\diiv a\nabla)^{-1} \phi_i$ since they span the same linear space (by the representation formula \eqref{eq:doehdd}).
The answer lies in the fact that each gamblet $\psi_i$ decays exponentially as a function of the distance from the support of $\phi_i$ and its computation can therefore be localized to a subdomain of diameter $  \mathcal{O}(H \ln \frac{1}{H})$  without impacting the order of accuracy \eqref{sidasaeweddaud}. Consider the construction \ref{cons:odoehdduhdI}. Let $\psi_i$ be defined as in Theorem \ref{thm:dkdehgjdhdgh0} and let $x_i$ be an element of $\tau_i$. Write $B(x,r)$ the ball of center $x$ and radius $r$.
\begin{Theorem}\label{thm:expdecay}{\bf Exponential decay of the basis elements $\psi_i$.}
It holds true that
\begin{equation}
\int_{\Omega \cap (B(x_i,r))^c}(\nabla \psi_i)^T a \nabla \psi_i\leq e^{1-\frac{ r}{l H}} \int_{\Omega}(\nabla \psi_i)^T a \nabla \psi_i
\end{equation}
with $l=1+ (e/\pi) \sqrt{\lambda_{\max}(a)/\lambda_{\min}(a)} (1+2^\frac{3}{2}(2/\delta)^{1+d/2})$ (where $e$ is Euler's number).
\end{Theorem}
\begin{proof}
 Let $k,l \in \mathbb{N}^*$ and $i\in \{1,\ldots,m\}$.  Let $S_0$ be the union of all the domains $\tau_j$ that are contained in the closure of
$  B(x_i, k l H)\cap \Omega$, let $S_{1}$ be the union of all the domains $\tau_j$ that are contained in the closure of
$(B(x_i, (k+1) l H))^c\cap \Omega$ and let $S^*=S_0^c \cap S_{1}^c \cap \Omega$ (be the union of the remaining elements $\tau_j$ not contained in $S_0$ or $S_{1}$). Let $\eta$ be the function on $\Omega$ defined by $\eta(x)=\operatorname{dist}(x,S_0)/(\operatorname{dist}(x,S_0)+\operatorname{dist}(x,S_1))$.
Observe that (1) $0\leq \eta \leq 1$ (2) $\eta$ is equal to zero on $S_0$ (3) $\eta$ is equal to one on $S_{1}$ (4) $\|\nabla \eta \|_{L^\infty(\Omega)}\leq \frac{1}{ l H}$. Observe that
$- \int_{\Omega}\eta \psi_i \diiv( a \nabla \psi_i)=\int_{\Omega}\nabla (\eta \psi_i) a \psi_i=\int_{\Omega} \eta (\nabla \psi_i)^T a \nabla \psi_i+\int_{\Omega} \psi_i (\nabla \eta)^T a \nabla \psi_i$. Therefore $\int_{S_1} (\nabla \psi_i)^T a \nabla \psi_i\leq I_1+I_2$ with
\begin{equation}\label{eqleldkdlkdjd}
I_1=\|\nabla \eta\|_{L^\infty(\Omega)}
\big(\sum_{\tau_j \subset S^*}\int_{\tau_j}\psi_i^2\big)^\frac{1}{2}
\big(\int_{S^*}(\nabla \psi_i)^T a \nabla \psi_i\big)^\frac{1}{2} \sqrt{\lambda_{\max}(a)}
\end{equation}
 and
$I_2=- \int_{\Omega}\eta \psi_i \diiv( a \nabla \psi_i)$. By \eqref{eq:doehdd}, $- \diiv( a \nabla \psi_i)$ is piecewise constant and equal to $\Theta_{i,j}^{-1}$ on $\tau_j$. By the constraints of \eqref{eq:dueihdbewdaisq} $\int_{\tau_j}\psi_i=0$ for $i\not=j$. Therefore (writing $\eta_j$ the volume average of $\eta$ over $\tau_j$) we have
\begin{equation}\label{eqkjkhkjhkejhccd}
I_2 \leq -\sum_{\tau_j \subset S_1 \cup S^*} \int_{\tau_j}(\eta-\eta_j) \psi_i \diiv( a \nabla \psi_i) \leq \frac{1}{l} \sum_{\tau_j \subset  S^*} \|\psi_i\|_{L^2(\tau_j)} \|\diiv( a \nabla \psi_i)\|_{L^2(\tau_j)}.
\end{equation}
We will now need the following lemma

\begin{Lemma}\label{lemdhkedjhdkjh}
If $v\in \operatorname{span}\{\psi_1,\ldots,\psi_m\}$ then
\begin{equation}
 \|\diiv( a \nabla v)\|_{L^2(\Omega)} \leq H^{-1}  \|v\|_a  (\lambda_{\max}(a) 2^{5+d}/\delta^{2+d})^\frac{1}{2}
\end{equation}
\end{Lemma}
\begin{proof}
Let $c\in \R^m$ and $v=\sum_{i=1}^m c_i \psi_i$.
Observing that $-\diiv(a\nabla v)=\sum_{i=1}^m c_i \Theta_{i,j}^{-1}$ in $\tau_j$ and using the decomposition
$ \|\diiv( a \nabla v)\|_{L^2(\Omega)}^2 =\sum_{i=1}^m  \|\diiv( a \nabla v)\|_{L^2(\tau_j)}^2$, we obtain that
\begin{equation}\label{eqhuhh7rfrtrd}
 \|\diiv( a \nabla v)\|_{L^2(\Omega)}^2 = \sum_{j=1}^m (\sum_{i=1}^m c_i \Theta_{i,j}^{-1})^2 |\tau_j|
\end{equation}
Furthermore, $v$ can be
decomposed over $\tau_j$ as $v=v_1+v_2$, where $v_1$ solves $-\diiv(a\nabla v_1)=\sum_{i=1}^m c_i \Theta_{i,j}^{-1}$ in $\tau_j$ with $v_1=0$ on $\partial\tau_j$, and $v_2$ solves $-\diiv(a\nabla v_2)=0$ in $\tau_j$ with $v_2=v$ on $\partial \tau_j$. Using the notation $|\xi|_a^2=\xi^T a \xi$,
observe that $\int_{\tau_j} |\nabla v|_a^2=\int_{\tau_j} |\nabla v_1|_a^2+ \int_{\tau_j} |\nabla v_2|_a^2$. Furthermore,
$\int_{\tau_j} |\nabla v_1|_a^2=\sum_{i=1}^m c_i \Theta_{i,j}^{-1} \int_{\tau_j} v_1$. Writing $G_j$ the Green's function of the operator $-\diiv(a\nabla \cdot)$ with Dirichlet boundary condition on $\partial \tau_j$, note that $\int_{\tau_j} v_1= (\sum_{i=1}^m c_i \Theta_{i,j}^{-1}) \int_{\tau_j^2}  G_j(x,y)\,dx\,dy$.
 Using the monotonicity of the Green's function as a quadratic form (as in the proof of Proposition \ref{prop:deihdidu}), we have $\int_{\tau_j^2}  G_j(x,y)\,dx\,dy \geq \frac{1}{\lambda_{\max}(a)}\int_{\tau_j^2}  G_j^*(x,y)\,dx\,dy $ where $G^*_j$ is the Green's function of the operator $-\Delta$ with Dirichlet boundary condition on $\partial \tau_j$.
 Recall that $2\int_{\tau_j}  G_j^*(x,y)\,dy$ is the mean exit time (from $\tau_j$) of a Brownian motion started from $x$ and the mean exit time of a Brownian motion started from $x$ to exit a ball of center $x$ and radius $r$ is $r^2$ (see for instance \cite{BenarousOwhadi:2003}). Since $\tau_j$ contains a ball of diameter $\delta H$, it follows that $2 \int_{\tau_j^2}  G_j^*(x,y)\,dx\,dy \geq (\delta H/4)^{2+d} V_d$   (where $V_d$ is the volume of the $d$-dimensional unit ball). Therefore (after  using $|\tau_j|\leq V_d (H/2)^d$ and simplification),
 \begin{equation}\label{eqhhdhdgdjhdg3e}
 \int_{\tau_j} |\nabla v_1|_a^2\geq (\sum_{i=1}^m c_i \Theta_{i,j}^{-1})^2   |\tau_j|  H^2 \delta^{2+d}/(2^{5+d}\lambda_{\max}(a)),
 \end{equation}
  which finishes the proof after taking the sum over $j$.
\end{proof}
Now observe that since $\int_{\tau_j}\psi_i=0$ for $i\not=j$, we obtain, using Poincar\'{e}'s inequality (with the optimal constant of \cite{PayneWeinberger:1960}), that  $\|\psi_i\|_{L^2(\tau_j)} \leq \|\nabla \psi_i\|_{L^2(\tau_j)} H/\pi$.
Therefore, combining \eqref{eqleldkdlkdjd}, \eqref{eqkjkhkjhkejhccd} and the result of Lemma \ref{lemdhkedjhdkjh}, we obtain after simplification
\begin{equation}
\int_{S_1}(\nabla \psi_i)^T a \nabla \psi_i\leq \frac{1}{\pi l} \sqrt{\lambda_{\max}(a)/\lambda_{\min}(a)} (1+2^\frac{3}{2}(2/\delta)^{1+d/2})\int_{S^*}(\nabla \psi_i)^T a \nabla \psi_i
\end{equation}
Taking $l\geq  \frac{e}{\pi } \sqrt{\lambda_{\max}(a)/\lambda_{\min}(a)} (1+2^\frac{3}{2}(2/\delta)^{1+d/2}) $ and enlarging the integration domain on the right hand side we obtain
$
 \int_{S_{1}}(\nabla \psi_i)^T a \nabla \psi_i\leq e^{-1} \int_{S^*\cup S_{1}}(\nabla \psi_i)^T a \nabla \psi_i.
$
We conclude the proof via straightforward iteration on $k$.
\end{proof}

\subsection{Localization of the basis elements}\label{subsec:localization}
Theorem \ref{thm:expdecay} allows us to localize the construction of basis elements $\psi_i$ as follows. For $r>0$ let $S_r$ be the union of the subdomains $\tau_j$ intersecting $B(x_i,r)$ (recall that $x_i$ is an element of $\tau_i$) and let $\psi_i^{\loc,r}$ be the minimizer of the following quadratic problem
\begin{equation}\label{eq:dwehhsiuhssq}
\begin{cases}
\text{Minimize }  &\int_{S_r} (\nabla \psi)^T a \nabla \psi\\
\text{Subject to } &\psi \in H^1_0(S_r)\text{ and }\int_{S_r}\phi_j \psi=\delta_{i,j}\text{ for }j \text{ such that } \tau_j \subset S_r.
\end{cases}
\end{equation}
We will naturally identify $\psi_i^{\loc,r}$ with its extension to $H^1_0(\Omega)$ by setting $\psi_i^{\loc,r}=0$ outside of $S_r$.
From now on,  to simplify the expression of constants, we will assume without loss of generality that the domain is rescaled so that $\diam(\Omega)\leq 1$.
\begin{Theorem}\label{thm:hieuhdds}
It holds true that
\begin{equation}
\|\psi_i-\psi_i^{\loc,r}\|_{a} \leq  C e^{-\frac{r}{2l H}},
\end{equation}
where $l$ is defined in Theorem \ref{thm:expdecay},
$C=  (\lambda_{\max}(a)/\sqrt{ \lambda_{\min}(a)}) H^{-\frac{d}{2}-2} 2^{2d+9}/(\sqrt{V_d} \delta^{d+2})$ and $V_d$ is the volume of the $d$-dimensional unit ball.
\end{Theorem}
\begin{proof}
We will need the following lemma.
\begin{Lemma}\label{lem:dihidue23}
It holds true that
\begin{equation}\label{eqgguguyg6}
\|\psi_i\|_a \leq  (H \delta)^{-\frac{d}{2}-1} \sqrt{\lambda_{\max}(a)} 2^{\frac{3}{2}d+2} (V_d)^{-\frac{1}{2}}
\end{equation}
where $V_d$ is the volume of the $d$-dimensional unit ball, and,
\begin{equation}\label{eqbbuybuybfcc}
|\<\psi_i,\psi_j\>_a|\leq e^{-\frac{ r_{i,j}}{2 l H}} H^{-2-d}  \lambda_{\max}(a)  2^{\frac{5d+11}{2}}/(V_d \delta^{d+2})
\end{equation}
where $l$ is the constant of Theorem \ref{thm:expdecay} and $r_{i,j}$ is the distance between $\tau_i$ and $\tau_j$.
\end{Lemma}
\begin{proof}
Since $\tau_i$ contains a ball $B(x_i,\delta H/2)$ of center $x_i\in \tau_i$ and diameter $\delta H/2$, there exists a piece-wise differentiable function $\eta$,
 equal to $1$ on $B(x_i,\delta H/4)$, equal to $0$ on $(B(x_i,\delta H/2))^c$ and such that $0\leq \eta \leq 1$ with $\|\nabla \eta\|_{L^\infty(\Omega)}\leq \frac{4}{H \delta}$. Since  $\psi=\eta/ (\int_{\tau_i}\eta)$ satisfies the constrains of the minimization problem \eqref{eq:dueihdbewdaisq} we have $\|\psi_i\|_a \leq \|\psi\|_a$, which proves \eqref{eqgguguyg6}.  Theorem \ref{thm:dkdehgjdhdgh0} implies that $\<\psi_i,\psi_j\>_a=\Theta_{i,j}^{-1}$.
 Observing that $-\diiv( a \nabla \psi_i)$ is piecewise constant and equal to $\Theta_{i,j}^{-1}$ on $\tau_j$ and applying \eqref{eqhhdhdgdjhdg3e} (with $v=\psi_i$ and using $\int_{\tau_j} |\nabla v_1|_a^2\leq \int_{\tau_j} |\nabla v|_a^2$),
 we obtain that
 \begin{equation} \label{equsiheiuhdihd}
 |\Theta_{i,j}^{-1}| \leq \big(\lambda_{\max}(a) 2^{5+d}/(\delta^{2+d} |\tau_j|)\big)^\frac{1}{2} H^{-1}  \big(\int_{\tau_j}(\nabla \psi_i)^T a \nabla \psi_i\big)^\frac{1}{2}.
\end{equation}
 which leads to   \eqref{eqbbuybuybfcc} by the exponential decay obtained in Theorem \ref{thm:expdecay} and \eqref{eqgguguyg6}.
\end{proof}
Let us now prove Theorem \ref{thm:hieuhdds}.
Let $S_0$ be the union of the subdomains $\tau_j$ not contained in $S_r$ and let $S_1$ be the union of the subdomains $\tau_j$ that are at distance at least
$H$ from $S_0$ (for $S_0=\emptyset$ the proof is trivial, so we may assume that $S_0\not=\emptyset$, similarly it is no restriction to assume that $S_1\not=\emptyset$).
Let $\eta$ be the function on $\Omega$ defined by $\eta(x)=\operatorname{dist}(x,S_0)/(\operatorname{dist}(x,S_0)+\operatorname{dist}(x,S_1))$. Observe that is a piecewise differentiable function on $\Omega$ such that (1) $\eta$ is equal to one on $S_1$ and zero on $S_0$ (2) $\|\nabla \eta\|_{L^\infty(\Omega)}\leq \frac{1}{H}$ and (3) $0\leq \eta \leq 1$.
Since $\psi_i^{\loc,r}$ satisfies the constraints of \eqref{eq:dueihdbewdaisq}, we have from \eqref{eq:hdkhdkjh33e},
\begin{equation}\label{eqkjhdkjdhdkjh}
\|\psi_i-\psi_i^{\loc,r}\|_a^2 =\|\psi_i^{\loc,r}\|_a^2-\|\psi_i\|_a^2.
\end{equation}
Let $\psi_k^{i,r}$ be the minimizer of $\int_{S_r} (\nabla \psi)^T a \nabla \psi$ subject to  $\psi \in H^1_0(S_r)$ and $\int_{S_r}\phi_j \psi=\delta_{k,j}$ for $\tau_j \subset S_r$.
Write $w_j=\int_{\Omega} \eta \psi_i \phi_j$. Let $\psi_w^{i,r}:=\sum_{j=1}^m w_j \psi_j^{i,r}$.
 Noting that $\psi_w^{i,r}=\psi_i^{\loc,r}+\sum_{\tau_j \subset S^*}w_j \psi_j^{i,r}$, where $S^*$ is the union of $\tau_j\subset S_r$ not contained in $S_1$, and using  property  (3) of Theorem \ref{thm:dkdehgjdhdgh0} (with $\Theta^{i,-1}_{k,k'}=\int_{S_r} (\nabla \psi_k^{i,r})^T a \nabla \psi_{k'}^{i,r}$) it follows that
\begin{equation}\label{eqkswkjh}
\|\psi_w^{i,r}\|_a^2 =\|\psi_i^{\loc,r}\|_a^2+ \|\sum_{\tau_j \subset S^*}w_j \psi_j^{i,r}\|^2_a+2 \sum_{\tau_j \subset S^*} \Theta^{i,-1}_{i,j} w_j\,.
\end{equation}
  Noting that $\eta \psi_i \in H^1_0(S_r)$,
 Theorem \ref{thm:dkdehgjdhdgh0} implies that $\|\psi_w^{i,r}\|_a \leq \|\eta \psi_i\|_a$, which, combined with  \eqref{eqkswkjh} and \eqref{eqkjhdkjdhdkjh} leads to $\|\psi_i-\psi_i^{\loc,r}\|_a^2 \leq \|\eta \psi_i\|_a^2-\|\psi_i\|_a^2-2 \sum_{\tau_j \subset S^*} \Theta^{i,-1}_{i,j} w_j$ and (using $\|\eta \psi_i\|_a^2-\|\psi_i\|_a^2\leq \int_{S*}\nabla (\eta \psi_i)^T a \nabla (\eta \psi_i)$)
 \begin{equation}\label{eqkjhdkgygejdhdkjh}
\|\psi_i-\psi_i^{\loc,r}\|_a^2 \leq \int_{S*}\nabla (\eta \psi_i)^T a \nabla (\eta \psi_i)+2 |\sum_{\tau_j \subset S^*} \Theta^{i,-1}_{i,j} w_j|\,.
\end{equation}
Now observe that $ \frac{1}{2}\int_{S*}\nabla (\eta \psi_i)^T a \nabla (\eta \psi_i)\leq  \int_{\Omega \cap (B(x_i,r-2 H))^c}(\nabla \psi_i)^T a \nabla \psi_i
 +\frac{\lambda_{\max}(a)}{H^2}   \int_{S^*}|\psi_i|^2$.
Applying Poincar\'{e}'s inequality  we obtain
$\int_{S^*}|\psi_i|^2 \leq \frac{1}{\pi^2} H^2\sum_{\tau_j \subset S^*} \int_{\tau_j}|\nabla \psi_i|^2$ (since $\int_{\tau_j} \psi_i=0$ for $\tau_j \subset S^*$), and $\int_{S^*}|\psi_i|^2 \leq
\frac{H^2}{ \pi^2 \lambda_{\min}(a)} \int_{\Omega \cap (B(x_i,r-2 H))^c}(\nabla \psi_i)^T a \nabla \psi_i$.
Combining  these equations with the exponential decay of Theorem \ref{thm:expdecay} we deduce
\begin{equation}\label{eqkjhhggjdhdkjh}
\int_{S*}\nabla (\eta \psi_i)^T a \nabla (\eta \psi_i)  \leq
 2 \Big(1+\lambda_{\max}(a)/\big(\pi^2 \lambda_{\min}(a)\big) \Big)   e^{1-\frac{ r-2H}{l H}} \|\psi_i\|_a^2\,.
\end{equation}
Similarly, using Cauchy-Schwartz and Poincar\'e inequalities we have for $\tau_j \subset S^*$,\\
$|w_j|\leq |\tau_j|^\frac{1}{2} \|\psi_i\|_{L^2(\tau_j)} \leq |\tau_j|^\frac{1}{2} (\int_{\tau_j} (\nabla \psi_i)^Ta(\nabla \psi_i))^\frac{1}{2} /\sqrt{\lambda_{\min}(a)}$
and \\
$
|\sum_{\tau_j \subset S^*} \Theta^{i,-1}_{i,j} w_j|\leq  |\sum_{\tau_j \subset S^*} (\Theta^{i,-1}_{i,j})^2|\tau_j| |^\frac{1}{2}
\big (\int_{S^*} (\nabla \psi_i)^Ta(\nabla \psi_i)/\lambda_{\min}(a)\big)^\frac{1}{2}.
$
 Using \eqref{equsiheiuhdihd}  we obtain that
$
 |\sum_{\tau_j \subset S^*} (\Theta^{i,-1}_{i,j})^2 |\tau_j| |^\frac{1}{2} \leq \big(\lambda_{\max}(a) 2^{5+d}/\delta^{2+d} \big)^\frac{1}{2} H^{-1} \big(\int_{S^*}(\nabla \psi_i^{i,r})^T a \nabla \psi_i^{i,r}\big)^\frac{1}{2},
$
which by the exponential decay of Theorem \eqref{thm:expdecay} (and $\|\psi_i\|_a\leq \|\psi_i^{i,r}\|_a$) leads to
\begin{equation}\label{eqhgygjhhhh}
|\sum_{\tau_j \subset S^*} \Theta^{i,-1}_{i,j} w_j|\leq   \big(\frac{\lambda_{\max}(a) 2^{5+d}}{\lambda_{\min}(a)\delta^{2+d}} \big)^\frac{1}{2} H^{-1}
\|\psi_i^{i,r}\|_a^2 e^{1-\frac{ r-2H}{l H}}\,.
 \end{equation}
Using \eqref{eqgguguyg6} to bound $\|\psi_i^{i,r}\|_a$ and combining \eqref{eqhgygjhhhh} with \eqref{eqkjhhggjdhdkjh} and \eqref{eqkjhdkgygejdhdkjh} concludes the proof.
\end{proof}

The following theorem shows that gamblets preserve the $\mathcal{O}(H)$ rate of convergence (in energy norm) after localization to sub-domains of size $\mathcal{O}(H \ln(1/H))$. They can therefore be used as localized basis functions in numerical homogenization
\cite{ BaLip10, OwZh:2011,  MaPe:2012, OwhadiZhangBerlyand:2014}. Section \ref{sechnh} will show that they can also be computed hierarchically at near linear complexity.

\begin{Theorem}\label{thm:dhfkdehgjdhdgh}
 Let  $u$ be the solution of \eqref{eqn:scalar} and $(\psi_1^{\loc,r})_{1\leq i \leq m}$ the localized gamblets identified in \eqref{eq:dwehhsiuhssq},
  then for $r \geq  H (C_1 \ln \frac{1}{H} + C_2) $ we have
\begin{equation}\label{sidasaeuyweddaud}
 \inf_{v\in \operatorname{span}\{\psi_1^{\loc,r},\ldots,\psi_m^{\loc,r}\}}   \|u - v\|_{a}\leq \frac{1}{ \sqrt{\lambda_{\min}(a)}} H \|g\|_{L^2(\Omega)}\,.
\end{equation}
The constants are $C_1=(d+4)l$ and $C_2=2l \ln \Big(\frac{\lambda_{\max}(a)}{ \lambda_{\min}(a)}\frac{ 2^{\frac{3}{2}d+11} }{\delta^{d+2} }\Big)$ where $l$ is the constant of Theorem \ref{thm:hieuhdds}.
Furthermore, the inequality \eqref{sidasaeuyweddaud} is achieved for $v=\sum_{i=1}^m \psi_i^{\loc,r}\int_{\Omega} u \phi_i$.
\end{Theorem}

\begin{proof}
Let $v_1:=\sum_{i=1}^m c_i \psi_i$ and  $v_2=\sum_{i=1}^m c_i \psi_i^{\loc,r}$ with $c_i=\int_{\Omega} u \phi_i$. Theorem \ref{thm:dkdehgjdhdgh} implies that
$\|u-v_1\|_a\leq 2/(\pi \sqrt{\lambda_{\min}(a)}) H \|g\|_{L^2(\Omega)}$.
Observe that $\|u - v_2\|_{a} \leq \|u - v_1\|_{a}+\|v_1 - v_2\|_{a}$ and $\|v_1 - v_2\|_{a}\leq \max_{i}\|\psi_i-\psi_i^{\loc,r}\|_{a} \sum_{i=1}^m |c_i|$.
Using Poincar\'e's inequality $\|u\|_{L^2(\Omega)}\leq \diam(\Omega) \|\nabla u\|_{L^2(\Omega)}$ (with $\diam(\Omega)\leq 1$) we obtain
 $\sum_{i=1}^m  |c_i| \leq  \int_{\Omega} |u| \leq  \|g\|_{L^2(\Omega)} 2^{-d/2}V_d^\frac{1}{2} /\lambda_{\min}(a) $. We conclude using Theorem \ref{thm:hieuhdds} to bound $\max_{i}\|\psi_i-\psi_i^{\loc,r}\|_{a}$.
 \end{proof}

\section{Multiresolution operator decomposition}\label{sechnh}

Building on the analysis of Section \ref{sec:contcase}, we will now gamble on the approximation of the solution of \eqref{eqn:scalar} based on measurements
performed at different levels of resolution. The resulting hierarchical (and nested) games will then be used to derive a
 multiresolution  decomposition of \eqref{eqn:scalar} (orthogonal across subscales) and  a near-linear complexity multiresolution algorithm with a priori error bounds.

\subsection{Hierarchy of nested measurement functions}\label{subsecdomdecomphi}

In order to define the hierarchy of games we will first define a hierarchy of nested measurement functions.
\begin{Definition}\label{defindextree}
We say that $\I$ is an index tree of depth $q$ if it is a finite set of $q$-tuples of the form $i=(i_1,\ldots,i_q)$ with $1\leq i_1 \leq m_0$ and
$1\leq i_j \leq m_{(i_1,\ldots,i_{j-1})}$ for $j\geq 2$, where $m_0$ and $m_{(i_1,\ldots,i_{j-1})}$ are strictly positive integers. For $1\leq k \leq q$ and $i=(i_1,\ldots,i_q)\in \I$, we write
$i^{(k)}:=(i_1,\ldots,i_k)$  and $\I^{(k)}:=\{i^{(k)}\,:\, i\in \I\}$.  For $k\leq k'\leq q$ and $j=(j_1,\ldots,j_{k'})\in \I^{(k')}$ we write
$j^{(k)}:=(j_1,\ldots,j_k)$. For $i\in \I^{(k)}$ and $k\leq k'\leq q$ we write
 $i^{(k,k')}$  the set of elements  $j\in \I^{(k')}$ such that $j^{(k)}=i$.
\end{Definition}

\begin{Construction}\label{defmulires}
Let $\I$ be an index tree of depth $q$. Let $\delta \in (0,1)$ and $0< H_q <\cdots <H_1<1$. Let $(\tau_i^{(k)}, k\in \{1,\ldots,q\}, i\in \I^{(k)})$ be a collection of subsets of $\Omega$ such that (1) for $1\leq k \leq q$, $(\tau_i^{(k)}, i\in \I^{(k)})$ is a partition of $\Omega$ such that each $\tau_i^{(k)}$ is a  Lipschitz, convex subset of $\Omega$ of diameter at most $H_k$ and contains a ball of diameter $\delta H_k$ (2) the sequence of partitions is nested, i.e.  for $k\in \{1,\ldots,q-1\}$ and $i\in \I^{(k)}$, $\tau_i^{(k)}:=\cup_{j\in i^{(k,k+1)}} \,\tau_j^{(k+1)}$.
\end{Construction}
As in Remark \ref{rmkiegudyg3}, the assumption of convexity of the subdomains $\tau_i^{(k)}$ is not necessary to the results presented here and is only used to derive sharper/simpler constants.
Let $\phi_i^{(k)}$ be the indicator function of the set $\tau_i^{(k)}$ (i.e. $\phi_i^{(k)}=1$ if $x\in \tau_i^{(k)}$ and $\phi_i^{(k)}=0$ if $x\not\in \tau_i^{(k)}$). Note that the nesting of the domain decomposition implies that of the measurement functions, i.e. for $k\in \{1,\ldots,q-1\}$ and $i\in \I^{(k)}$,
\begin{equation}\label{eq:eigdeiud3dd}
\phi^{(k)}_i=\sum_{j\in \I^{(k+1)}}\pi^{(k,k+1)}_{i,j}  \phi^{(k+1)}_j
\end{equation}
where  $\pi^{(k,k+1)}$ is the $\I^{(k)}\times \I^{(k+1)}$ matrix defined by $\pi^{(k,k+1)}_{i,j}=1$ if $j\in i^{(k,k+1)}$ and $\pi^{(k,k+1)}_{i,j}=0$ if $j\not\in i^{(k,k+1)}$. We will assume without loss of generality  that $\|\phi_i^{(k)}\|_{L^2(\Omega)}^2=|\tau_i^{(k)}|$ is constant in $i$ (for the general case,  rescale/renormalize each $\phi_i^{(k)}$   and the entries of $\pi^{(k,k+1)}$ by the corresponding multiplicative factors, we will keep track of the dependence of some of the constants on $\max_{i,j}|\tau_i^{(k)}|/|\tau_j^{(k)}|$).

\subsection{Hierarchy of nested gamblets and multiresolution approximations}
Let us now consider the problem of recovering the solution of \eqref{eqn:scalar} based on the nested measurements $(\int_{\Omega} u \phi_i^{(k)})_{i\in  \I^{(k)}}$ for $k\in \{1,\ldots,q\}$. As in Section \ref{sec:contcase} we are lead to investigate the mixed strategy (for Player II) expressed by replacing the source term $g$ with a centered Gaussian field with covariance function $\L=-\diiv(a\nabla)$. Under that mixed strategy, Player II's bet on the value of the solution of \eqref{eqn:scalar}, given the measurements $(\int_{\Omega} u(y)\phi^{(k)}_i(y)\,dy)_{i\in \I^{(k)}}$, is (see Subsection \ref{subsecoptaccuracy})
\begin{equation}\label{eqdefuk}
u^{(k)}(x):=\sum_{i\in \I^{(k)}} \psi^{(k)}_i(x) \int_{\Omega} u(y)\phi^{(k)}_i (y)\,dy,
\end{equation}
where (see Theorem \ref{thm:dkdehgjdhdgh0}), for  $k\in \{1,\ldots,q\}$ and $i\in  \I^{(k)}$, $\psi^{(k)}_i$ is the minimizer of
\begin{equation}\label{eq:dfddeytfewdaisq}
\begin{cases}
\text{Minimize }  &\|\psi\|_a\\
\text{Subject to } &\psi \in H^1_0(\Omega)\text{ and }\int_{\Omega}\phi_j^{(k)} \psi=\delta_{i,j}\text{ for } j\in \I^{(k)}\,.
\end{cases}
\end{equation}
Define $\V^{(q+1)}:=H^1_0(\Omega)$ and, for  $k\in \{1,\ldots,q\}$,
\begin{equation}\label{eqdefvk}
\V^{(k)}:=\operatorname{span}\{\psi^{(k)}_i \mid i\in \I^{(k)}\} .
\end{equation}

By Theorem \ref{thmgufufg0} $\operatorname{span}\{\psi_i^{(k)} \mid  i\in  \I^{(k)}\}=\operatorname{span} \{\L^{-1} \phi_i^{(k)} \mid  i\in \I^{(k)}\}$, and the nesting \eqref{eq:eigdeiud3dd} of the measurement functions implies the nesting of the spaces $\V^{(k)}$. The following theorem is (which is a direct application of theorems \ref{thm:dkdehgjdhdgh02} and \ref{thm:dkdehgjdhdgh}) shows that $u^{(k)}$  is the best (energy norm) approximation of the solution of \eqref{eqn:scalar} in $\V^{(k)}$.

\begin{Theorem}\label{thmgugyug0}
 It holds true that (1) for $k\in \{1,\ldots,q\}$, $\V^{(k)} \subset \V^{(k+1)}$ and
$\V^{(k)}=\operatorname{span} \{\L^{-1} \phi_i^{(k)} \mid  i\in \I^{(k)}\}$
  and (2) If $u$ is the solution of \eqref{eqn:scalar} and $u^{(k)}$ defined in \eqref{eqdefuk} then
\begin{equation}\label{eqdhekjddkjhdjk}
\|u - u^{(k)}\|_{a}=\inf_{v\in \V^{(k)}}   \|u - v\|_{a}\leq \frac{2}{\pi \sqrt{\lambda_{\min}(a)}} H_k \|g\|_{L^2(\Omega)}
\end{equation}
\end{Theorem}

\subsection{Nested games  and martingale/multiresolution  decomposition}\label{subsecmaringaleopdecomposition}
 As in Section \ref{sec:contcase} we  consider the mixed strategy (for Player II) expressed by replacing the source term $g$ with a centered Gaussian field with covariance function $\L$. Under this mixed strategy, Player II's bet \eqref{eqdefuk} on the value of the solution of \eqref{eqn:scalar}, given the measurements $(\int_{\Omega} u(y)\phi^{(k)}_i(y)\,dy)_{i\in \I^{(k)}}$, can also be obtained by conditioning the solution  $v$
of the SPDE  \eqref{eqn:scalarspde} (see \eqref{eqmawybdyd}), i.e.
\begin{equation}\label{eqkdkjdkjdh}
u^{(k)}(x)=\E\Big[v(x)\Big| \int_{\Omega} v(y)\phi^{(k)}_i(y)\,dy=\int_{\Omega} u(y)\phi^{(k)}_i(y)\,dy ,\,i\in \I^{(k)}\Big]
\end{equation}
Furthermore, each gamblet $\psi_i^{(k)}$ represents Player II's bet on the value of the solution of \eqref{eqn:scalar} given the measurements $\int_{\Omega} u(y)\phi^{(k)}_j(y)\,dy=\delta_{i,j}$, i.e.
\begin{equation}\label{eqrepphii}
\psi^{(k)}_i=\E\Big[v\Big| \int_{\Omega} v(y)\phi^{(k)}_j(y)\,dy=\delta_{i,j},\,j\in  \I^{(k)}\Big]
\end{equation}

Now consider the nesting of non-cooperative games where Player I chooses $g$ in \eqref{eqn:scalar} and Player II is shown the measurements $(\int_{\Omega} u \phi_i^{(k)})_{i\in \I^{(k)}}$, step by step, in a hierarchical manner, from coarse ($k=1$) to fine ($k=q$) and must, at each step $k$ of the game, gamble on the value of solution $u$. The following theorem and \eqref{eqkdkjdkjdh} show that the resulting sequence of approximations  $u^{(k)}$ form the realization of a martingale with independent increments.

\begin{Theorem}\label{thmdgdjdgygugyd}
Let $\F_k$  be the $\sigma$-algebra generated by the random variables $(\int_{\Omega}v(x)\phi_i^{(k)})_{i\in  \I^{(k)}}$ and
\begin{equation}\label{eqmdedeanvspde}
v^{(k)}(x):=\E\big[v(x)\big| \F_k \big]=\sum_{i\in  \I^{(k)}} \psi^{(k)}_i(x) \int_{\Omega} v(y)\phi^{(k)}_i (y)\,dy
\end{equation}
It holds true that (1) $\F_1,\ldots,\F_q$ forms a filtration, i.e. $\F_k\subset \F_{k+1}$
(2) For $x\in \Omega$, $v^{(k)}(x)$ is a martingale with respect to the filtration $(\F_k)_{k\geq 1}$, i.e.
$v^{(k)}(x)=\E\big[v^{(k+1)}(x)\big| \F_{k}\big]$ (3) $v^{(1)}$ and the increments $(v^{(k+1)}-v^{(k)})_{k\geq 1}$ are independent Gaussian fields.
\end{Theorem}
\begin{proof}
The nesting \eqref{eq:eigdeiud3dd} of the measurement functions implies  $\F_k \subset \F_{k+1}$ and $(\F_k)_{k\geq 1}$ is therefore filtration. The fact that $v^{(k)}$ is a martingale follows from  $v^{(k)}=\E\big[v\big| \F_k \big]$. Since $v^{(1)}$ and the increments $(v^{(k+1)}-v^{(k)})_{k\geq 1}$ are  Gaussian fields belonging to the same Gaussian space their independence is equivalent to zero covariance, which follows from the martingale property, i.e. for $k\geq 1$
$\E\big[v^{(1)}(v^{(k+1)}-v^{(k)})\big]=\E\Big[\E\big[v^{(1)}(v^{(k+1)}-v^{(k)})\big|\F_k\big]\Big]=\E\Big[v^{(1)} \E\big[(v^{(k+1)}-v^{(k)})\big|\F_k\big]\Big]=0$
and for $k>j\geq 1$,
$\E\big[(v^{(j+1)}-v^{(j)})(v^{(k+1)}-v^{(k)})\big]=\E\Big[(v^{(j+1)}-v^{(j)}) \E\big[(v^{(k+1)}-v^{(k)})\big|\F_k\big]\Big]=0$.
\end{proof}
\begin{Remark}
 Theorem \ref{thmdgdjdgygugyd}   enables the application of classical  results concerning martingales to the numerical analysis of $v^{(k)}$ (and $u^{(k)}$). In particular (1) Martingale (concentration) inequalities can  be used to control the fluctuations of $v^{(k)}$ (2) Optimal stopping times can be used to derive optimal strategies for stopping numerical simulations  based on  loss functions mixing computation costs with the cost of imperfect decisions (3) Taking $q=\infty$ in the construction of the basis elements $\psi^{(k)}_i$ (with a sequence $H_k$ decreasing towards 0) and using the martingale convergence theorem imply that,  for all $\varphi\in C_0^\infty(\Omega)$, $\int_{\Omega} v^{(k)}\varphi \rightarrow \int_{\Omega} v\varphi$
 as $k\rightarrow \infty$ (a.s. and in $L^1$).
\end{Remark}

The independence of the increments $v^{(k+1)}-v^{(k)}$ is related to the following orthogonal multiresolution decomposition of the operator \eqref{eqn:scalar}. For $\V^{(k)}$ defined as in \eqref{eqdefvk} and for $k\in \{2,\ldots,q+1\}$ let $\W^{(k)}$ be the orthogonal complement of $\V^{(k-1)}$ within $\V^{(k)}$ with respect to the scalar product $\<\cdot,\cdot\>_a$. Write $\oplus_a$ the orthogonal direct sum with respect to the scalar product $\<\cdot,\cdot\>_a$.
Note that by Theorem \ref{thmgugyug0},  $u^{(k)}$ defined by \eqref{eqdefuk} is the finite element solution of \eqref{eqn:scalar} in $\V^{(k)}$ (in particular we will write $u^{(q+1)}=u$).

\begin{Theorem}\label{thmgugyug2}
 It holds true that (1) For $k\in \{2,\ldots,q+1\}$,
\begin{equation}\label{eqdedhhiuhe3}
\V^{(k)}=\V^{(1)}\oplus_a \W^{(2)} \oplus_a  \cdots \oplus_a \W^{(k)},
\end{equation}
(2) for $k\in \{1,\ldots,q\}$, $u^{(k+1)}-u^{(k)}$ belongs to $\W^{(k+1)}$ and
\begin{equation}\label{equdecom}
u=u^{(1)}+(u^{(2)}-u^{(1)})+\cdots+(u^{(q)}-u^{(q-1)})+(u-u^{(q)})
\end{equation}
is the orthogonal decomposition of $u$ in
$H^1_0(\Omega)=\V^{(1)}\oplus_a \W^{(2)}\oplus_a \cdots \oplus_a \W^{(q)} \oplus_a \W^{(q+1)}$,
and (3) $u^{(k+1)}-u^{(k)}$ is the finite element solution of \eqref{eqn:scalar} in $\W^{(k+1)}$.
\end{Theorem}
\begin{proof}
Observe that since  the  $\V^{(k)}$ are nested (Theorem \ref{thmgugyug0}) $u^{(k+1)}-u^{(k)}$ belongs to $\V^{(k+1)}$.
 Furthermore (by Property (1) of Theorem \ref{thmgugyug0} and integration by parts), for $i\in  \I^{(k)}$,
$\<u^{(k+1)}-u^{(k)}, \psi_i^{(k)}\>_a$ belongs to $\operatorname{span}\{\int_{\Omega}(u^{(k+1)}-u^{(k)}) \phi_i^{(k)}\mid i\in \I^{(k)} \}$.
Finally, \eqref{eqdefuk}, the constraints of \eqref{eq:dfddeytfewdaisq} and the nesting property \eqref{eq:eigdeiud3dd} imply that for $i\in  \I^{(k)}$,
 $\int_{\Omega}(u^{(k+1)}-u^{(k)}) \phi_i^{(k)}=\sum_{j\in \I^{(k+1)}}\pi^{(k,k+1)}_{i,j}\int_{\Omega}u \phi_{j}^{(k+1)} -\int_{\Omega}u \phi_i^{(k)}=0$
which implies that $u^{(k+1)}-u^{(k)}$ belongs to $\W^{(k+1)}$.
\end{proof}

\subsection{Interpolation and restriction matrices/operators}

Since the  spaces $\V^{(k)}$ are nested there exists a $\I^{(k)}\times  \I^{(k+1)}$   matrix $R^{(k,k+1)}$ such that
for $1\leq k \leq q-1$ and $i\in \I^{(k)}$
\begin{equation}\label{eq:ftfytftfx}
\psi^{(k)}_i=\sum_{j \in  \I^{(k+1)}} R_{i,j}^{(k,k+1)} \psi_j^{(k+1)}
\end{equation}
We will refer to $R^{(k,k+1)}$ as the restriction matrix and to its transpose
$R^{(k+1,k)}:=(R^{(k,k+1)})^T$ as the interpolation/prolongation matrix.
The following theorem shows that (see Figure \ref{fig:bets}) $R^{(k,k+1)}_{i,j}$ is Player II's best bet on the value of $\int_{\Omega} u\phi^{(k+1)}_j$ given the information that $\int_{\Omega} u\phi^{(k)}_s=\delta_{i,s},\,s\in  \I^{(k)}$).
\begin{Theorem}\label{eqhjgjhgjgjg}
It holds true
that for $i\in \I^{(k)}$ and $j\in \I^{(k+1)}$,\\
$R^{(k,k+1)}_{i,j}= \int_{\Omega} \psi^{(k)}_i \phi_j^{(k+1)}= \E\big[\int_{\Omega} v(y)\phi^{(k+1)}_j(y)\,dy\big| \int_{\Omega} v(y)\phi^{(k)}_l(y)\,dy=\delta_{i,l},\,l\in \I^{(k)}\big]$.
\end{Theorem}
\begin{proof}
The  first equality  is obtained by integrating \eqref{eq:ftfytftfx} against $\phi_j^{(k+1)}$ and using the constraints satisfied by $\psi^{(k+1)}_j$ in \eqref{eq:dfddeytfewdaisq}.
For the second equality, observe that
since $\F_k$ is a filtration we can replace $v$ in the representation formula \eqref{eqrepphii} by $v^{(k)}$ (as defined by the r.h.s. of \eqref{eqmdedeanvspde}) and obtain\\
$
 \psi^{(k)}_i(x) =\sum_{j\in \I^{(k+1)}} \psi^{(k+1)}_j(x) \E\big[\int_{\Omega} v(y)\phi^{(k+1)}_j(y)\,dy\big| \int_{\Omega} v(y)\phi^{(k)}_l(y)\,dy=\delta_{i,l},\,l\in \I^{(k)}\big]
$
which corresponds to \eqref{eq:ftfytftfx}.
\end{proof}

\subsection{Nested computation of the interpolation and stiffness matrices}

Let $v$ be the solution of \eqref{eqn:scalarspde}. Observe that $(\int_{\Omega}v(x)\phi_i^{(k)})_{i\in \I^{(k)}}$ is a Gaussian vector with (symmetric, positive definite) covariance matrix
$\Theta^{(k)}$  defined by for $i,j \in \I^{(k)}$,
\begin{equation}\label{eq:kldldje34}
\Theta^{(k)}_{i,j}:=\int_{\Omega^2 }\phi^{(k)}_i(x) G(x,y) \phi^{(k)}_j(y)\,dx\,dy\,.
\end{equation}
As in \eqref{eqtheta0}, $\Theta^{(k)}$ is invertible and we write  $\Theta^{(k),-1}$ its inverse.
Observe that, as in Theorem \ref{thm:dkdehgjdhdgh0}, $\psi_i^{(k)}$  admits the following representation formula
\begin{equation}\label{eq:doehddcasek}
\psi_i^{(k)}(x)=\sum_{j\in \I^{(k)}} \Theta_{i,j}^{(k),-1} \int_{\Omega}G(x,y)\phi_j^{(k)}(y)\,dy
\end{equation}

Observe that, as in Theorem \ref{thm:dkdehgjdhdgh0}, $\Theta^{(k),-1}=A^{(k)}$ where $A^{(k)}$ is the (symmetric, positive definite) stiffness matrix of the elements $\psi^{(k)}_i$, i.e., for $i,j \in \I^{(k)}$,
\begin{equation}\label{eq:iwihud3de}
A^{(k)}_{i,j}:=\< \psi^{(k)}_i, \psi^{(k)}_j\>_a
\end{equation}

Write $\pi^{(k+1,k)}$ the transpose of the matrix $\pi^{(k,k+1)}$ (defined below \eqref{eq:eigdeiud3dd}) and $I^{(k)}$  the $\I^{(k)}\times \I^{(k)}$ identity matrix. The following theorem enables the hierarchical/nested computation of $A^{(k)}$ from $A^{(k+1)}$.

\begin{Theorem}\label{thmhggfees5}
For $b\in \R^{\I^{(k)}}$, $R^{(k+1,k)}b$ is the (unique) minimizer  $c\in \R^{\I^{(k+1)}}$ of
\begin{equation}\label{eq:dfdytffdeytfewdaisq}
\begin{cases}
\text{Minimize }  &c^T A^{(k+1)} c \\
\text{Subject to } &\pi^{(k,k+1)}c=b
\end{cases}
\end{equation}
Furthermore $R^{(k,k+1)} \pi^{(k+1,k)}=\pi^{(k,k+1)}R^{(k+1,k)} =I^{(k)}$, $R^{(k,k+1)}=A^{(k)}\pi^{(k,k+1)}\Theta^{(k+1)}$, $\Theta^{(k)}=\pi^{(k,k+1)}\Theta^{(k+1)}\pi^{(k+1,k)}$ and
\begin{equation}\label{eqhuhiuv}
A^{(k)}= R^{(k,k+1)}A^{(k+1)}R^{(k+1,k)}\,.
\end{equation}
\end{Theorem}
\begin{proof}
Using the decompositions \eqref{eq:ftfytftfx} and \eqref{eq:eigdeiud3dd} in $\int_{\Omega}\phi_j^{(k)} \psi_i^{(k)}=\delta_{i,j}$ leads to\\
$R^{(k,k+1)} \pi^{(k+1,k)}=I^{(k)}$. Using \eqref{eq:doehddcasek} and \eqref{eq:eigdeiud3dd} to expand $\psi^{(k)}_i$ Theorem \ref{eqhjgjhgjgjg} leads to $R^{(k,k+1)}=A^{(k)}\pi^{(k,k+1)}\Theta^{(k+1)}$. Using \eqref{eq:eigdeiud3dd} to expand $\phi_i^{(k)}$ and $\phi_j^{(k)}$ in \eqref{eq:kldldje34} leads to $\Theta^{(k)}=\pi^{(k,k+1)}\Theta^{(k+1)}\pi^{(k+1,k)}$. Using \eqref{eq:ftfytftfx} to expand $\psi^{(k)}_i$ and $\psi^{(k)}_j$ in \eqref{eq:iwihud3de} leads to \eqref{eqhuhiuv}. Let $b\in \R^{\I^{(k)}}$. Theorem \ref{thm:dkdehgjdhdgh0} implies that $\sum_{i\in \I^{(k)}}b_i \psi_i^{(k)}$ is the unique minimizer
of $\|v\|_a^2$ subject to $v\in H^1_0(\Omega)$ and $\int_{\Omega}\phi_j^{(k)} v=b_j$ for $j\in \I^{(k)}$. Since $\V^{(k)}\subset \V^{(k+1)}$ and since the minimizer is in $\V^{(k)}$, the minimization over $v\in H^1_0(\Omega)$ can be reduced to $v\in \V^{(k+1)}$ of the form $v=\sum_{i\in \I^{(k+1)}}c_i \psi_i^{(k+1)}$, which after using \eqref{eq:eigdeiud3dd} to expand the constraint $\int_{\Omega}\phi_j^{(k)} v=b_j$, corresponds to \eqref{eq:dfdytffdeytfewdaisq}.
\end{proof}

\subsection{Multiresolution gamblets}\label{subsecmultiresbasis}
The interpolation and restriction operators are sufficient to derive a multigrid method for solving \eqref{eqn:scalar}. To design a multiresolution algorithm we need to continue the analysis and identify basis functions for the subspaces $\W^{(k)}$.
For $k=2,\ldots,q$ let $\J^{(k)}$ be the finite set of $k$-tuples of the form $i=(i_1,\ldots,i_k)$ with $1\leq i_1 \leq m_0$, $1\leq i_j \leq m_{(i_1,\ldots,i_{j-1})}$ for $ 2\leq j \leq k-1$ and $1\leq i_k \leq m_{(i_1,\ldots,i_{k-1})}-1$. Where the integers $m_\cdot$ are the same as those defining the index tree $\I$. For a matrix $M$ write $\Img(M)$ and $\Ker(M)$ its image and kernel.

\begin{Lemma}\label{lemwk}
For $k=2,\ldots,q$ let $W^{(k)}$ be a $\J^{(k)}\times \I^{(k)}$ matrix such that  $\Img(W^{(k),T})=\Ker(\pi^{(k-1,k)})$. It holds true that the elements  $(\chi_i^{(k)})_{i\in \J^{(k)}}\in \V^{(k)}$ defined as
  \begin{equation}\label{eqjkhdkdh}
\chi^{(k)}_i:=\sum_{j \in \I^{(k)}} W_{i,j}^{(k)} \psi_j^{(k)}
 \end{equation}
form a basis of $\W^{(k)}$.
\end{Lemma}
\begin{proof}
Since $\L \V^{(k-1)}=\operatorname{span}\{\phi^{(k-1)}_i\mid i\in \I^{(k-1)}\}$,
$w\in \V^{(k)}$ belongs to $\W^{(k)}$ if and only if $\int_{\Omega} \phi^{(k-1)}_j w =0$ for all $j\in \I^{(k-1)}$, which, taking $w=\chi^{(k)}_i$ and using \eqref{eq:eigdeiud3dd}, translates into $(\pi^{(k-1,k)}W^{(k),T})_{j,i}=0$. Writing $|\J^{(k)}|$ the number of elements  of $\J^{(k)}$ (which is equal to the dimension of $\W^{(k)}$), observe that $|\J^{(k)}|=|I^{(k)}|-|I^{(k-1)}|$. Therefore $\Img(W^{(k),T})=\Ker(\pi^{(k-1,k)})$ also implies that the $|\J^{(k)}|$ elements $\chi^{(k)}_i$ are linearly independent and, therefore, form a basis of $\W^{(k)}$.
\end{proof}
\begin{Remark}\label{rmkljdlkdjiji}
Observe that since $0=\<\psi_i^{(k-1)}, \chi_j^{(k)}\>_a=(R^{(k-1,k)} A^{(k)}W^{(k),T})_{i,j}$, it also holds true that $\Img(W^{(k),T})=\Ker(R^{(k-1,k)}A^{(k)})$ and $\Img(A^{(k)} W^{(k),T})= \Ker(R^{(k-1,k)})$ .
\end{Remark}
From now on we choose, for each $k\in \{2,\ldots,q\}$, a $\J^{(k)}\times \I^{(k)}$ matrix $W^{(k)}$
as in Lemma \ref{lemwk}.  This choice is not unique and to enable fast  multiplication by $W^{(k)}$ (or its transpose) we require that for $(j,i)\in \J^{(k)}\times \I^{(k)}$, $W^{(k)}_{j,i}=0$ if $j^{(k-1)}\not=i^{(k-1)}$. Therefore, the construction of $W^{(k)}$  requires, for each $s\in \I^{(k-1)}$, to specify a number $m_s-1$ of $m_s$-dimensional vectors   $W^{(k)}_{(s,1),(s,\cdot)},\ldots, W^{(k)}_{(s,m_s-1),(s,\cdot)}$ that are linearly independent and orthogonal to the $m_s$-dimensional vector $(1,1,\ldots,1,1)$. We propose two simple constructions.

\begin{Construction}\label{const1}
For $k\in \{2,\ldots,q\}$, choose $W^{(k)}$ such (1) $W^{(k)}_{j,i}=0$ for $(j,i)\in \J^{(k)}\times \I^{(k)}$ with  $j^{(k-1)}\not=i^{(k-1)}$ and (2) for
$s\in \I^{(k-1)}$, $t\in \{1,\ldots,m_s-1\}$ and $t'\in \{1,\ldots,m_s\}$,  $W^{(k)}_{(s,t),(s,t')}=\delta_{t,t'}-\delta_{t+1,t'}$.
\end{Construction}
For $k\in \{2,\ldots,q\}$  and $i=(i_1,\ldots,i_{k-1},i_k)\in \J^{(k)}$ define $i^+:=(i_1,\ldots,i_{k-1},i_k+1)$ and observe that under construction \ref{const1},
\begin{equation}\label{eqchipsi}
\chi^{(k)}_i=\psi^{(k)}_{i}-\psi^{(k)}_{i^+}
\end{equation}
whose game-theoretic interpretation is provided in Figure \ref{fig:bets}.
 \begin{figure}[h!]
	\begin{center}
			\includegraphics[width=\textwidth]{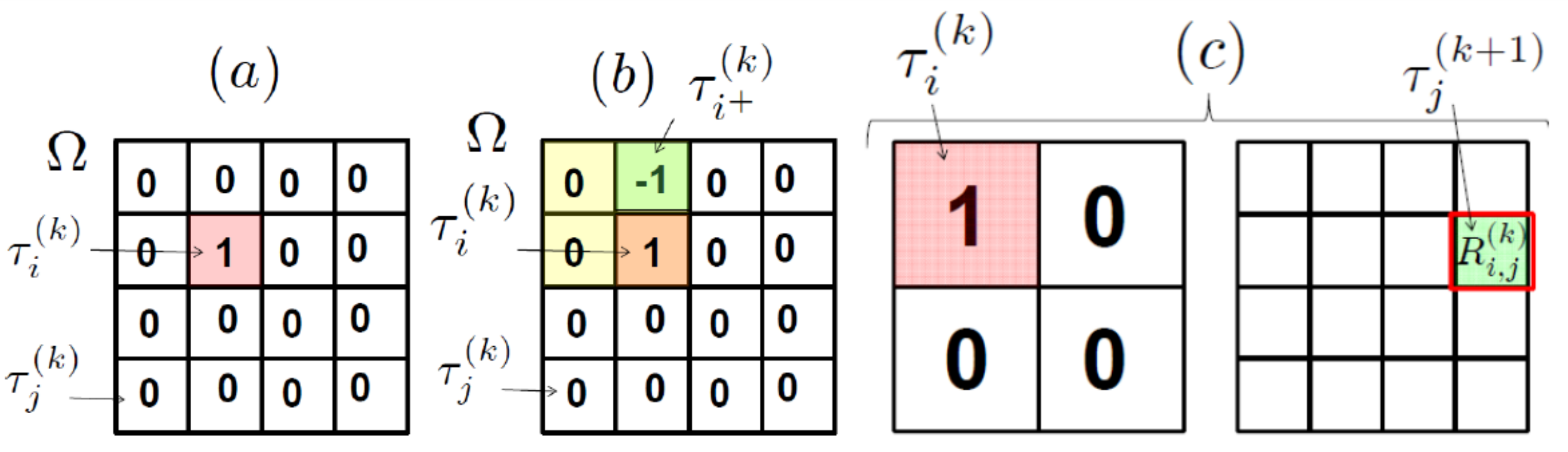}
		\caption{If $(\tau_s^{(k)}, s\in \I^{(k)})$ is a nested rectangular partition of $\Omega$ then (a) $\psi_i^{(k)}$ is Player II's best bet on the value of the solution $u$ of \eqref{eqn:scalar} given  $\int_{\tau^{(k)}_j}u=\delta_{i,j}$ for $j\in \I^{(k)}$ (b) $\chi_i^{(k)}$ is Player II's best bet on $u$ given  $\int_{\tau^{(k)}_j}u=\delta_{i,j}-\delta_{i^+,j}$ for $j\in \I^{(k)}$ (c) $R^{(k,k+1)}_{i,j}$ is Player II's best bet on  $\int_{\tau^{(k+1)}_j} u$  given  $\int_{\tau^{(k)}_j}u=\delta_{i,j}$ for $j\in \I^{(k)}$.}\label{fig:bets}
	\end{center}
\end{figure}

For the second construction we need the following lemma whose proof is trivial.
\begin{Lemma}\label{const0}
Let $U^{(n)}$ be the sequence of $n\times n$ matrices defined (1) for $n=2$ by $U^{(2)}_{1,\cdot}=(1,-1)$ and $U^{(2)}_{2,\cdot}=(1,1)$ and (2) iteratively for $n\geq 2$ by $U^{(n+1)}_{i,j}=U^{(n)}_{i,j}$ for $1\leq i,j \leq n$, $U^{(n+1)}_{n+1,j}=1$ for $1\leq j \leq n+1$, $U^{(n+1)}_{i,n+1}=0$ for $1\leq i\leq n-1$ and
$U^{(n+1)}_{n,n+1}=-n$. Then for $n\geq 2$, the rows of $U^{(n)}$ are orthogonal, $U^{(n)}_{n,j}=1$ for $1\leq j \leq n$ and we write $\bar{U}^{(n)}$ the corresponding orthonormal matrix obtained by renormalizing the rows of $U^{(n)}$.
\end{Lemma}

\begin{Construction}\label{const2}
For $k\in \{2,\ldots,q\}$, choose $W^{(k)}$ such (1) $W^{(k)}_{j,i}=0$ for $(j,i)\in \J^{(k)}\times \I^{(k)}$ with  $j^{(k-1)}\not=i^{(k-1)}$ and (2) for
$s\in \I^{(k-1)}$ and $t\in \{1,\ldots,m_s-1\}$ and $t' \in \{1,\ldots,m_s\}$ $W^{(k)}_{(s,t),(s,t')}=\bar{U}^{(m_s)}_{t,t'}$ (where $\bar{U}^{(m_s)}$ is defined in Lemma \ref{const0}).
\end{Construction}
Observe that under Construction \ref{const2} (1) the complexity of constructing $W^{(k)}$ is $|\I^{(k-1)}|\times m_s^2$ and (2) $W^{(k)}W^{(k),T}=J^{(k)}$ where $J^{(k)}$ is the $\J^{(k)}\times \J^{(k)}$ identity matrix.

\subsection{Multiresolution operator inversion}\label{subsecmultiresoperatorinversion}
We will now use the basis functions $\psi_i^{(1)}$ and $\chi^{(k)}_i$ to perform the multiresolution inversion of \eqref{eqn:scalar}. Let $B^{(k)}$ be the $\J^{(k)}\times \J^{(k)}$ (stiffness)matrix $B^{(k)}_{i,j}=\<\chi_i^{(k)},\chi_j^{(k)}\>_a$ and observe that
\begin{equation}\label{eqjgfytfjhyyyg}
B^{(k)}= W^{(k)}A^{(k)}W^{(k),T}
\end{equation}
Observe that $B^{(k)}$ is positive, symmetric, definite and write $B^{(k),-1}$ its inverse.
Let $\bar{\pi}^{(k,k+1)}$ be the $\I^{(k)}\times \I^{(k+1)}$ matrix defined by
\begin{equation}
\bar{\pi}^{(k,k+1)}_{i,j}=\pi^{(k,k+1)}_{i,j}/ (\pi^{(k,k+1)} \pi^{(k+1,k)})_{i,i}
\end{equation}
Using the notations of Definition \ref{defindextree} note that $(\pi^{(k,k+1)} \pi^{(k+1,k)})_{i,i}=m_{i}$.
Let $D^{(k,k-1)}$ be the $\J^{(k)}\times \I^{(k-1)}$ matrix defined as
\begin{equation}\label{eqdkdjhdkhse}
 D^{(k,k-1)}:=- B^{(k),-1}W^{(k)}A^{(k)}\bar{\pi}^{(k,k-1)}
\end{equation}
and write $D^{(k-1,k)}:=D^{(k,k-1),T}$ its transpose.

\begin{Theorem}
It holds true that for $k\in \{1,\ldots,q-1\}$ and $i\in \I^{(k)}$,
\begin{equation}\label{eqdidhduhh}
\psi_i^{(k)}= \sum_{l\in \I^{(k+1)}}\bar{\pi}^{(k,k+1)}_{i,l}\psi^{(k+1)}_l+ \sum_{j\in \J^{(k+1)}} D^{(k,k+1)}_{i,j} \chi^{(k+1)}_j\,.
\end{equation}
In particular,
\begin{equation}\label{eqhuhiddeuv}
R^{(k,k+1)}= \bar{\pi}^{(k,k+1)}+D^{(k,k+1)}W^{(k+1)}
\end{equation}
\end{Theorem}
\begin{proof}
For $s\in \I^{(k)}$ write $\bar{\psi}_s^{(k)}:= \sum_{l\in \I^{(k+1)}}\bar{\pi}^{(k,k+1)}_{s,l}\psi^{(k+1)}_l$ and $\bar{\V}^{(k)}:=\Span\{\bar{\psi}_s^{(k)}\mid s\in \I^{(k)} \}$. Let $x\in \R^{I^{(k)}}$, $y\in \R^{\J^{(k+1)}}$ and
\begin{equation}\label{eqdhdihedu65c}
\psi =\sum_{s\in \I^{(k)}} x_s \bar{\psi}_s^{(k)} + \sum_{j\in \J^{(k+1)}}y_j \chi_j^{(k+1)}\,.
 \end{equation}
 If $\psi=0$ then integrating $\psi$ against $\phi_i^{(k)}$ for $i\in \I^{(k)}$ (and observing that $\int_{\Omega}\phi_i^{(k)} \bar{\psi}_s^{(k)}= \delta_{i,s}$) implies $x=0$ and $y=0$. Therefore the elements $\bar{\psi}_s^{(k)}, \chi_j^{(k+1)}$ form a basis for  $\bar{\V}^{(k)}+\W^{(k+1)}$.  Observing that $\operatorname{dim}(\V^{(k+1)})=\operatorname{dim}(\bar{\V}^{(k)})+\operatorname{dim}(\W^{(k+1)})$ we deduce that $\V^{(k+1)}=\bar{\V}^{(k)}+\W^{(k+1)}$. Therefore, since $\V^{(k)}\subset \V^{(k+1)}$, $\psi_i^{(k)}$ can be decomposed as in \eqref{eqdhdihedu65c}. The constraints $\int_{\Omega} \phi_s^{(k)}\psi_i^{(k)}=\delta_{i,s}$ lead to $x_s=  \delta_{i,s}$.  The orthogonality between $\psi$ and $\W^{(k+1)}$ leads to the equations $\<\psi,\chi^{(k+1)}_j\>_a=0$ for $j\in \J^{(k+1)}$, i.e.\\
$\sum_{l\in \I^{(k+1)}}\bar{\pi}^{(k,k+1)}_{i,l}\<\psi^{(k+1)}_l,\chi^{(k+1)}_j\>_a+\sum_{j'\in \J^{(k+1)}} y_{j'} \<\chi^{(k+1)}_{j'},\chi^{(k+1)}_j\>_a=0$, which translates into $ W^{(k+1)} A^{(k+1)} \bar{\pi}^{(k+1,k)}_{\cdot,i} + B^{(k+1)} y$, that is \eqref{eqdidhduhh}.
 Plugging \eqref{eqjkhdkdh} in \eqref{eqdidhduhh} and comparing with \eqref{eq:ftfytftfx} leads to \eqref{eqhuhiddeuv}.
\end{proof}

Let $g$ be the r.h.s of \eqref{eqn:scalar}. For $k\in \{1,\ldots,q\}$ let $g^{(k)}$ be the $|\I^{(k)}|$-dimensional vector defined by
$g^{(k)}_i=\int_{\Omega}\psi_i^{(k)}g \text{ for }i\in \I^{(k)}$.
Observe that $g^{(k)}$ can be computed iteratively using
\begin{equation}\label{eqyguugy6t}
g^{(k)}=R^{(k,k+1)} g^{(k+1)}\,.
\end{equation}

For $k\in \{2,\ldots,q\}$, let $w^{(k)}$ be $|\J^{(k)}|$-dimensional vector defined as the solution of
\begin{equation}\label{eqsdjoejddi1}
B^{(k)} w^{(k)}=W^{(k)} g^{(k)}
\end{equation}
Furthermore let $U^{(1)}$ be the $|\I^{(1)}|$-dimensional vector
 defined as the solution of
\begin{equation}\label{eqsdjoejddi2}
A^{(1)} U^{(1)}=g^{(1)}
\end{equation}
According to following theorem, which is a direct consequence of Theorem \ref{thmgugyug2}, the solution of \eqref{eqn:scalar} can be computed at any scale by solving the decoupled linear systems \eqref{eqsdjoejddi1} and \eqref{eqsdjoejddi2}.
\begin{Theorem}\label{thddwedmgugyug}
For $k\in \{2,\ldots,q\}$, let $u^{(k)}$ be the finite element solution of \eqref{eqn:scalar} in $\V^{(k)}$. It holds true that $u^{(k)}-u^{(k-1)}=\sum_{i\in \J^{(k)}}w^{(k)}_i \chi^{(k)}_i$ and, in particular,
\begin{equation}
u^{(k)}=\sum_{i \in \I^{(1)}} U^{(1)}_i \psi^{(1)}_i+\sum_{k'=2}^k \sum_{i\in \J^{(k')}}w^{(k')}_i \chi^{(k')}_i
\end{equation}
\end{Theorem}

\subsection{Uniformly bounded condition numbers across subscales/subbands}\label{subsecmultires}

Taking $q=\infty$ in Theorem \ref{thmgugyug2},  the construction of the basis elements $\psi^{(k)}_i$ leads to the  multiresolution orthogonal decomposition,
\begin{equation}
H^1_0(\Omega)=\V^{(1)}\underset{i=2}{\overset{\infty}{\oplus_a}}  \W^{(i)}.
\end{equation}
In that sense the basis elements $\psi^{(k)}_i$ and $\chi^{(k)}_i$ could be seen as a generalization of wavelets to the orthogonal decomposition of $H^1_0(\Omega)$ (rather than $L^2(\Omega)$) adapted to the solution space of the PDE \eqref{eqn:scalar}. We will now show that
this orthogonal decomposition induces a subscale decomposition of the operator $-\diiv(a\nabla)$ into layered subbands of increasing frequencies. Moreover the condition number of the operator $-\diiv(a\nabla)$ restricted to each subspace $\W^{(k)}$ will be shown to be  uniformly bounded if $H_{k-1}/H_{k}$ is uniformly bounded (e.g. if $H_k$ is a geometric sequence). Write $H_0:=1$ and let $\delta$ be defined as in Construction \ref{defmulires}.

\begin{Theorem}\label{thmuuhiuhddu}
If   $k\in \{1,\ldots,q\}$ and $v\in \V^{(k)}$ then
\begin{equation}\label{eqguygugu68lhs}
 \frac{\delta^{1+d/2}}{\sqrt{\lambda_{\max}(a)} 2^{5/2+d/2}} H_{k} \leq  \frac{\|v\|_a}{\|\diiv(a\nabla v)\|_{L^2(\Omega)}}\,.
\end{equation}
If $(k=1$ and $v\in \V^{(1)})$ or $(k\in \{2,\ldots,q\}$ and $v\in \W^{(k)})$ then
\begin{equation}\label{eqguygugu68}
  \frac{\|v\|_a}{\|\diiv(a\nabla v)\|_{L^2(\Omega)}}  \leq \frac{1}{ \sqrt{\lambda_{\min}(a)}} H_{k-1}\,.
\end{equation}
\end{Theorem}
\begin{proof}
\eqref{eqguygugu68lhs} is a direct consequence of Lemma \ref{lemdhkedjhdkjh}. For $k=1$ \eqref{eqguygugu68} is a simple consequence of Poincar\'{e}'s inequality.
Let $k\in \{2,\ldots,q\}$.   $\V^{(k)}=\V^{(k-1)}\oplus_a \W^{(k)}$ and Theorem \ref{thmgugyug0} imply\\
$
\sup_{v\in \W^{(k)}} \frac{\|v\|_a}{\|\diiv(a\nabla v)\|_{L^2(\Omega)}} \leq  \sup_{v\in \V^{(k)}} \inf_{v' \in \V^{(k-1)}} \frac{\|v-v'\|_a}{\|\diiv(a\nabla v)\|_{L^2(\Omega)}} \leq \frac{2}{\pi \sqrt{\lambda_{\min}(a)}} H_{k-1}.
$\\
\end{proof}

Write  $|c|$ the Euclidean norm of $c$ and for $k\in \{1,\ldots,q\}$ let
\begin{equation}\label{eqgam1}
\ubar{\gamma}_k:=\inf_{c\in \R^{\I^{(k)}}} \frac{\| \sum_{i\in \I^{(k)}} c_i  \,  \phi_i^{(k)} \|_{L^2(\Omega)}^2}{|c|^2} \text{ and }\bar{\gamma}_k:=\sup_{c\in \R^{\I^{(k)}}} \frac{\| \sum_{i\in \I^{(k)}} c_i \,   \phi_i^{(k)} \|_{L^2(\Omega)}^2}{|c|^2}
\end{equation}
Write $|\tau|$ the volume of a set $\tau$ and note that $\bar{\gamma}_k\leq \max_{i\in \I^{(k)}} |\tau_i^{(k)}|$ and $\ubar{\gamma}_k\geq \min_{i\in \I^{(k)}} |\tau_i^{(k)}|$, therefore $\bar{\gamma}_k/\ubar{\gamma}_k \leq \delta^{-d}$.

For a given matrix $M$, write $\operatorname{Cond}(M):=\sqrt{\lambda_{\max}(M^T M)}/\sqrt{\lambda_{\min}(M^T M)}$ its  condition number.
\begin{Theorem}\label{thmodhehiudhehd}
It holds true that
\begin{equation}\label{eqcond1}
\operatorname{Cond}(A^{(1)})\leq \frac{1}{H_{1}^2}\frac{\lambda_{\max}(a) 2^{5+d}}{\lambda_{\min}(a)  \delta^{2+2 d} },
\end{equation}
and for $k\in \{2,\ldots,q\}$,
\begin{equation}\label{eqcond2}
\operatorname{Cond}(B^{(k)})\leq  \big(\frac{ H_{k-1}}{ H_{k}}\big)^{2} \big(\frac{\lambda_{\max}(a)}{\lambda_{\min}(a)}\big)^2
 \frac{ 2^{11+2d}}{\delta^{4+7d} \pi^2  } \operatorname{Cond}(W^{(k)}W^{(k),T})\,.
\end{equation}
Furthermore,  $\operatorname{Cond}(W^{(k)}W^{(k),T})=1$ under Construction \ref{const2} and
$\operatorname{Cond}(W^{(k)}W^{(k),T})\leq 2 \big(H_{k-1}/(\delta H_k)\big)^{2d}$ under Construction \ref{const1}.
\end{Theorem}
\begin{proof}
Let $k\in \{1,\ldots,q\}$ and $c\in \R^{\I^{(k)}}$. Write $v=\sum_{i\in \I^{(k)}} c_i \psi_i^{(k)}$. Observing that $\|v\|_a^2=c^T A^{(k)} c$ and
$\|\diiv(a\nabla v)\|_{L^2(\Omega)}^2=\|\sum_{i\in \I^{(k)}} (A^{(k)} c)_i \phi_i^{(k)}\|_{L^2(\Omega)}^2 \geq \ubar{\gamma}_k |A^{(k)} c|^2$, \eqref{eqguygugu68lhs} implies that $\ubar{\gamma}_k H_k^2 \delta^{2+d}/(\lambda_{\max}(a) 2^{5+d}) \leq c^T A^{(k)}c/ |A^{(k)} c|^2$, which, after taking the minimum in $c$ leads to (for $k\geq 1$)
\begin{equation}\label{eqkhiduhdf7d}
 \lambda_{\max}(A^{(k)})\leq \lambda_{\max}(a) 2^{5+d}/(H_{k}^2 \delta^{2+d}  \ubar{\gamma}_k ),
\end{equation}
and for $k\geq 2$ (using \eqref{eqjgfytfjhyyyg})
\begin{equation}\label{eqkhiduhde2df7d}
 \lambda_{\max}(B^{(k)})\leq  \lambda_{\max}(W^{(k)}W^{(k),T}) \lambda_{\max}(a) 2^{5+d}/(H_{k}^2 \delta^{2+d}  \ubar{\gamma}_k )\,.
\end{equation}
Similarly for $k=1$ \eqref{eqguygugu68} leads to  $\lambda_{\min}(A^{(1)})\geq \lambda_{\min}(a)/\bar{\gamma}_1$.
Now let us consider $k\in \{2,\ldots,q\}$ and $c\in \R^{\J^{(k)}}$. Write $w=\sum_{i\in \J^{(k)},\,j \in \I^{(k)}} c_i W_{i,j}^{(k)} \psi_j^{(k)}$.
\eqref{eqjkhdkdh} and \eqref{eqjgfytfjhyyyg} imply that $\|w\|_a^2=c^T B^{(k)} c$ and (using \eqref{eqgam1})\\  $\|\diiv(a\nabla w) \|_{L^2(\Omega)}=\| \sum_{i\in \J^{(k)},\,j \in \I^{(k)}}  (A^{(k)} W^{(k),T} c)_{j}  \phi_j^{(k)}\|_{L^2(\Omega)}\leq \bar{\gamma}_k |A^{(k)} W^{(k),T} c|^2$. Observing that $w\in \W^{(k)}$,
\eqref{eqguygugu68} implies that
$\frac{c^T B^{(k)} c}{c^T W^{(k)} (A^{(k)})^2 W^{(k),T} c} \leq \bar{\gamma}_k\,\frac{1}{ \lambda_{\min}(a)} H_{k-1}^2$.
Taking  $c=B^{(k),-1} y$ for $y\in \R^{\J^{(k)}}$ we deduce that
$ \frac{y^T B^{(k),-1} y}{|A^{(k)} W^{(k),T} B^{(k),-1} y|^2 } \leq \bar{\gamma}_k\,\frac{1}{ \lambda_{\min}(a)} H_{k-1}^2$.
Writing $N^{(k)}=-A^{(k)} W^{(k),T} B^{(k),-1}$, we have obtained that
\begin{equation}\label{eqjgjhdjhgdy}
\lambda_{\min}(a)/\big(H_{k-1}^2 \bar{\gamma}_k\lambda_{\max}(N^{(k),T}N^{(k)})\big) \leq \lambda_{\min}(B^{(k)})\,.
\end{equation}
For $k\in \{2,\ldots,q\}$ let $P^{(k)}:=\pi^{(k,k-1)} R^{(k-1,k)}$. Using $R^{(k-1,k)}=A^{(k-1)}\pi^{(k-1,k)}\Theta^{(k)}$ and
$\pi^{(k-1,k)}\Theta^{(k)}\pi^{(k,k-1)}=\Theta^{(k-1)}$ (Theorem \ref{thmhggfees5}) we obtain that $(P^{(k)})^2=P^{(k)}$, i.e. $P^{(k)}$ is a projection. Write $\|P^{(k)}\|_{\Ker(\pi^{(k-1,k)})}:=\sup_{x\in \Ker(\pi^{(k-1,k)})} |P^{(k)} x|/|x|$.

\begin{Lemma}\label{lemfdhgdf}
It holds true that for $k\in \{2,\ldots,q\}$,
\begin{equation}
\lambda_{\max}(N^{(k),T}N^{(k)})\leq \frac{1+\|P^{(k)}\|_{\Ker(\pi^{(k-1,k)})}^2}{\lambda_{\min}(W^{(k)} W^{(k),T})}\,.
\end{equation}
\end{Lemma}
\begin{proof}
Since $\Img(W^{(k),T})$ and $\Img(\pi^{(k,k-1)})$ are orthogonal and $\dim(\R^{\I^{(k)}})=\dim\big(\Img(W^{(k),T})\big)+\dim\big(\Img(\pi^{(k,k-1)})\big)$,
for $x\in \R^{\I^{(k)}}$ there exists a unique $y\in \R^{\J^{(k)}}$ and $z\in \R^{\I^{(k-1)}}$ such that
$x=W^{(k),T}y+\pi^{(k,k-1)} z$
and
$
|x|^2=|W^{(k),T}y|^2+|\pi^{(k,k-1)} z|^2.
$
Observe that $W^{(k)} x=W^{(k)}W^{(k),T}y$ (since $W^{(k)}\pi^{(k,k-1)}=0$) and  $R^{(k-1,k)}x =R^{(k-1,k)} W^{(k),T}y+z$ (since $R^{(k-1,k)}\pi^{(k,k-1)}=I^{(k-1)}$ from Theorem \ref{thmhggfees5}). Therefore,
$
|x|^2=|W^{(k),T} y|^2+|P^{(k)} (x-W^{(k),T}y)|^2
$
with $y=(W^{(k)} W^{(k),T})^{-1}W^{(k)} x$.
Let $v\in \R^{\J^{(k)}}$. Taking $x=A^{(k)}W^{(k),T}v$  and observing that $P^{(k)}x=0$ (since $R^{(k-1,k)}A^{(k)}W^{(k),T}=0$ from the $\<\cdot,\cdot\>_a$-orthogonality between $\V^{(k-1)}$ and $\W^{(k)}$)
 leads to
$
|A^{(k)}W^{(k),T}v|^2=|W^{(k),T} y |^2+ |P^{(k)} W^{(k),T}y |^2
$
with $y=(W^{(k)} W^{(k),T})^{-1} B^{(k)} v$. Therefore
$
|A^{(k)}W^{(k),T}v|^2\leq (1+\|P^{(k)}\|_{\Ker(\pi^{(k-1,k)})}^2)\frac{|B^{(k)}v|^2}{\lambda_{\min}(W^{(k)} W^{(k),T})},
$
which concludes the proof after taking $v=B^{(k),-1}v'$ and maximizing the l.h.s. over $|v'|=1$.
\end{proof}
\begin{Lemma}\label{lemdjoidjdi}
Writing $\|M\|_2:=sup_x |M x|/x$ the spectral norm, we have
\begin{equation}
\|P^{(k)}\|_{\Ker(\pi^{(k-1,k)})}^2 \leq \|\pi^{(k,k-1)}A^{(k-1)}\pi^{(k-1,k)}\|_2  \sup_{x\in \Ker(\pi^{(k-1,k)})}\frac{x^T \Theta^{(k)} x}{x^T x}
\end{equation}
\end{Lemma}
\begin{proof}
Let $x\in \Ker(\pi^{(k-1,k)})$. Using
$P^{(k)}=\pi^{(k,k-1)}A^{(k-1)}\pi^{(k-1,k)}\Theta^{(k)}$
 we obtain that
$
|P^{(k)}x|=\|\pi^{(k,k-1)}A^{(k-1)}\pi^{(k-1,k)}(\Theta^{(k)})^\frac{1}{2}\|_2 |(\Theta^{(k)})^\frac{1}{2} x|
$.
Observing that   for \\$M=\pi^{(k-1,k)}(\Theta^{(k)})^\frac{1}{2}$ we have $M M^T=\Theta^{(k-1)}$ and for $N=\pi^{(k,k-1)}A^{(k-1)}\pi^{(k-1,k)}(\Theta^{(k)})^\frac{1}{2}$ we have $\lambda_{\max}(N^T N)=\lambda_{\max}(N N^T)$ we deduce
\\$\|\pi^{(k,k-1)}A^{(k-1)}\pi^{(k-1,k)}(\Theta^{(k)})^\frac{1}{2}\|_2^2=\|\pi^{(k,k-1)}A^{(k-1)}\pi^{(k-1,k)}\|_2$ and conclude by taking the supremum over
$x\in \Ker(\pi^{(k-1,k)})$.
\end{proof}
\begin{Lemma}\label{lemddjoj3ir}
It holds true that
\begin{equation}
 \sup_{x\in \Ker(\pi^{(k-1,k)})}\frac{x^T \Theta^{(k)} x}{x^T x} \leq    H_{k-1}^2   \frac{\bar{\gamma}_k^2}{\ubar{\gamma}_k\pi^2 \lambda_{\min}(a)}
\end{equation}
\end{Lemma}
\begin{proof}
Let $y\in \R^{\J^{(k)}}$ and $\alpha \in \R$. Let $x=\alpha W^{(k),T} y$. Write $\phi=\sum_{i\in \I^{(k)}} x_i \phi_i^{(k)}$ and $\psi=(-\diiv(a\nabla \cdot))^{-1} \phi$. Observe that $ \|\psi\|^2_a=x^T \Theta^{(k)} x \geq \alpha y^T W^{(k)}\Theta^{(k)} W^{(k),T}y $.  Using $\int_{\Omega}\phi_i^{(k)} \phi_l^{(k)} =0$ for $i\not=l$ and selecting $\alpha=\|\phi_i^{(k)}\|_{L^2(\Omega)}^{-2}$ (assuming, without loss of generality,  that $\|\phi_i^{(k)}\|_{L^2(\Omega)}^2=|\tau_i^{(k)}|$ is constant in $i$, for the general case,  rescale each $\phi_i^{(k)}$ by a multiplicative constant) we obtain that for $j\in \I^{(k-1)}$,
$\int_{\Omega}\phi \phi_j^{(k-1)} =\sum_{i\in \I^{(k)}} x_i  \|\phi_i^{(k)}\|_{L^2(\Omega)}^2 \pi^{(k-1,k)}_{j,i}=(\pi^{(k-1,k)} W^{(k),T} y)_j=0
$.
Therefore, since $\|\psi\|^2_a=\int_{\Omega}\phi \psi$,  we have for $\psi'\in \operatorname{span}\{\phi_i^{(k-1)}\mid i\in \I^{(k-1)}\}$
$
\|\psi\|^2_a =\int_{\Omega}\phi (\psi-\psi') \leq \|\phi\|_{L^2(\Omega)} \|\psi-\psi'\|_{L^2(\Omega)}
$.
Choosing $\psi'= \sum_{i\in \I^{(k-1)}} \phi_i^{(k-1)} \int_{\Omega}\psi \phi_i^{(k-1)}/ \|\phi_i^{(k-1)}\|_{L^2(\Omega)}^2$ we obtain (via Poincar\'{e} and Cauchy-Schwartz inequalities as in the proof of Proposition \ref{prop:gegddgdjdef}) that
$\|\psi-\psi'\|_{L^2(\Omega)} \leq  H_{k-1} \|\psi\|_a/(\pi \sqrt{\lambda_{\min}(a)})$ and deduce $\|\psi\|_a \leq  H_{k-1} \|\phi\|_{L^2(\Omega)} /(\pi \sqrt{\lambda_{\min}(a)})$.
 Observing that $\|\phi\|_{L^2(\Omega)}^2\leq |x|^2 \bar{\gamma}_k$  and $ \ubar{\gamma}_k \leq \alpha^{-1} \leq \bar{\gamma}_k$ we
summarize and obtain that\\
$
y^T W^{(k)}\Theta^{(k)} W^{(k),T}y \leq   H_{k-1}^2 |x|^2 \bar{\gamma}_k^2 /(\pi^2 \lambda_{\min}(a)) \leq      H_{k-1}^2  |W^{(k),T}y|^2 \bar{\gamma}_k^2 /(\ubar{\gamma}_k\pi^2 \lambda_{\min}(a))
$, which concludes the proof of the lemma (since $\Ker(\pi^{(k-1,k)})=\Img(W^{(k),T})$).
\end{proof}

Observing that $ \|\pi^{(k,k-1)}A^{(k-1)}\pi^{(k-1,k)}\|_2 \leq \lambda_{\max}(\pi^{(k,k-1)}\pi^{(k-1,k)}) \lambda_{\max}(A^{(k-1)})$ and using \eqref{eqkhiduhdf7d}, we derive from lemmas \ref{lemdjoidjdi} and \ref{lemddjoj3ir} that
\begin{equation}\label{eqdihduhuq2he}
\|P^{(k)}\|_{\Ker(\pi^{(k-1,k)})}^2 \leq \lambda_{\max}(\pi^{(k,k-1)}\pi^{(k-1,k)})      \frac{  \bar{\gamma}_k^2  2^{5+d} \lambda_{\max}(a)}{  \ubar{\gamma}_k \ubar{\gamma}_{k-1} \delta^{2+d}   \pi^2 \lambda_{\min}(a)}\,.
\end{equation}
Observing that $\pi^{(k,k-1)}\pi^{(k-1,k)}$ is block-diagonal and using the notations of Definition \ref{defindextree} we have
 $\lambda_{\max}(\pi^{(k,k-1)}\pi^{(k-1,k)})=\max_{j\in \I^{(k-1)}}\sup_{x\in \R^{m_j}} |\sum_{i=1}^{m_j} x_i|^2/|x|^2=\max_{j\in \I^{(k-1)}} m_j$. Noting that a set $\tau_j^{(k-1)}$ can contain at most $(\max_{j\in \I^{(k-1)}} |\tau_j^{(k-1)}|)/ (\min_{i\in \I^{(k)}} |\tau_i^{(k)}|)$ subsets $\tau_i^{(k)}$ we have
 \begin{equation}\label{eqddkhjji}
 \max_{j\in \I^{(k-1)}} m_j \leq \big(H_{k-1}/(\delta H_k)\big)^d
 \end{equation}
  and conclude that $\lambda_{\max}(\pi^{(k,k-1)}\pi^{(k-1,k)}) \leq \big(H_{k-1}/(\delta H_k)\big)^d$.
 Therefore \eqref{eqjgjhdjhgdy} and Lemma \ref{lemfdhgdf} imply, after simplification, that
\begin{equation}\label{eqjgjhdjghgfhgdy}
 \lambda_{\min}(B^{(k)}) \geq \frac{\lambda_{\min}(a)}{H_{k-1}^2 \bar{\gamma}_k} \lambda_{\min}(W^{(k)} W^{(k),T}) \frac{ H_{k}^{d} \ubar{\gamma}_{k-1} \ubar{\gamma}_k \delta^{2+2d}   \pi^2 \lambda_{\min}(a)}{ H_{k-1}^{d} \bar{\gamma}_k^2  2^{6+d} \lambda_{\max}(a)}\,.
\end{equation}
Recalling that $\bar{\gamma}_k/\ubar{\gamma}_k \leq \delta^{-d}$, using $\bar{\gamma}_k/\ubar{\gamma}_{k-1}\leq H_k^d/ (H_{k-1} \delta)^d$, and summarizing we conclude the proof of \eqref{eqcond1} and \eqref{eqcond2}. Recall that under construction \ref{const2} we have $W^{(k)}W^{(k),T}=J^{(k)}$ which implies $\operatorname{Cond}(W^{(k)}W^{(k),T})=1$.  Under
 construction \ref{const1}, $W^{(k)}W^{(k),T}$ is block diagonal with for $j\in \I^{(k-1)}$, diagonal blocks corresponding to $(m_j-1)\times (m_j-1)$ matrices $M^{(m_j-1)}$ such that (1) for $n=1$ and $x\in \R$, $x^T M^{(1)}x=2 x^2$ (2) for $n=2$ and $x\in \R^2$, $x^T M^{(2)}x=x_1^2+(x_2-x_1)^2+x_2^2$ and (3) for
 for $n\geq 3$, and $x\in \R^n$, $x^T M^{(n)} x=x_1^2+\sum_{i=1}^{n-2} (x_{i}-x_{i+1})^2+ x_{n}^2$. Note that for all $n\geq 1$, $\lambda_{\max}(M^{(n)})\leq 3$. Furthermore, for $n\geq 3$ ($n\leq 2$ is trivial), introducing the variables $y_2=x_2-x_1,\ldots,y_n=x_n-x_{n-1}$ we obtain that
 $x^T M^{(n)} x=x_1^2+y_2^2+\cdots+y_n^2+ x_n^2$ and $|x|^2=x_1^2+(x_1+y_2)^2+\cdots+(x_1+y_2+\cdots+y_n)^2 \leq (x^T M^{(n)} x) n(n+1)/2$. Therefore,
 $\lambda_{\min}(M^{(n)})\geq  2/(n(n+1))$. We conclude that under   construction \ref{const1} $\operatorname{Cond}(W^{(k)}W^{(k),T})\leq \max_{j\in \I^{(k-1)}} 3 (m_j-1)m_j/2$ and bound $m_j$ as in \eqref{eqddkhjji}.
\end{proof}

\subsection{Well conditioned relaxation across subscales}\label{subseccg}
If $H_k$ is a geometric sequence or if $H_{k-1}/H_k$ is uniformly bounded, then, by Theorem \eqref{thmodhehiudhehd}, the linear systems  (\eqref{eqsdjoejddi1} and \eqref{eqsdjoejddi2})  entering in the calculation of the gamblets $\chi^{(k)}_i$ (and therefore $\psi_i^{(k)}$) and the subband/subscale solutions $u^{(1)}$ and $u^{(k+1)}-u^{(k)}$ have uniformly bounded condition numbers (in particular, these condition numbers are bounded independently from  mesh size/resolution and the regularity of $a(x)$).
Therefore these systems can be solved efficiently using iterative methods.
One such methods is the Conjugate Gradient (CG) method \cite{HestenesStiefel1952}.
Recall \cite{Shewchuk1994} that the application of  the CG method to a linear system $A x=b$ (where $A$ is a $n\times n$ symmetric positive definite matrix) with initial guess $x^{(0)}$, yields a sequence of approximations $x^{(l)}$ satisfying (writing $|e|_A^2:=e^T A e$)
$|x-x^{(l)}|_A \leq 2 \Big(\frac{\sqrt{\Cond(A)}-1}{\sqrt{\Cond(A)}+1}\Big)^l|x-x^{(0)}|_A$
where $\Cond(A):=\lambda_{\max}(A)/\lambda_{\min}(A)$.
Recall \cite{Shewchuk1994} also that the maximum number of iterations required to reduce the error by a factor $\epsilon$ ($|x-x^{(l)}|_A \leq \epsilon |x-x^{(0)}|_A$) is bounded by $\frac{1}{2}\sqrt{\Cond(A)}\ln \frac{2}{\epsilon}$ and has complexity (number of required arithmetic operations)
$\mathcal{O}(\sqrt{\Cond(A)} N_A)$ (writing $N_A$ the number of non-zero entries of $A$).

\subsection{Hierarchical localization and error propagation across scales}\label{sechierarloc}
Although the multi-resolution decomposition presented  in this section leads to well conditioned linear systems, the resulting matrices $B^{(k)}$ and $A^{(k)}$ are dense and to achieve near-linear complexity in the resolution of  \eqref{eqn:scalar} these matrices must be truncated by localizing the computation of the basis functions $\psi_i^{(k)}$ (and therefore $\chi_i^{(k)}$).
The approximation error induced by these localization/truncation steps is controlled by the exponential decay of gamblets and the uniform bound on the condition numbers of the matrices $B^{(k)}$. To make this control explicit and derive a bound the size of the localization sub-domains we need to quantify the propagation of truncation/localization errors across scales and this is the purpose of this subsection.

For $k\in \{1,\ldots,q\}$, $\rho\geq 1$ and $i\in \I^{(k)}$ we define (1) $i^{\rho}$ as the subset of indices $j\in \I^{(k)}$ whose corresponding subdomains $\tau_j^{(k)}$ are at distance at most $H_k \rho$ from $\tau_i^{(k)}$ and (2) $S^i_\rho:=\cup_{j\in i^{\rho}}\tau_j^{(k)}$.
Let $\rho_1,\ldots,\rho_q\geq 1$. For $k\in \{1,\ldots,q-1\}$ and $i\in \I^{(k)}$, write $i^{\rho,k+1}$ as the subset of indices $j\in \I^{(k+1)}$ such that
$j^{(k)}\in i^{\rho}$.
For $i\in \I^{(q)}$ let $\V^{(q+1),\loc}_i:=H^1_0(S^i_{\rho_q})$. For $k\in \{1,\ldots,q\}$ and $i\in \I^{(k)}$, let $\psi^{(k),\loc}_i$ be the minimizer of
\begin{equation}\label{eqpsiikloc}
\text{Minimize } \|\psi\|_a  \text{ subject to } \psi\in \V^{(k+1),\loc}_i \text{ and } \int_{\Omega}\psi \phi_j^{(k)}=\delta_{i,j} \text{ for }j\in i^{\rho_k}
\end{equation}
where for $k\in \{1,\ldots,q-1\}$ and $i\in \I^{(k)}$, $\V^{(k+1),\loc}_i$ is defined (via induction) by $\V^{(k+1),\loc}_i:=\Span\{\psi_j^{(k+1),\loc}\mid j\in i^{\rho_k,k+1}\}$.

From now on we will assume that $H_k=H^k$ for some $H\in (0,1)$  (or simply that $H_k$ is uniformly bounded from below and above by $H^k$).
To simplify the presentation,  we will also, from now on, write $C$ any constant that depends only  $d, \Omega, \lambda_{\min}(a), \lambda_{\max}(a), \delta$  (e.g., $2 C \lambda_{\max}(a)$ will still be written $C$). The following theorem allows us to control the localization error propagation across scales.
 For $k\in \{1,\ldots,q\}$, let $A^{(k),\loc}$ be the $\I^{(k)}\times \I^{(k)}$ matrix defined by $A^{(k),\loc}_{i,j}:=\<\psi^{(k),\loc}_i,\psi^{(k),\loc}_j\>_a$ and let $\er(k)$ be the (localization) error
 $\er(k):=\big(\sum_{i\in \I^{(k)}} \|\psi_i^{(k)}-\psi_i^{(k),\loc}\|_a^2\big)^\frac{1}{2}$.

\begin{Theorem}\label{thmerrorpropagation}
For $k\in \{1,\ldots,q-1\}$, we have
\begin{equation}
\er(k)\leq C H^{-\frac{d}{2}} \er(k+1)+ C  e^{-\rho_{k}/C}  H^{\frac{d}{2}-(k+1)(d+1)}.
\end{equation}
\end{Theorem}
\begin{proof}
We will need the following lemma summarizing and simplifying some results obtained in Theorem \ref{thmuuhiuhddu} when $H_k=H^k$.
\begin{Lemma}\label{lembase}
Let $H_k=H^{k}$ and $W^{(k)}$ be as in Construction \ref{const1} or Construction \ref{const2}.
It holds true that for $k\in \{q,\ldots,2\}$ (1) $\|W^{(k)}\|_2\leq \sqrt{3}$ (2) $1/\lambda_{\min}(W^{(k)} W^{(k),T}) \leq C H^{-2d}$
 (3) $\|\bar{\pi}^{(k-1,k)}\|_2\leq C H^\frac{d}{2}$ (4)  $\|\pi^{(k-1,k)}\|_2 \leq  C H^{-d/2}$
(5) $\|R^{(k-1,k)}\|_2 \leq C H^{-d/2}$ (6) $\Cond(B^{(k)})\leq C H^{-2-2d}$ (7)  $ \lambda_{\max}(B^{(k)})\leq C H^{-k(2+d)}$\\ (8) $ 1/\lambda_{\min}(B^{(k)})\leq C H^{k (2+d)-2-2d}$. Furthermore, (9) $\operatorname{Cond}(A^{(1)})\leq C H^{-2}$ (10) $1/\lambda_{\min}(A^{(1)})\leq C H^d$
and for $k\in \{1,\ldots,q\}$ (11) $\lambda_{\max}(A^{(k)})\leq C H^{-k(2+d)}$.
\end{Lemma}
\begin{proof}
From the proof of Theorem \ref{thmodhehiudhehd} we have (1) and
$1/\lambda_{\min}(W^{(k)} W^{(k),T})\leq$ \\$ \max_{i\in \I^{(k-1)}}  (m_i-1)m_i/2$, which implies (2). For (3),
noting that $\bar{\pi}^{(k-1,k)}_{i,j}=0$ if $j^{(k-1)}\not=i$ and $\bar{\pi}^{(k-1,k)}_{i,j}=1/m_i$ otherwise, we have $\|\bar{\pi}^{(k-1,k)}\|_2=\max_{i\in \I^{(k-1)}}1/\sqrt{m_i}\leq C H^\frac{d}{2}$. (4) follows from $\lambda_{\max}(\pi^{(k,k-1)}\pi^{(k-1,k)})= \max_{i\in \I^{(k-1)}} m_i\leq C H^{-d}$.
Let us now prove (5). Using \eqref{eqhuhiddeuv}, \eqref{eqdkdjhdkhse} and defining $N^{(k)}=-A^{(k)} W^{(k),T} B^{(k),-1}$ as in Lemma \ref{lemfdhgdf}, we have
$R^{(k-1,k)}= \bar{\pi}^{(k-1,k)}+ \bar{\pi}^{(k-1,k)}N^{(k)}  W^{(k)}$, which leads to
$\|R^{(k-1,k)}\|_2 \leq \|\bar{\pi}^{(k-1,k)}\|_2 (1+ \|N^{(k)}\|_2 \|W^{(k)}\|_2)$.
Using Lemma \ref{lemfdhgdf} and  \eqref{eqdihduhuq2he} we obtain that $\lambda_{\max}(N^{(k),T}N^{(k)})\leq \big(1+C \lambda_{\max}(\pi^{(k,k-1)}\pi^{(k-1,k)}) H^{d}\big)/\lambda_{\min}(W^{(k)} W^{(k),T})$ and therefore $\|N^{(k)}\|_2 \leq C H^{-d}$. Summarizing we have obtained (5).
(6), (7), (8) and (11) follow from   Theorem \ref{thmuuhiuhddu}  and in particular \eqref{eqkhiduhde2df7d}, \eqref{eqjgjhdjghgfhgdy} and \eqref{eqkhiduhdf7d}. See \eqref{eqcond1} and the proof of Theorem \ref{thmodhehiudhehd} for (9) and (10).
\end{proof}
We will also need the following lemma.
\begin{Lemma}\label{lemfyfyfyvh}
Let $k\in \{1,\ldots,q-1\}$ and let $R$ be the $\I^{(k)}\times \I^{(k+1)}$ matrix defined by and $R_{i,j}=0$ for $j \in i^{\rho_k,k+1}$ and $R_{i,j}=R_{i,j}^{(k,k+1)}$ for $j \in \I^{(k+1)}/i^{\rho_k,k+1}$. It holds true that $\|R\|_2 \leq  C H^{d/2}  e^{- \rho_k/C }$.
\end{Lemma}
\begin{proof}
Observe that $\|R\|_2^2 \leq |\I^{(k)}| \max_{i\in \I^{(k)}}\sum_{j\in \I^{(k+1)}/i^{\rho_k,k+1}} |R_{i,j}^{(k,k+1)}|^2$ with $|\I^{(k)}|\leq C H_k^{-d}$.
Let $i\in \I^{(k)}$. Using Theorem \ref{eqhjgjhgjgjg} and Cauchy-Schwartz inequality we have $|R_{i,j}^{(k,k+1)}|\leq  \|\psi^{(k)}_i\|_{L^2(\tau_j^{(k+1)})} \|\phi_j^{(k+1)}\|_{L^2(\tau_j^{(k+1)})}$. Therefore
 $\sum_{j\in \I^{(k+1)}/i^{\rho_k,k+1}} |R_{i,j}^{(k,k+1)}|^2 \leq C H_{k+1}^d \sum_{j\in \I^{(k+1)}/i^{\rho_k,k+1}} \|\psi^{(k)}_i\|_{L^2(\tau_j^{(k+1)})}^2$.
 Observe that $\sum_{j\in \I^{(k+1)}/i^{\rho_k,k+1}} \|\psi^{(k)}_i\|_{L^2(\tau_j^{(k)})}^2=\sum_{s\in \I^{(k)}/i^{\rho_k}} \|\psi^{(k)}_i\|_{L^2(\tau_s^{(k)})}^2$. Since $\int_{\tau_s^{(k)}}\psi^{(k)}_i =0$ for $s\not=i$ we obtain
 from  Poincar\'{e}'s inequality that $\|\psi^{(k)}_i\|_{L^2(\tau_s)}\leq C \|\nabla \psi^{(k)}_i\|_{L^2(\tau_s^{(k)})} H_k $. Therefore
 $\sum_{j\in \I^{(k+1)}/i^{\rho_k,k+1}} |R_{i,j}^{(k,k+1)}|^2 \leq C H_{k+1}^d H_k^2 \sum_{s\in \I^{(k)}/i^{\rho_k}}  \|\nabla \psi^{(k)}_i\|_{L^2(\tau_s^{(k)})}^2$.
Using Theorem \ref{thm:expdecay} we obtain that\\ $\sum_{s\in \I^{(k)}/i^{\rho_k}}  \|\nabla \psi^{(k)}_i\|_{L^2(\tau_s^{(k)})}^2 \leq C e^{-C^{-1} \rho_k} \|\psi_i^{(k)}\|_a^2$. Using \eqref{eqgguguyg6} we have $\|\psi_i^{(k)}\|_a^2\leq C H_k^{-d-2}$, therefore
$\sum_{j\in \I^{(k+1)}/i^{\rho_k,k+1}} |R_{i,j}^{(k,k+1)}|^2 \leq  C H^d  e^{-C^{-1} \rho_k}$.
\end{proof}
Let us now prove Theorem \ref{thmerrorpropagation}.
We obtain by induction (using the constraints in \eqref{eqpsiikloc}) that for $k\in \{1,\ldots,q\}$ and $i\in \I^{(k)}$, $\psi_i^{(k),\loc}$ satisfies the constraints of \eqref{eq:dfddeytfewdaisq}. Moreover  \eqref{eq:hdkhdkjh33e} implies that if $\psi$ satisfies the constraints of \eqref{eq:dfddeytfewdaisq} then
$\|\psi\|_a^2= \|\psi_i^{(k)}\|_a^2+\|\psi- \psi_i^{(k)}\|_a^2$.
Therefore, for $k\in \{2,\ldots,q-1\}$,  $\psi_i^{(k),\loc}$ is also the minimizer of $\|\psi- \psi_i^{(k)}\|_a$ over functions $\psi$ of the form $\psi=\sum_{j \in i^{\rho_k,k+1}} c_j \psi_j^{(k+1),\loc}$ satisfying the constraints of \eqref{eqpsiikloc}. Thus, writing $\psi^*:= \sum_{j \in i^{\rho_k,k+1}} R_{i,j}^{(k,k+1)} \psi_j^{(k+1),\loc} $,  we have (since $\psi^*$ satisfies the constraints of \eqref{eqpsiikloc})
$\| \psi_i^{(k),\loc}-\psi_i^{(k)}\|_a\leq \| \psi^*-\psi_i^{(k)}\|_a$. Write
  $\psi_1:= \sum_{j \in \I^{(k+1)}} R_{i,j}^{(k,k+1)} \psi_j^{(k+1),\loc} $ and  $\psi_2:=\sum_{j \in \I^{(k+1)}/i^{\rho_k,k+1}} R_{i,j}^{(k,k+1)} \psi_j^{(k+1),\loc}$.
Observing that $\psi^*=\psi_1-\psi_2$ we deduce that $\| \psi_i^{(k),\loc}-\psi_i^{(k)}\|_a^2 \leq 2 \| \psi_1-\psi_i^{(k)}\|_a^2 + 2 \| \psi_2\|_a^2$ and after summing over $i$,
$\big(\er(k)\big)^2\leq 2(I_1+I_2)$ with $I_1= \sum_{i \in  \I^{(k)}} \|\sum_{j \in  \I^{(k+1)}} R_{i,j}^{(k,k+1)} (\psi_j^{(k+1)}-\psi_j^{(k+1),\loc})\|_a^2$
and $I_2=\sum_{i \in  \I^{(k)}}  \|\sum_{j \in \I^{(k+1)}/i^{\rho_k,k+1}} R_{i,j}^{(k,k+1)} \psi_j^{(k+1),\loc}\|_a^2$.
Writing $S$ the $\I^{(k+1)}\times \I^{(k+1)}$ symmetric positive matrix with entries $S_{i,j}=\<\psi_i^{(k+1)}-\psi_i^{(k+1),\loc},\psi_j^{(k+1)}-\psi_j^{(k+1),\loc}\>_a$, note that $I_1=\Tr[R^{(k,k+1)}S R^{(k+1,k)}]$. Writing $S^\frac{1}{2}$ the matrix square root of $S$, observe that for a matrix $U$, using the cyclic property of the trace, $\Tr[U S U^T]=\Tr[S^\frac{1}{2} U^T U S^\frac{1}{2}]\leq \lambda_{\max}(U^T U) \Tr[S]$, which (observing that $\Tr[S]=(\er(k+1))^2$ and $\lambda_{\max}(U^T U)=\|U\|_2^2$) implies
$I_1 \leq \|R^{(k,k+1)}\|_2^2 \big(\er(k+1)\big)^2$. Therefore (using Lemma \ref{lembase}) we have  $\sqrt{I_1} \leq C H^{\frac{d}{2}} \er(k+1)$.
Let us now bound $I_2$. Let $R$ be defined as in Lemma \ref{lemfyfyfyvh}. Noting that $\<\psi_i^{(k+1),\loc},\psi_j^{(k+1),\loc}\>_a=A^{(k+1),\loc}_{i,j}$ we have (as above)
$I_2=\Tr[R A^{(k+1),\loc} R^T]\leq \lambda_{\max}(R^T R) \Tr[A^{(k+1),\loc}]$. Summarizing and using Lemma \ref{lemfyfyfyvh} we deduce that
$\er(k)\leq C H^{\frac{d}{2}} \er(k+1)+ C H^{\frac{d}{2}} e^{-\rho_{k}/C}  \sqrt{\Tr[A^{(k+1),\loc}]}$. Observing that
$\sqrt{\Tr[A^{(k+1),\loc}]}\leq \er(k+1)+\sqrt{\Tr[A^{(k+1)}]}$ and using $\Tr[A^{(k+1)}]\leq C H_{k+1}^{-d} \max_{i\in \I^{(k+1)}}\|\psi_i^{(k+1)}\|_{a}^2 $
and (Lemma \ref{lem:dihidue23}) $\|\psi_i^{(k+1)}\|_a \leq  C H_{k+1}^{-\frac{d}{2}-1}$,
we conclude the proof of the theorem.
\end{proof}
Let  $u^{(1),\loc}$ be the finite element solution of \eqref{eqn:scalar} in $\V^{(1),\loc}:=\Span\{\psi_j^{(k),\loc}\mid j\in \I^{(1)}\}$.
For $k\in \{2,\ldots,q\}$, let $W^{(k)}$ be defined as in Construction \ref{const1} or Construction \ref{const2}.
For $i\in \J^{(k)}$, let $\chi^{(k),\loc}_i:=\sum_{j \in \I^{(k)}} W_{i,j}^{(k)} \psi_j^{(k),\loc}$.
For $k\in \{2,\ldots,q\}$ let $u^{(k),\loc}-u^{(k-1),\loc}$ be the finite element solution of \eqref{eqn:scalar} in $\W^{(k),\loc}:=\Span\{\chi_j^{(k),\loc}\mid j\in \J^{(k)}\}$.
For $k\in \{2,\ldots,q\}$, write $u^{(k),\loc}:=u^{(1),\loc}+\sum_{j=2}^k (u^{(j),\loc}-u^{(j-1),\loc})$.
Let $B^{(k),\loc}$ be the $\J^{(k)}\times \J^{(k)}$ matrix defined by $B^{(k),\loc}_{i,j}:=\<\chi^{(k),\loc}_i,\chi^{(k),\loc}_j\>_a$.
Observe that $B^{(k),\loc}= W^{(k)}A^{(k),\loc}W^{(k),T}$.
Write for $k\in \{2,\ldots,q\}$, $\er(k,\chi):=\big(\sum_{j\in \J^{(k)}} \|\chi_j^{(k)}-\chi_j^{(k),\loc}\|_a^2\big)^\frac{1}{2}$. The following theorem allows
us to control the effect of the localization error on the approximation of the solution of \eqref{eqn:scalar}.
\begin{Theorem}\label{thmdhdjh3}
It holds true that for $k\in \{2,\ldots,q\}$ (1) $\er(k,\chi) \leq C H^{-d/2}\er(k)$. Furthermore for $k\in \{2,\ldots,q\}$ and $\er(k,\chi)\leq C^{-1} H^{-k (1+d/2)+1+d}$  we have\\ (2)  $\Cond(B^{(k),\loc})\leq C H^{-2-2d}$, and (3) $\|u^{(k)}-u^{(k-1)}-(u^{(k),\loc}-u^{(k-1),\loc})\|_a \leq$\\$ C \er(k,\chi) \|g\|_{H^{-1}(\Omega)}
 H^{k (1+d/2)-3d-3}$. Similarly for $\er(1)\leq C^{-1}H^{-d/2} $,  we have\\ (4)   $\Cond(A^{(1),\loc})\leq C H^{-2}$, and (5)
$\|u^{(1)}-u^{(1),\loc}\|_a \leq C \er(1) \|g\|_{H^{-1}(\Omega)} H^{-2+d/2}$.
\end{Theorem}
\begin{proof}
We will  need the following lemma.
\begin{Lemma}\label{lemshgjhgdhg3e}
Let $\chi_1,\ldots,\chi_m$ be linearly independent elements of $H^1_0(\Omega)$. Let $\chi_1',\ldots,\chi_m'$ be another set of linearly independent elements of $H^1_0(\Omega)$. Write $\er:=\big(\sum_{i=1}^m \|\chi_i-\chi_i'\|_a^2\big)^\frac{1}{2}$. Let $B$ (resp. $B'$) be the $m\times m$  matrix defined by
$B_{i,j}=\<\chi_i,\chi_j\>_a$ (resp. $B_{i,j}'=\<\chi_i',\chi_j'\>_a$). Let $u_m$ (resp. $u_m'$) be the solution of \eqref{eqn:scalar} in $\Span\{\chi_i\mid i=1,\ldots,m\}$ (resp. $\Span\{\chi_i'\mid i=1,\ldots,m\})$. It holds true that for  $\er \leq \sqrt{\lambda_{\min}(B)} /2$ (1)  $\Cond(B')\leq 8 \Cond(B)$  (2)
$\|B-B'\|_2 \leq 3 \sqrt{\lambda_{\max}(B)} \er$ (3) $\|B^{-1}-(B')^{-1}\|_2 \leq 12 \sqrt{\lambda_{\max}(B)}  \big(\lambda_{\min}(B)\big)^{-2} \er$  and (4)
$\|u_m-u_m'\|_a\leq C \er \|g\|_{H^{-1}(\Omega)} \frac{\Cond(B)}{\sqrt{\lambda_{\min}(B)}}$.
\end{Lemma}
\begin{proof}
For (1) observe that $\sqrt{\lambda_{\max}(B')}=\sup_{|x|=1}\|\sum_{i=1}^m x_i \chi_i'\|_a\leq \sqrt{\lambda_{\max}(B)}+ \er$ and $\sqrt{\lambda_{\min}(B')}=\inf_{|x|=1}\|\sum_{i=1}^m x_i \chi_i'\|_a\geq \sqrt{\lambda_{\min}(B)}
- \er$. For (2) observe that for $x,y\in \R^m$ with $|x|=|y|=1$ we have $y^T(B-B')x=\<\sum_{i=1}^m y_i (\chi_i-\chi_i'),\sum_{i=1}^m x_i \chi_i\>_a-\<\sum_{i=1}^m y_i \chi_i',\sum_{i=1}^m x_i (\chi_i'-\chi_i)\>_a\leq (\sqrt{\lambda_{\max}(B')}+\sqrt{\lambda_{\max}(B)})\er$.
(3) follows from (2) and $\|B^{-1}-(B')^{-1}\|_2 \leq \|B-B'\|_2/\big(\lambda_{\min}(B) \lambda_{\min}(B')\big)$.
For (4) observe that $u_m=\sum_{i=1}^m w_i\chi_i$ (resp. $u_m'=\sum_{i=1}^m w_i'\chi_i'$) where $w=B^{-1} b$ with $b_i=\int_{\Omega} g \chi_i$ (resp. $w'=(B')^{-1} b'$ with $b_i'=\int_{\Omega} g \chi_i'$). Therefore $\|u_m-u_m'\|_a\leq |w| \er+|w-w'|\sqrt{\lambda_{\max}(B)}$. $w-w'=B^{-1}(b-b')-B^{-1}(B-B')w'$ leads to
$|w-w'|\leq C (\|g\|_{H^{-1}(\Omega)}\er +\|B-B'\|_2 |w'|)/\lambda_{\min}(B)$. Using (2),
  $\lambda_{\min}(B) |w|^2\leq \|\sum_{i=1}^m w_i\chi_i\|_a^2 \leq \|u\|_a^2 \leq C \|g\|_{H^{-1}(\Omega)}^2$, and $\lambda_{\min}(B') |w'|^2\leq C \|g\|_{H^{-1}(\Omega)}^2$ we conclude the proof of (4) after simplification.
\end{proof}
Let us now prove Theorem \ref{thmdhdjh3}.
Using $\chi_j^{(k)}-\chi_j^{(k),\loc}=\sum_{i\in \I^{(k)}} W^{(k)}_{j,i} (\psi_i^{(k)}-\psi_i^{(k),\loc})$ and noting that $W^{(k)}_{j,i}=0$ for $i^{(k-1)}\not=j^{(k-1)}$ we have
$\big(\er(k,\chi)\big)^2 \leq \sum_{j\in \J^{(k)}}$\\ $\big(\sum_{i\in \I^{(k)}} (W^{(k)}_{j,i})^2 \sum_{i\in \I^{(k)}, i^{(k-1)}=j^{(k-1)}} \|\psi_i^{(k)}-\psi_i^{(k),\loc}\|_a^2\big)$. Therefore,
$\big(\er(k,\chi)\big)^2 \leq$\\ $\big(\er(k)\big)^2 \max_{i\in \I^{(k-1)}}m_i \max_{j\in \J^{(k)}} \sum_{i\in \I^{(k)}} (W^{(k)}_{j,i})^2  $.
Observing that (see \eqref{eqddkhjji})\\ $\max_{i\in \I^{(k-1)}}m_i \leq 1/(H \delta)^d$ and $\sum_{i\in \I^{(k)}} (W^{(k)}_{j,i})^2 \leq \lambda_{\max}(W^{(k)}W^{(k),T})\leq 3$ (see Lemma \ref{lembase}) we conclude that (1) holds true with $C=(3/\delta^d)^\frac{1}{2}$. (2) and (3) are a direct application of lemmas \ref{lemshgjhgdhg3e} and \ref{lembase}. For (3), observe that
$u^{(k)}-u^{(k-1)}$ (resp. $u^{(k),\loc}-u^{(k-1),\loc}$) is the finite element solution of \eqref{eqn:scalar} in $\W^{(k)}$ (resp. $\W^{(k),\loc}:=\Span\{\chi_j^{(k),\loc}\mid j\in \J^{(k)}\}$). The proof of (4) and (5) is similar to that of (2) and (3).
\end{proof}

\begin{Theorem}\label{thmdjjuud}
Let $k\in \{1,\ldots,q\}$. We have, $\er(k)\leq C \sum_{j=k}^q e^{-\rho_{j}/C} C^{j-k}H^{-\frac{d}{2}+k\frac{d}{2}-j3\frac{d}{2}}$.
\end{Theorem}
\begin{proof}
By Theorem \ref{thmerrorpropagation}, for $k\in \{1,\ldots,q-1\}$, we have
 $\er(k)\leq a_k+b_k \er(k+1)$ with $a_k=C  e^{-\rho_{k}/C}  H^{\frac{d}{2}-(k+1)(d+1)}$ and $b_k=C H^{-\frac{d}{2}}$. Therefore we obtain by induction that
$\er(k) \leq a_k + b_k a_{k+1}+ b_k b_{k+1} a_{k+2}+\cdots + b_k \cdots b_{q-2}a_{q-1}+ b_k \cdots b_{q-1} \er(q)$.
Using Theorem \ref{thm:hieuhdds} we have $\er(q)\leq C H^{-d/2-q(2+d/2)}e^{-\rho_q/C}$ and obtain the result after simplification.
\end{proof}

\begin{Theorem}\label{tmshjgeydg}
Let $\epsilon \in (0,1)$. It holds true that if $\rho_k\geq C \big((1+\frac{1}{\ln(1/H)})\ln \frac{1}{H^k}+\ln \frac{1}{\epsilon}\big)$ for $k\in \{1,\ldots,q\}$ then (1) for $k\in \{1,\ldots,q\}$ we have
$\|u^{(k)} - u^{(k),\loc}\|_a \leq   \epsilon \|g\|_{H^{-1}(\Omega)}$\\ and $\|u - u^{(k),\loc}\|_a \leq   C (H^k+\epsilon) \|g\|_{L^2(\Omega)}$ (2) $\Cond(A^{(1),\loc})\leq C H^{-2}$, and for $k\in \{2,\ldots,q\}$ we have (3)
 $\Cond(B^{(k),\loc})\leq C H^{-2-2d}$ and (4) $\|u^{(k)}-u^{(k-1)}-(u^{(k),\loc}-u^{(k-1),\loc})\|_a \leq  \frac{\epsilon}{2 k^2} \|g\|_{H^{-1}(\Omega)}$.
\end{Theorem}
\begin{proof}
Theorems \ref{thmgugyug0} and \ref{thmdhdjh3} imply that the results of Theorem \ref{tmshjgeydg} hold true if for $k\in \{1,\ldots,q\}$
$\er(k)  \leq C^{-1}  H^{-k (1+d/2)+7d/2+3} \epsilon/k^2$. Using Theorem \ref{thmdjjuud} we deduce that the results of Theorem \ref{tmshjgeydg} hold true if for $k\in \{1,\ldots,q\}$ and $k\leq j \leq q$ we have $C  e^{-\rho_{j}/C} C^{j-k}H^{-\frac{d}{2}+k\frac{d}{2}-j3\frac{d}{2}} \leq  H^{-k (1+d/2)+7d/2+3} \epsilon/(k^2 j^2)$. We conclude after simplification.
\end{proof}

\section{The algorithm, its implementation and complexity}\label{secnumimple}

\subsection{The initialisation of the algorithm}\label{subsechierarlocnested}

To describe the practical implementation of the algorithm we consider the (finite-element) discretized version of \eqref{eqn:scalar}. Let $\T_h$ be a regular fine mesh  discretization of $\Omega$ of resolution $h$ with $0<h \ll 1$. Let $\N$ be the set of interior nodes $z_i$ and $N=|\mathcal{N}|$ be the number of interior nodes ($N= \mathcal{O}(h^{-d})$) of $\T_h$. Write $(\varphi_i)_{i\in \N}$ a set of regular nodal basis elements (of $H^1_0(\Omega)$) constructed from $\T_h$ such that
for each $i\in \N$,  $\supp(\varphi_i)\subset B(z_i, C_0 h)$ and for $y\in \R^N$,
\begin{equation}\label{eqhhgfff65f}
\ubar{\gamma} h^d |y|^2  \leq \|\sum_{i\in \N} y_i \varphi_i \|_{L^2(\Omega)}^2 \leq \bar{\gamma} h^d |y|^2
\end{equation}
for some constants $\ubar{\gamma}, \bar{\gamma}, C_0\approx \mathcal{O}(1)$. In addition to \eqref{eqhhgfff65f} the regularity of the finite elements is used to ensure the availability of the inverse Poincar\'{e} inequality
\begin{equation}\label{eqinvpoincdiscrete}
\|\nabla v\|_{L^2(\Omega)}\leq C_1 h^{-1} \| v\|_{L^2(\Omega)}
\end{equation}
 for $v\in \Span\{\varphi_i\mid i\in \N\}$ and some constant $C_1 \approx \mathcal{O}(1)$, used to generalize the proof of Theorem \ref{thmuuhiuhddu} to the discrete case.

Given $g=\sum_{i\in \N} g_i \varphi_i$ we want to find $u\in \Span\{\varphi_i \mid i\in \N\}$ such that for all $j\in \N$,
\begin{equation}\label{eqhuiuhiuhuiu}
 \<\varphi_j,u\>_a=\int_{\Omega} \varphi_j g  \text{ for all } j\in \N
\end{equation}
In practical applications
 $a$ is naturally assumed to be piecewise constant over the fine mesh (e.g. of constant value in each triangle or square of $\T_h$) and one purpose of the algorithm is the fast resolution of the linear system \eqref{eqhuiuhiuhuiu} up to accuracy $\epsilon \in (0,1)$.

 \begin{figure}[h!]
	\begin{center}
			\includegraphics[width=\textwidth]{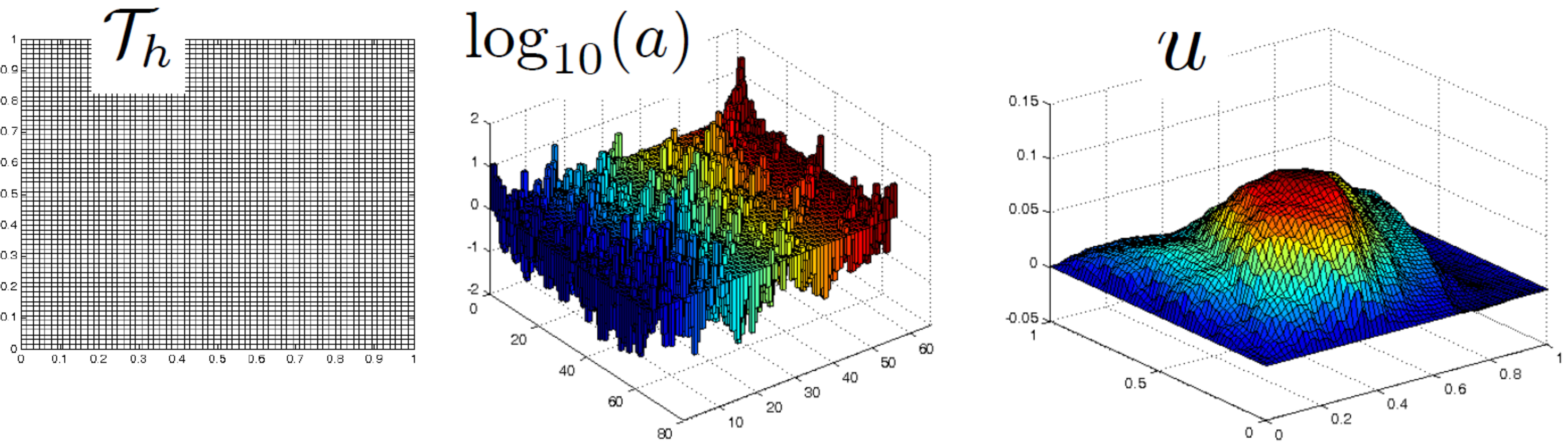}
		\caption{The (fine) mesh $\T_h$, $a$ (in $\log_{10}$ scale) and $u$.}\label{fig:tau}
	\end{center}
\end{figure}
\begin{Example}\label{ex1}
We will illustrate the presentation of the algorithm with a numerical example in which  $\T_h$ is a square grid of mesh size $h=(1+2^{q})^{-1}$ with $q=6$ and $64\times 64$ interior nodes (Figure \ref{fig:tau}). $a$ is piecewise constant on each square of $\T_h$ and given by
$a(x)=\prod_{k=1}^6 \Big(1+0.5 \cos\big(2^k \pi (\frac{i}{2^q+1}+\frac{j}{2^q+1})\big)\Big) \Big(1+0.5 \sin\big(2^k \pi (\frac{j}{2^q+1}-3\frac{i}{2^q+1})\big)\Big)$
for $x\in [\frac{i}{2^q+1},\frac{i+1}{2^q+1})\times [\frac{j}{2^q+1},\frac{j+1}{2^q+1})$. The contrast of $a$ (i.e., when $a$ is scalar, the ratio between its maximum and minimum value)  is  $1866$. The finite-element discretization  \eqref{eqhuiuhiuhuiu} is  obtained using
continuous nodal basis elements $\varphi_i$ spanned by $\{1,x_1,x_2,x_1 x_2\}$ in each square of $\T_h$. Writing $z_i$ the positions of the interior nodes of $\T_h$, we choose, for our numerical example, $g(x)=\sum_{i\in \N} \big(\cos(3z_{i,1}+z_{i,2})+\sin(3z_{i,2})+\sin(7z_{i,1}-5z_{i,2})\big) \varphi_i(x)$.
\end{Example}
 \begin{figure}[h!]
	\begin{center}
			\includegraphics[width=\textwidth]{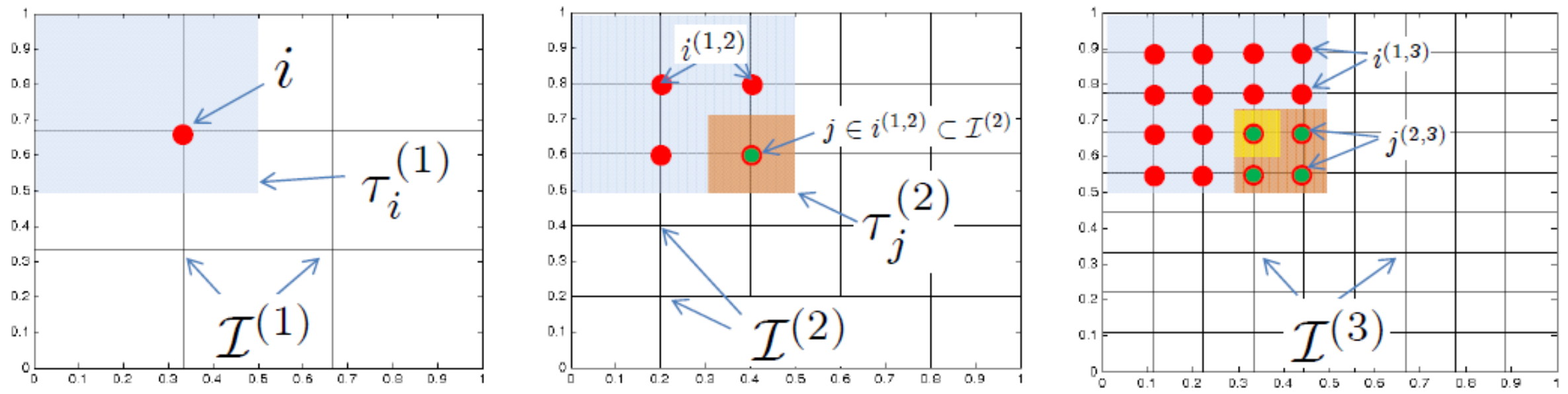}
		\caption{$\I^{(1)}$, $\I^{(2)}$ and $\I^{(3)}$.}\label{fig:Pi}
	\end{center}
\end{figure}
The first step of the proposed algorithm is the construction of the index tree $\I$ of Definition \ref{defmulires} describing the domain decomposition of Definition \ref{defmulires}. To ensure a uniform bound on the condition numbers of the stiffness matrices  \eqref{eqjgfytfjhyyyg} one must select the resolutions $H_k$ to form a geometric sequence (or simply such that $H_{k-1}/H_k$ is uniformly bounded), i.e. $H_k=H^k$ for some $H\in (0,1)$ {\it (for our numerical example $H=1/2$, $q=6$ and we identify $\I^{(k)}$ as the indices of the interior nodes of a square grid of  resolution $(1+2^{k})^{-1}$  as illustrated in
Figure \ref{fig:Pi})}. In this construction $H^q=h$ corresponds to the resolution of the fine mesh and each subset $\tau_i^{(q)}$ ($i\in \I^{(q)}$) contains one and only one element of $\N$ (interior node of the fine mesh). Using this one to one correspondence we use the elements of $\I=\I^{q}$ to (re)label the nodal elements $(\varphi_i)_{i\in \N}$ as $(\varphi_i)_{i\in \I}$.
  \begin{figure}[h!]
	\begin{center}
			\includegraphics[width=\textwidth]{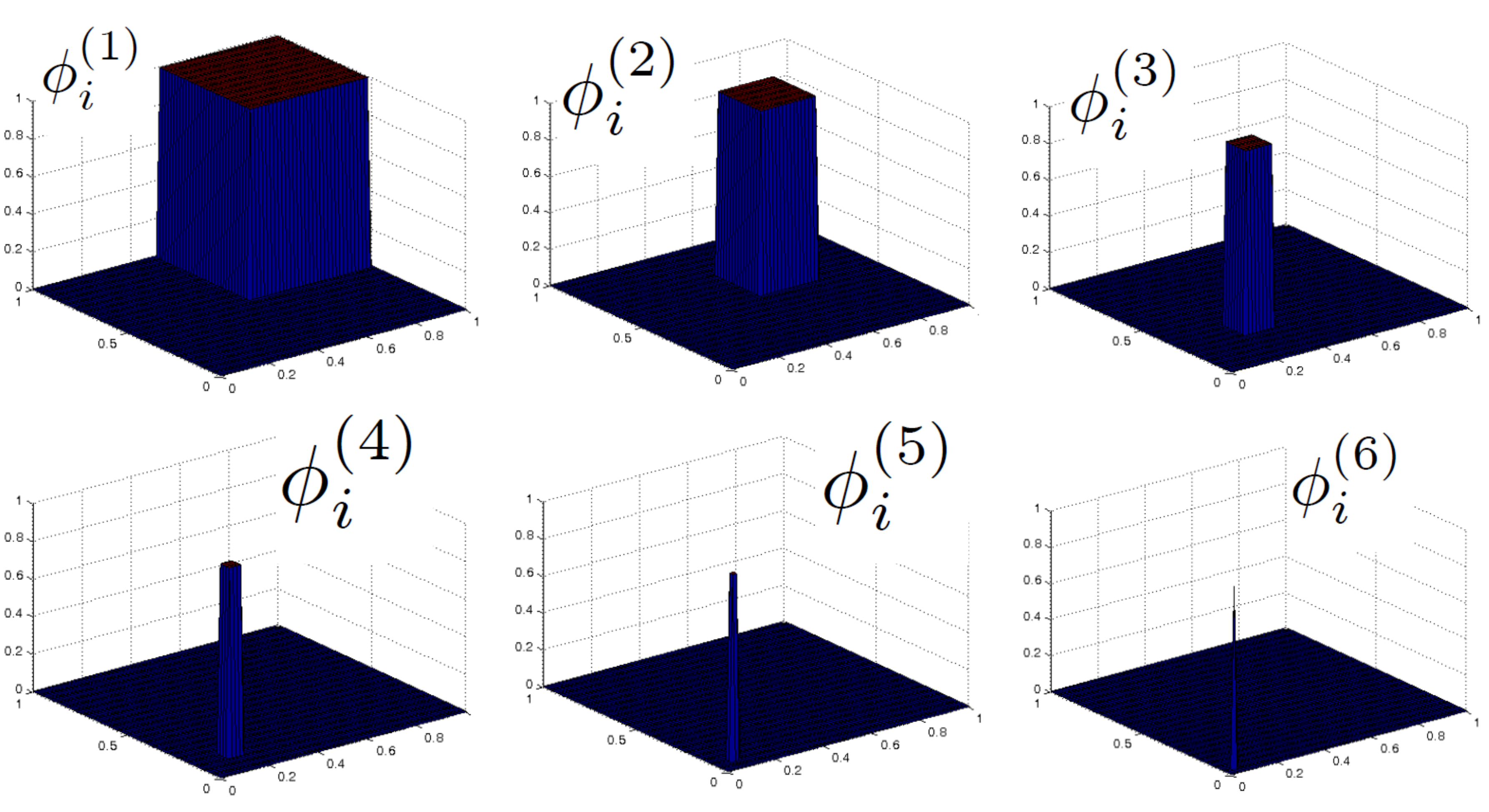}
		\caption{The functions $\phi^{k}_i$ with $k\in \{1,\ldots,q\}$ and $q=6$.}\label{fig:phi}
	\end{center}
\end{figure}
The measurement functions $(\phi_i^{(k)})_{i\in \I^{(k)}}$ are then identified  (1) by selecting $\phi_i^{(q)}=\varphi_i$ for $i\in \I^{(q)}$ and (2)
 via the nested aggregation \eqref{eq:eigdeiud3dd} of the nodal elements (as commonly done in AMG), i.e. $\phi^{(k)}_i=\sum_{j\in \I^{(k+1)}}\pi^{(k,k+1)}_{i,j}  \phi^{(k+1)}_j=\sum_{j \in i^{(k,k+1)}} \phi^{(k+1)}_j$ for $k\in \{1,\ldots,q-1\}$ and $i\in \I^{(k)}$.

\begin{Remark}
We refer to Figure \ref{fig:phi} for an illustration of these measurement functions for our numerical example.
Note that the support of each $\phi^{(k)}_i$ is only approximatively (and not exactly) $\tau^{(k)}_i$ and that the $\phi^{(k)}_i$ are only approximate set functions (and not exact ones). This does not affect the design, accuracy and localization of the  algorithm presented here because the frame inequalities \eqref{eqgam1}, and the Poincar\'e inequalities $\|\sum_{i\in \I^{(k)}}x_i \phi_i^{(k)}\|_{H^{-1}(\Omega)}\leq  C\,H^{k-1}\|\phi_i^{(k)}\|_{L^2(\Omega)}$ for $x\in \Ker(\pi^{(k-1,k)})$,  hold true. Indeed,  \eqref{eqhhgfff65f} and Construction \ref{defmulires} imply that the frame inequalities \eqref{eqgam1} with $\bar{\gamma}_k\leq \bar{\gamma}\delta^{-d}$ and $\ubar{\gamma}_k\geq \ubar{\gamma}\delta^{d}$, and the Poincar\'e inequalities are regularity/homogeneity conditions on the mesh and the aggregated elements. Although a fine mesh has been used to facilitate the presentation of the algorithm, the proposed method is meshless (it only requires the specification of the basis elements $(\varphi_i)_{i\in \I}$).
\end{Remark}

\begin{algorithm}[!ht]
\caption{Exact Gamblet transform/solve.}\label{gambletsolve}
\begin{algorithmic}[1]
\STATE\label{step1} For $i,j\in \I^{(q)}$, $M_{i,j}=\int_{\Omega} \varphi_i \varphi_j$  \COMMENT{Mass matrix}
\STATE\label{step2} For $i,j\in \I^{(q)}$, $A_{i,j}=\int_{\Omega} (\nabla \varphi_i)^T a \nabla \varphi_j$ \COMMENT{Stiffness matrix}
\STATE\label{step2a} Compute $M^{-1}$  \COMMENT{Mass matrix inversion}
\STATE\label{step3} For $i\in \I^{(q)}$, $\psi^{(q)}_i=\sum_{j \in  \I^{(q)}} M^{-1}_{i,j} \varphi_j$  \COMMENT{Level $q$ gamblets}
\STATE\label{step4} For $i\in \I^{(q)}$, $g^{(q)}_i=g_i$      \COMMENT{$g^{(q)}_i=\int_{\Omega} \psi_i^{(q)} g$ with $g=\sum_{i\in \I^{(q)}} g_i \varphi_i$}
\STATE\label{step5} For $i,j\in \I^{(q)}$, $A^{(q)}_{i,j}= \int_{\Omega} (\nabla \psi_i^{(q)})^T a \nabla \psi_j^{(q)}$   \COMMENT{$A^{(q)}=M^{-1} A M^{-1,T}$}
\FOR{$k=q$ to $2$}
\STATE\label{step7} $B^{(k)}= W^{(k)}A^{(k)}W^{(k),T}$ \COMMENT{Eq.~\eqref{eqjgfytfjhyyyg}}
\STATE\label{step8} $w^{(k)}=B^{(k),-1} W^{(k)} g^{(k)}$ \COMMENT{Eq.~\eqref{eqsdjoejddi1}}
\STATE\label{step9}  For $i\in \J^{(k)}$, $\chi^{(k)}_i=\sum_{j \in \I^{(k)}} W_{i,j}^{(k)} \psi_j^{(k)}$  \COMMENT{Eq.~\eqref{eqjkhdkdh}}
\STATE\label{step10} $u^{(k)}-u^{(k-1)}=\sum_{i\in \J^{(k)}}w^{(k)}_i \chi^{(k)}_i$ \COMMENT{Thm.~\ref{thddwedmgugyug}}
\STATE\label{step11}  $ D^{(k,k-1)}= -B^{(k),-1}W^{(k)}A^{(k)}\bar{\pi}^{(k,k-1)}$ \COMMENT{Eq.~\eqref{eqdkdjhdkhse}}
\STATE\label{step12} $R^{(k-1,k)}=\bar{\pi}^{(k-1,k)}+D^{(k-1,k)}W^{(k)}$ \COMMENT{Eq.~\eqref{eqhuhiddeuv}}
\STATE\label{step13} $A^{(k-1)}= R^{(k-1,k)}A^{(k)}R^{(k,k-1)}$ \COMMENT{Eq.~\eqref{eqhuhiuv}}
\STATE\label{step14} For $i\in \I^{(k-1)}$, $\psi^{(k-1)}_i=\sum_{j \in  \I^{(k)}} R_{i,j}^{(k-1,k)} \psi_j^{(k)}$ \COMMENT{Eq.~\eqref{eq:ftfytftfx}}
\STATE\label{step15} $g^{(k-1)}=R^{(k-1,k)} g^{(k)}$ \COMMENT{Eq.~\eqref{eqyguugy6t}}
\ENDFOR
\STATE\label{step16} $ U^{(1)}=A^{(1),-1}g^{(1)}$ \COMMENT{Eq.~\eqref{eqsdjoejddi2}}
\STATE\label{step17} $u^{(1)}=\sum_{i \in \I^{(1)}} U^{(1)}_i \psi^{(1)}_i$ \COMMENT{Thm.~\ref{thddwedmgugyug}}
\STATE\label{step18} $u=u^{(1)}+(u^{(2)}-u^{(1)})+\cdots+(u^{(q)}-u^{(q-1)})$    \COMMENT{Thm.~\ref{thmgugyug2} with $u=u^{(q)}$}
\end{algorithmic}
\end{algorithm}
\subsection{Exact gamblet transform and multiresolution operator inversion}
The near-linear complexity of the proposed multi-resolution algorithm  (Algorithm \ref{fastgambletsolve}) is based on three properties (i) nesting (ii) uniformly bounded condition numbers (iii) localization/truncation based on exponential decay. Truncation/localization levels/subsets are, a priori, functions of the desired level of accuracy $\epsilon \in (0,1)$ in approximating the solution of \eqref{eqhuiuhiuhuiu} and to distinguish the implementation of localization/truncation (and its consequences)
we will first describe this algorithm in its \emph{zero approximation error version} (i.e. $\epsilon=0$ and without using localization/truncation, Algorithm \ref{gambletsolve}). Although this \emph{error-free} version (Algorithm \ref{gambletsolve}) performs the decomposition of the resolution of the linear system \eqref{eqhuiuhiuhuiu} (whose condition number is of the order of $h^{-d-2}\gg 1$) into the resolutions of a nesting of linear systems with uniformly bounded condition numbers, it is not of near linear complexity due to the presence of dense matrices.  Algorithm \ref{fastgambletsolve} achieves near-linear complexity by truncating/localizing the dense matrices appearing in Algorithm \ref{gambletsolve} ($\epsilon$-accuracy is ensured using the off-diagonal exponential decay of these dense matrices).
Let us now describe Algorithm \ref{gambletsolve} in detail. Lines \ref{step1} and \ref{step2} correspond to the computation of the (sparse) mass and stiffness matrices of \eqref{eqhuiuhiuhuiu}. Line \ref{step3} corresponds to the calculation of level $q$ gamblets $\psi_i^{(q)}$ defined as the minimizer of $\|\psi\|_a$ subject to $\int_{\Omega} \psi \phi_j^{(q)}=\delta_{i,j}$ and $\psi \in \Span\{\varphi_l \mid l \in \I\}$, note that since the number of constraints is equal to the number of degrees of freedom of $\psi$, and since $\int_{\Omega} \varphi_l \phi_j^{(q)}=M_{l,j}$, level $q$ gamblets do not depend on $a$ and are obtained by inverting the mass matrix in Line \ref{step2a} (note that by \eqref{eqhhgfff65f}, the mass matrix is of $\mathcal{O}(1)$ condition number).
Although not done here, one can also initialize the algorithm (and its fast version) with $\psi_i^{(q)}=\varphi_i$ (which is equivalent to using $\sum_{j\in \I^{(q)}}M^{-1}_{i,j}\varphi_j^{(q)}$ as level $q$ measurement functions).
 Line \ref{step4} corresponds to initialization of the vector $g^{(q)}$ introduced above \eqref{eqyguugy6t}. Line \ref{step5} corresponds to the initialization of the stiffness matrix $A^{(q)}$ introduced in \eqref{eq:iwihud3de}.
The core of the algorithm is the nested computation performed (iteratively from $k=q$ down to $k=2$) in lines \ref{step7} to \ref{step15}.
Note that this nested computation takes $A^{(k)}, g^{(k)}$ and $(\psi_i^{(k)})_{i\in \I^{(k)}}$ as inputs and produces (1)
$A^{(k-1)}, g^{(k-1)}$ and $(\psi_i^{(k-1)})_{i\in \I^{(k)}}$ as outputs for the next iteration and (2) the subband $u^{(k)}-u^{(k-1)}$ of the solution and subband gamblets $(\chi_i^{(k)})_{i\in \J^{(k)}}$ (which, do not need to be explicitly computed/stored since Line \ref{step10} is equivalent to
$u^{(k)}-u^{(k-1)}=\sum_{i\in \I^{(k)}}(W^{(k),T} w^{(k)})_i \psi^{(k)}_i$). Note also that the gamblets $(\psi_i^{(k)})_{i\in \I^{(k)}}$ and $(\chi_i^{(k)})_{i\in \J^{(k)}}$  can be stored and displayed using the hierarchical structure \eqref{eq:ftfytftfx}. Through this section and the remaining part of the paper we assume that the matrices $W^{(k)}$ are obtained as in Construction \ref{const1} or \ref{const2}.
Note that the number of non-zero entries of $\pi^{(k-1,k)}$ and $W^{(k)}$ is $\mathcal{O}( |\I^{(k)}|)$ (proportional to $H^{-k}$ in our numerical example). Lines \ref{step8} corresponds to solving the well conditioned linear system $B^{(k)} w^{(k)}=W^{(k)} g^{(k)}$ and the $|\I^{(k-1)}|$ well conditioned linear systems $B^{(k)} D^{(k,k-1)}=-W^{(k)}A^{(k)}\pi^{(k,k-1)}$. Note that by Theorem \ref{thmodhehiudhehd} the matrices $B^{(k)}$ have uniformly bounded condition numbers and these linear systems can be solved efficiently using iterative methods (such as the Conjugate Gradient method recalled in Subsection \eqref{subseccg}).
$u^{(1)}$ is computed in lines \ref{step17} and \ref{step18} (recall that $A^{(1)}$ is also of uniformly bounded condition number) and the last step of the algorithm, is to obtain $u$  via simple addition of the subband/subscale solution $u^{(1)}$ and $(u^{(k)}-u^{(k-1)})_{2\leq k \leq q}$.
Observe that the operating diagram of Algorithm \ref{gambletsolve} is not  a V or W  but an inverted pyramid (or a comb).
More precisely, the basis functions $\psi_i^{(k)}$ are computed hierarchically from fine to coarse scales. Furthermore as soon as the elements $\psi_i^{(k)}$ have been computed, they can be applied (independently from the other scales) to the computation of $u^{(k)}-u^{(k-1)}$ (the projection of $u$ onto $\W^{(k)}$ corresponding to the bandwidth $[H^k,H^{k-1}]$).

 \begin{figure}[h!]
	\begin{center}
			\includegraphics[width=\textwidth]{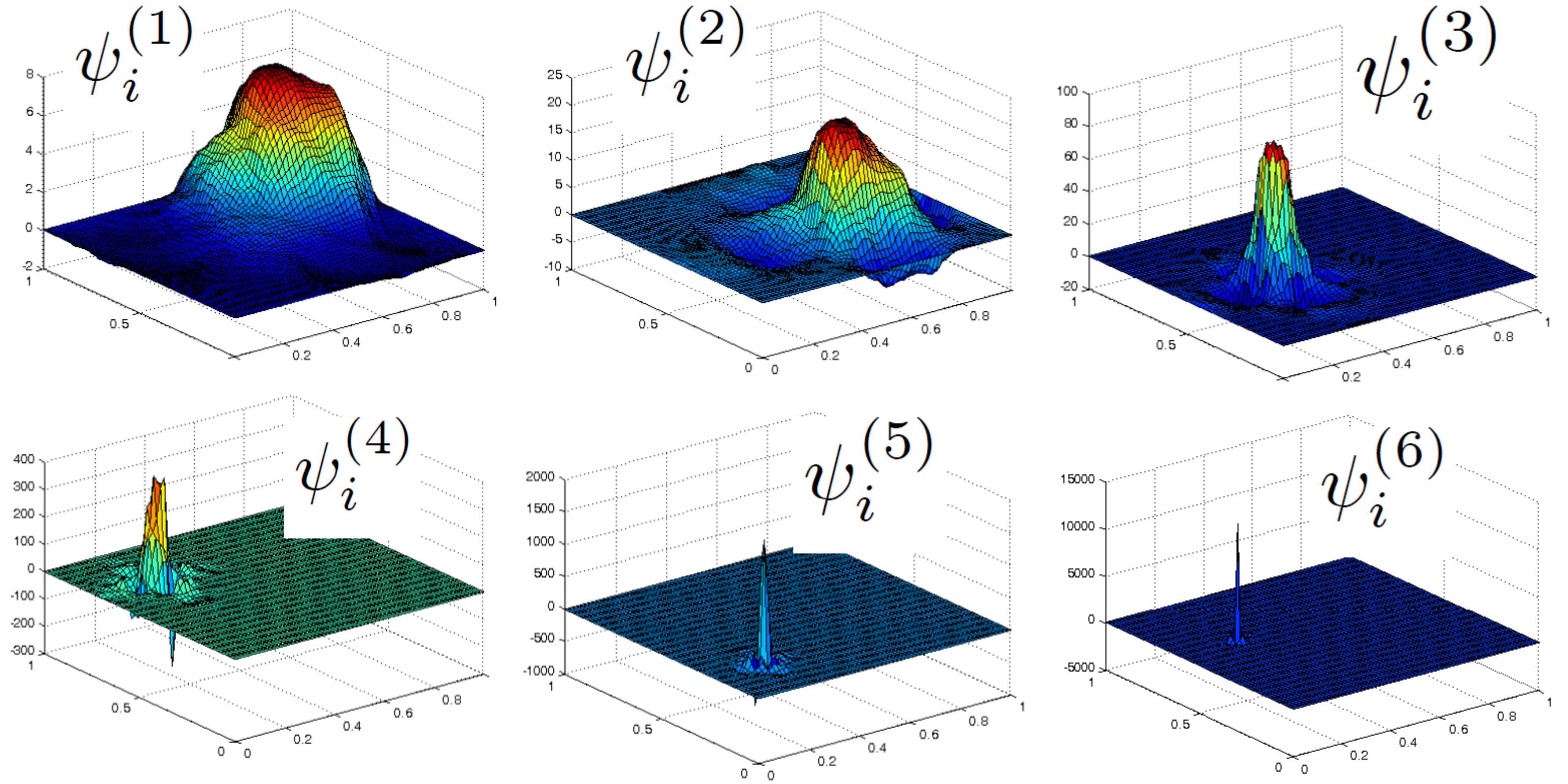}
		\caption{The basis elements $\psi^{k}_i$ with $k\in \{1,\ldots,6\}$.}\label{fig:psi}
	\end{center}
\end{figure}
 \begin{figure}[h!]
	\begin{center}
			\includegraphics[width=\textwidth]{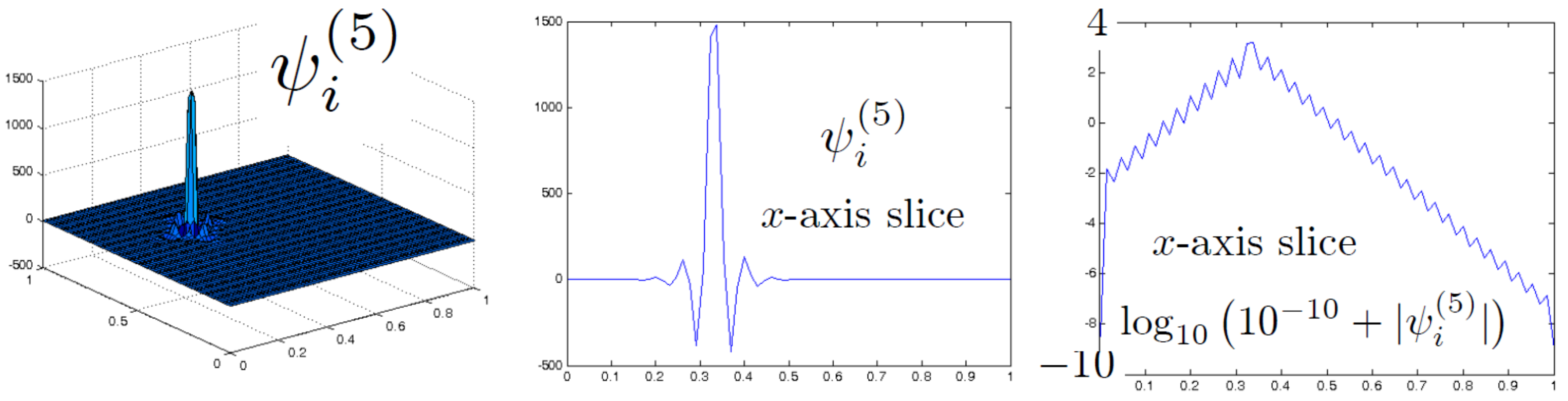}
		\caption{Exponential Decay.}\label{fig:expdecay}
	\end{center}
\end{figure}
 \begin{figure}[h!]
	\begin{center}
			\includegraphics[width=\textwidth]{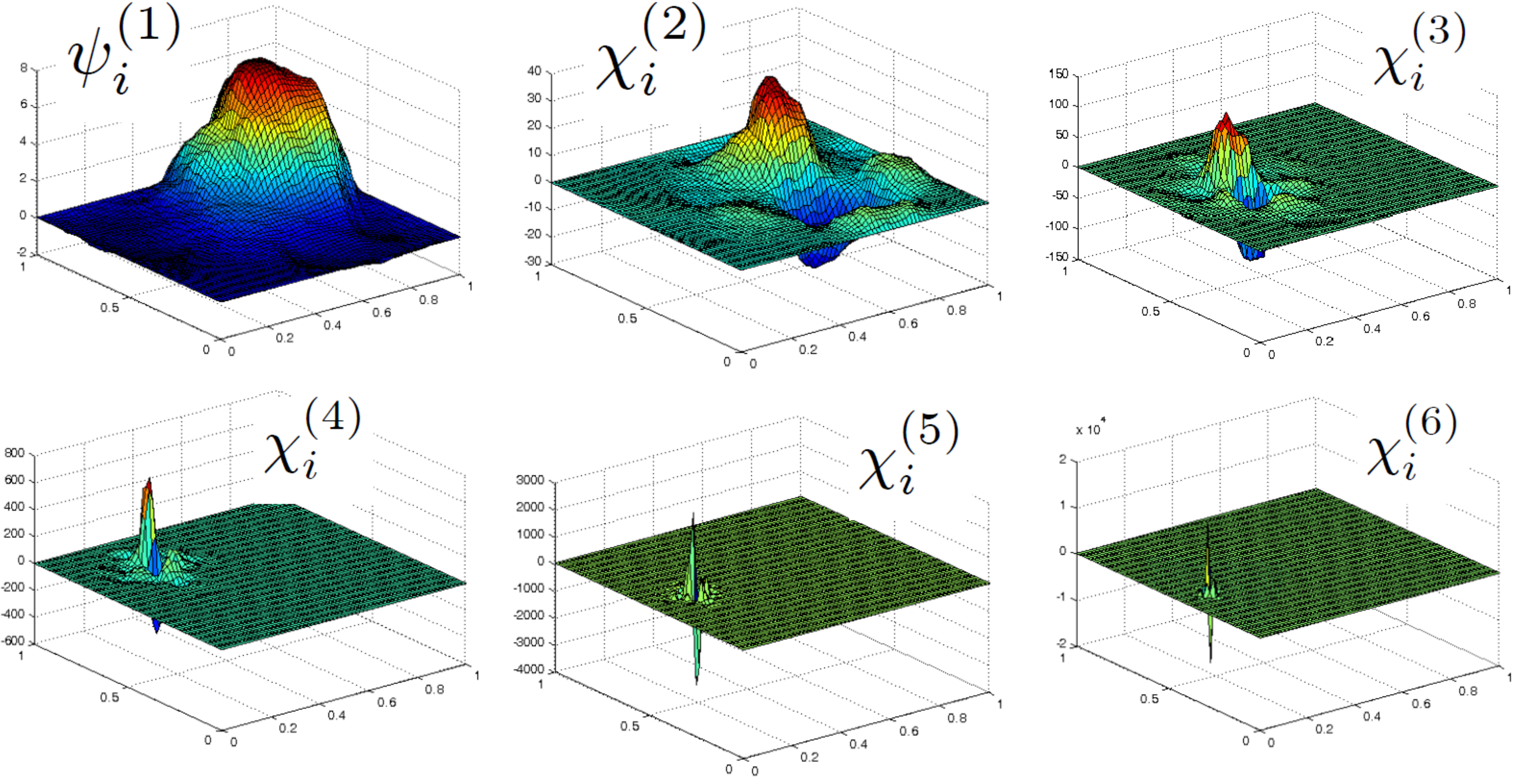}
		\caption{The basis elements $\psi^{1}_i$ and $\chi^{k}_i$ with $k\in \{2,\ldots,6\}$.}\label{fig:chi}
	\end{center}
\end{figure}
 \begin{figure}[h!]
	\begin{center}
			\includegraphics[width=\textwidth]{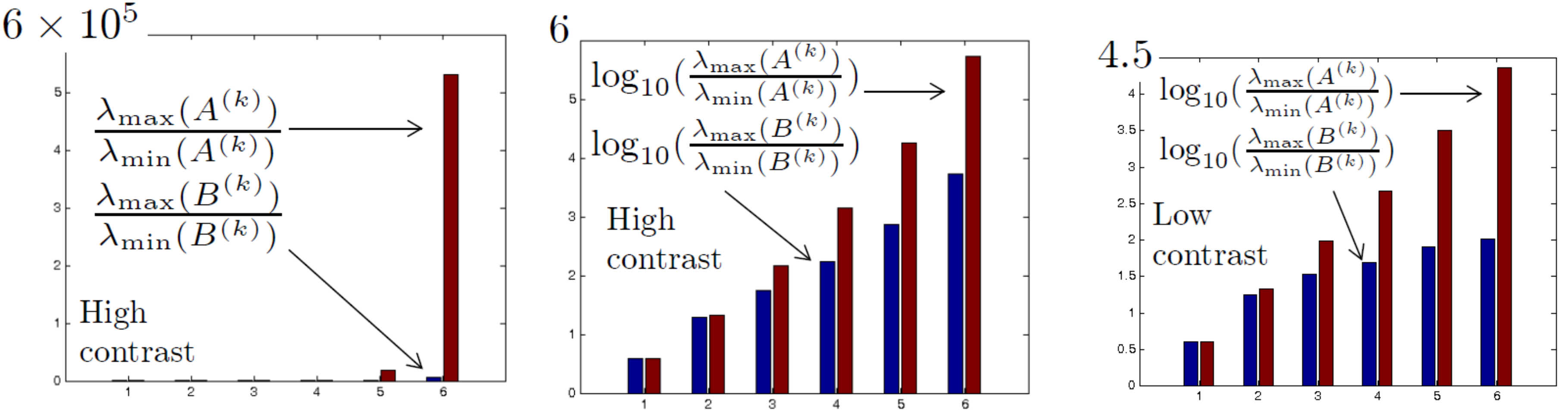}
		\caption{Condition numbers of $A^{(k)}$ and $B^{(k)}$.}\label{fig:conditionnumbers}
	\end{center}
\end{figure}
 \begin{figure}[h!]
	\begin{center}
			\includegraphics[width=\textwidth]{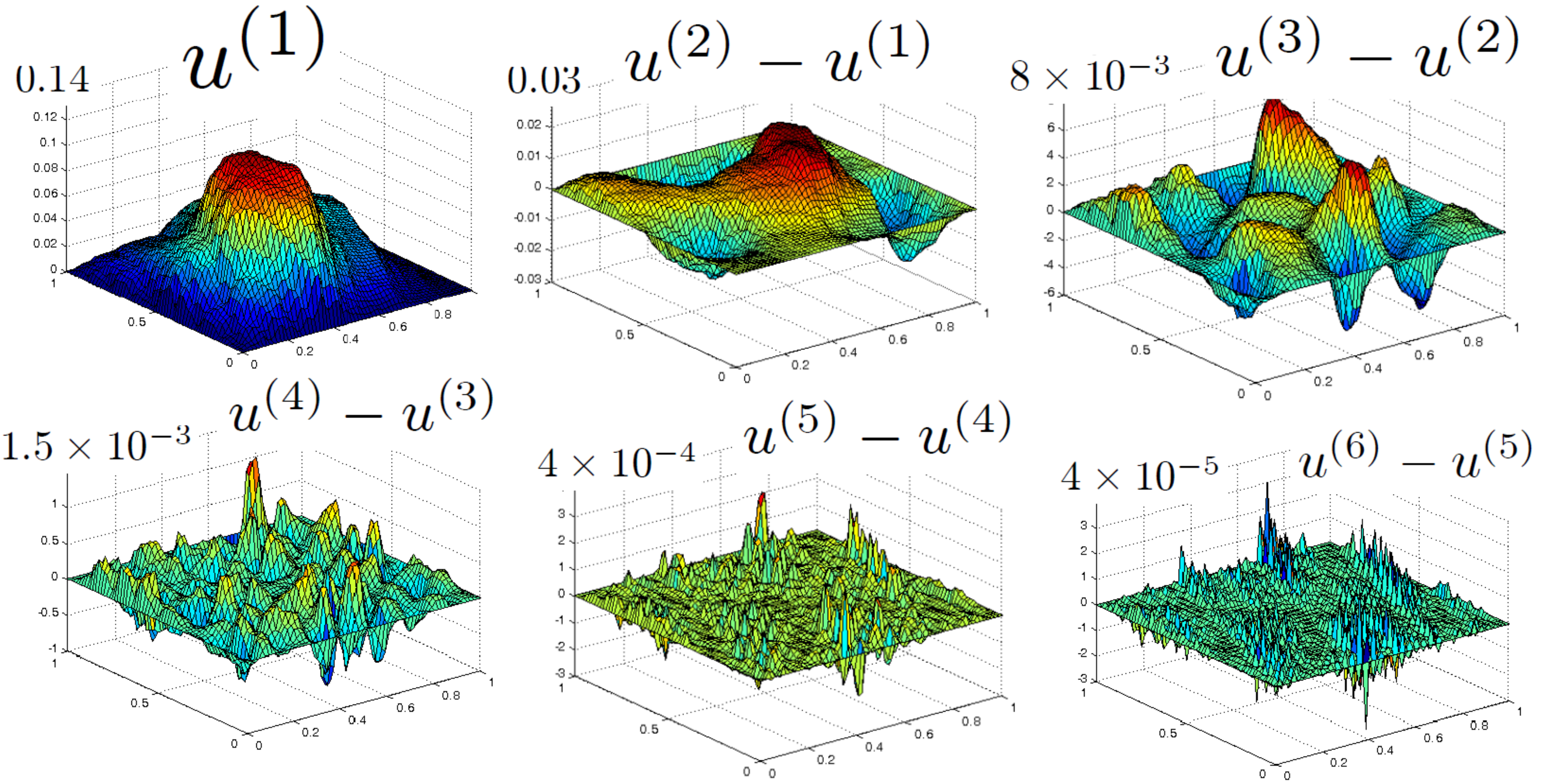}
		\caption{$u^{(1)}$, $u^{(2)}-u^{(1)}$,\ldots, and  $u^{(q)}-u^{(q-1)}$.}\label{fig:udiff}
	\end{center}
\end{figure}
 \begin{figure}[h!]
	\begin{center}
			\includegraphics[width=\textwidth]{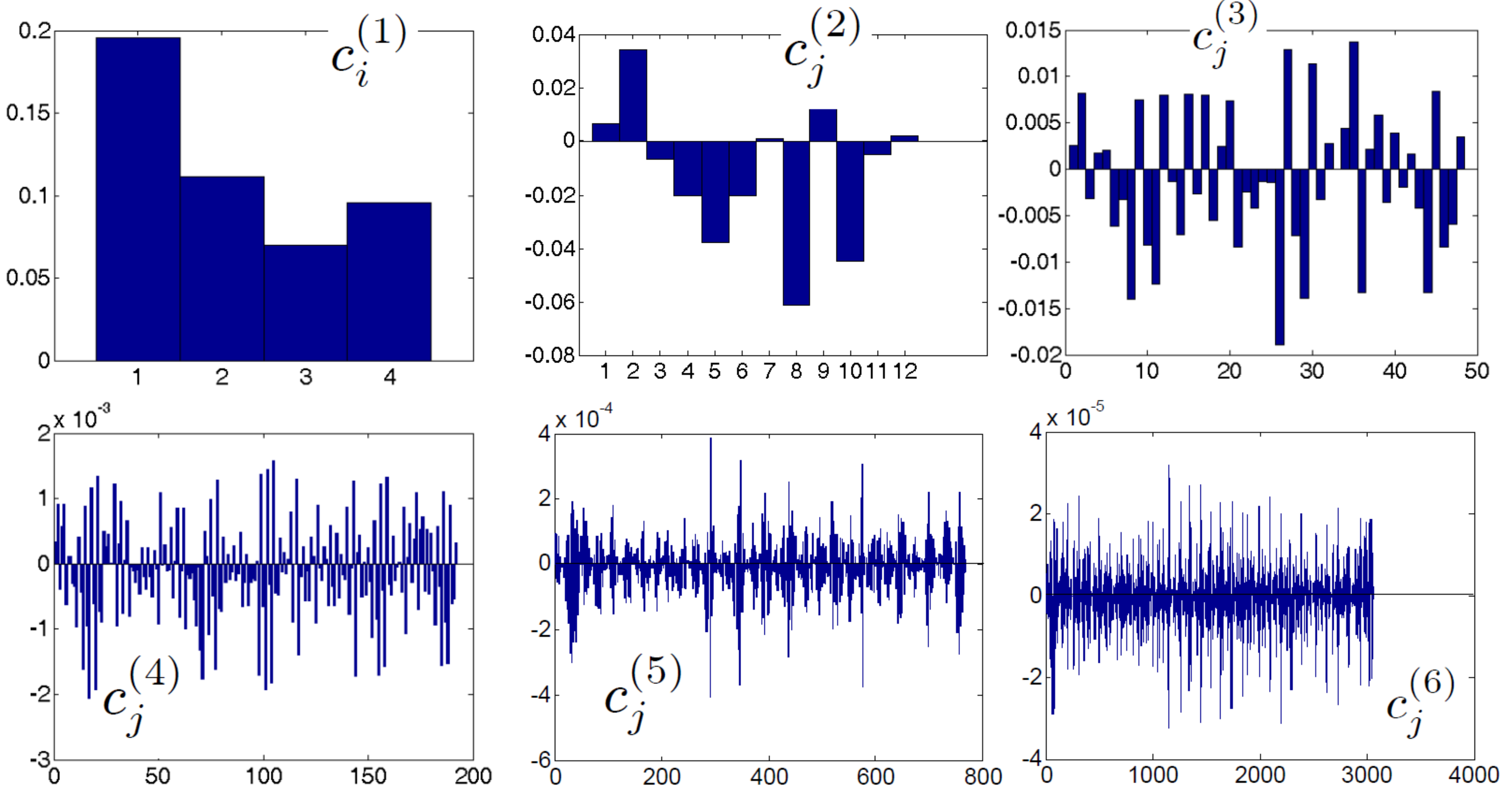}
		\caption{The coefficients of $u$ in the expansion $u=\sum_{i} c^{(1)}_i \frac{\psi^{(1)}_i}{\|\psi_i^{(1)}\|_a }+\sum_{k=2}^q \sum_j c^{(k)}_j \frac{\chi^{(k)}_j}{\|\chi_j^{(k)}\|_a }$.}\label{fig:coeffnormalized}
	\end{center}
\end{figure}
 \begin{figure}[h!]
	\begin{center}
			\includegraphics[width=\textwidth]{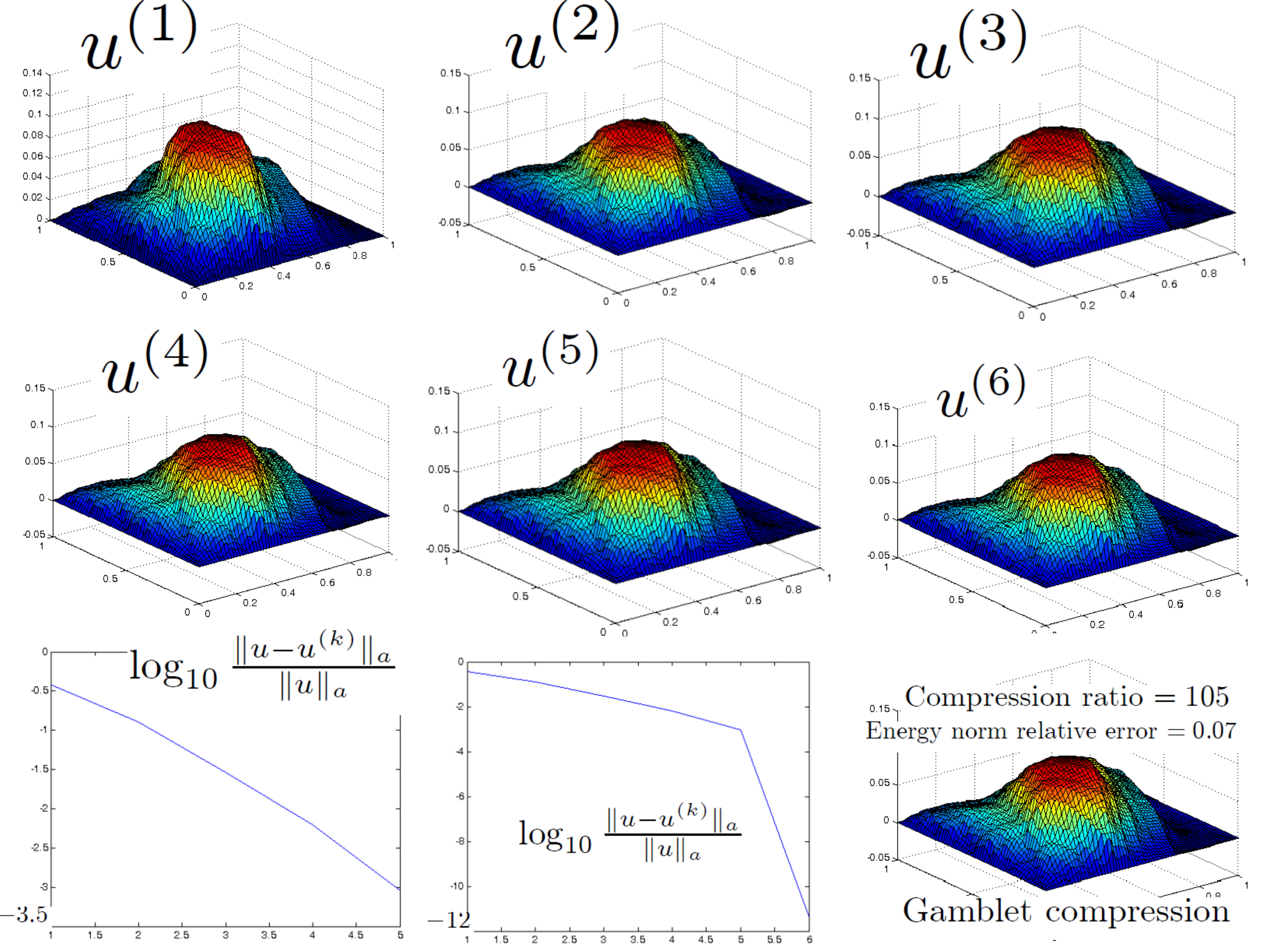}
		\caption{$u^{(1)}$, \ldots, $u^{(q)}$. Relative approximation error in energy norm in $\log_{10}$ scale. Compression of $u$ over the basis functions $\psi_i^{(1)}, \chi_i^{(2)}, \ldots, \chi_i^{(q)}$ by setting  $99$\% of the smallest coefficients to zero in the decomposition of Figure \ref{fig:coeffnormalized}.}\label{fig:uk}
	\end{center}
\end{figure}

\begin{Example}
We refer to figures \ref{fig:psi} and \ref{fig:chi} for an illustration of the gamblets $\psi_i^{(k)}$ and $\chi_j^{(k)}$  corresponding to Example \ref{ex1} with $W^{(k)}$ defined by Construction \ref{const1}. We refer to Figure \ref{fig:expdecay} for an illustration of the exponential decay of  the gamblets $\psi_i^{(k)}$.
 We refer to Figure \ref{fig:conditionnumbers} for an illustration of the condition numbers of $A^{(k)}$ and $B^{(k)}$ (with $W^{(k)}$ still defined by Construction \ref{const1}). Observe that the bound on the condition numbers of  $B^{(k)}$ depends on the contrast and the saturation of that bound occurs for smaller values of $k$ under low contrast. We refer to Figure \ref{fig:udiff} for an illustration of the subband solutions $u^{(1)}, u^{(2)}-u^{(1)},\ldots,u^{(q)}-u^{(q-1)}$ corresponding to Example \ref{ex1}. Observe that these (subband) solutions form a multiresolution decomposition of $u$ as a sum of functions characterising the behavior of $u$  at subscales $[H,1]$, $[H^2,H]$,\ldots,$[H^q,H^{q-1}]$.
Once the components $u^{(1)}$, $u^{(2)}-u^{(1)}$,\ldots, and  $u^{(q)}-u^{(q-1)}$ have been computed one obtains, via simple summation, $u^{(1)}$, \ldots, $u^{(q)}$, the finite-element approximation of $u$ at resolutions $H$, $H^2$, \ldots, $H^q$ illustrated in Figure \ref{fig:uk}. As described in Theorem \ref{thmgugyug0} the  error of the approximation of $u$ by $u^{(k)}$ is proportional to $H^k$ for $k\in \{1,\ldots,q-1\}$. For $k=q$, as illustrated in Figure \ref{fig:uk}, this approximation error drops down to zero because there is no gap between $H^q$ and the fine mesh (i.e., $\psi^{(q)}_i$ and $\varphi_i$ span the same linear space in the discrete case).
Moreover, as illustrated in Figure \ref{fig:coeffnormalized}, the representation of $u$ in the basis formed by the functions $\frac{\psi^{(1)}_i}{\|\psi_i^{(1)}\|_a}$ and $\frac{\chi^{(k)}_j}{\|\chi_j^{(k)}\|_a}$ is sparse. Therefore, as illustrated in Figure \ref{fig:uk} one can compress $u$, in this basis, by setting the smallest coefficients to zero without loss in energy norm.
\end{Example}

\begin{algorithm}[!ht]
\caption{Fast Gamblet transform/solve.}\label{fastgambletsolve}
\begin{algorithmic}[1]
\STATE\label{line1} For $i,j\in \I^{(q)}$, $M_{i,j}=\int_{\Omega} \varphi_i \varphi_j$  \COMMENT{$\mathcal{O}(N)$}
\STATE\label{line2} For $i,j\in \I^{(q)}$, $A_{i,j}=\int_{\Omega} (\nabla \varphi_i)^T a \nabla \varphi_j$ \COMMENT{$\mathcal{O}(N)$}
\STATE\label{line2a}  $M^{i,\rho_q} M^{-1,\rho_q}_{\cdot,i}=\delta_{\cdot,i}$  \COMMENT{Def.~\ref{deflocinvm}, Thm.~\ref{tmdiscreteaccuracy}, $\mathcal{O}\big(N \rho_q^{d} \ln \max( \frac{1}{\epsilon},q )\big)$}
\STATE\label{line3} For $i\in \I^{(q)}$, $\psi^{(q),\loc}_i=\sum_{j \in  i^{\rho_q}} M^{-1,\rho_q}_{j,i} \varphi_j$  \COMMENT{$\mathcal{O}(N \rho_q^d)$}
\STATE\label{line4} For $i\in \I^{(q)}$, $g^{(q),\loc}_i=\int_{\Omega} \psi^{(q),\loc}_i g$      \COMMENT{$\mathcal{O}(N \rho_q^d)$}
\STATE\label{line5} For $i,j\in \I^{(q)}$, $A^{(q),\loc}_{i,j}= \int_{\Omega} (\nabla \psi_i^{(q),\loc})^T a \nabla \psi_j^{(q),\loc}$   \COMMENT{ $\mathcal{O}(N \rho_q^{2d})$}
\FOR{$k=q$ to $2$}
\STATE\label{line7} $B^{(k),\loc}= W^{(k)}A^{(k),\loc}W^{(k),T}$ \COMMENT{$\mathcal{O}(|\I^{(k)}| \rho_k^{d})$}
\STATE\label{line8} $ w^{(k),\loc}= (B^{(k),\loc})^{-1}  W^{(k)} g^{(k),\loc}$ \COMMENT{Thm.~\ref{tmdiscreteaccuracy}, $\mathcal{O}(|\I^{(k)}| \rho_k^{d} \ln \frac{1}{\epsilon})$}
\STATE\label{line9}  For $i\in \J^{(k)}$, $\chi^{(k),\loc}_i=\sum_{j \in \I^{(k)}} W_{i,j}^{(k)} \psi_j^{(k),\loc}$  \COMMENT{$\mathcal{O}(|\I^{(k)}| \rho_k^{d})$}
\STATE\label{line10} $u^{(k),\loc}-u^{(k-1),\loc}=\sum_{i\in \J^{(k)}} w^{(k),\loc}_i \chi^{(k),\loc}_i$ \COMMENT{$\mathcal{O}(N \rho_k^{d})$}
\STATE\label{line11}  $\Inv(B^{(k),\loc} D^{(k,k-1),\loc}=   -W^{(k)}A^{(k),\loc}\bar{\pi}^{(k,k-1)},\rho_{k-1})$ \COMMENT{Def.~\ref{defb}, Thm.~\ref{tmdiscreteaccuracy},  $\mathcal{O}(|\I^{(k)}| \rho_{k-1}^{d} \rho_k^d \ln \frac{1}{\epsilon})$}
\STATE\label{line12} $R^{(k-1,k),\loc}=\bar{\pi}^{(k-1,k)}+D^{(k-1,k),\loc}W^{(k)}$ \COMMENT{Def.~\ref{defb}, $\mathcal{O}(|\I^{(k-1)}| \rho_{k-1}^{d})$}
\STATE\label{line13} $A^{(k-1),\loc}= R^{(k-1,k),\loc}A^{(k),\loc}R^{(k,k-1),\loc}$ \COMMENT{ $\mathcal{O}(|\I^{(k-1)}| \rho_{k-1}^{2d} \rho_{k}^{d})$}
\STATE\label{line14} For $i\in \I^{(k-1)}$, $\psi^{(k-1),\loc}_i=\sum_{j \in  \I^{(k)}} R_{i,j}^{(k-1,k),\loc} \psi_j^{(k),\loc}$ \COMMENT{$\mathcal{O}(|\I^{(k-1)}| \rho_{k-1}^{d} )$}
\STATE\label{line15} $g^{(k-1),\loc}=R^{(k-1,k),\loc} g^{(k),\loc}$ \COMMENT{$\mathcal{O}(|\I^{(k-1)}| \rho_{k-1}^{d} )$}
\ENDFOR
\STATE\label{line16} $ U^{(1),\loc}=A^{(1),\loc,-1}g^{(1),\loc}$ \COMMENT{$\mathcal{O}(|\I^{(1)}| \rho_{1}^{d} \ln \frac{1}{\epsilon})$}
\STATE\label{line17} $u^{(1),\loc}=\sum_{i \in \I^{(1)}} U^{(1)}_i \psi^{(1)}_i$ \COMMENT{$\mathcal{O}(N \rho_1^{d})$}
\STATE\label{line18} $u^{\loc}=u^{(1),\loc}+(u^{(2),\loc}-u^{(1),\loc})+\cdots+(u^{(q),\loc}-u^{(q-1),\loc})$    \COMMENT{ $\mathcal{O}(N q)$}
\end{algorithmic}
\end{algorithm}

\subsection{Fast Gamblet transform/solve}

Algorithm \ref{fastgambletsolve}  achieves near linear complexity (1) in approximating the solution of \eqref{eqhuiuhiuhuiu} to a given level of accuracy $\epsilon$ and (2) in performing an approximate Gamblet transform (sufficient to achieve that level of accuracy). This fast algorithm
is obtained by localizing/truncating the linear systems corresponding to lines \ref{line2a} and \ref{line11} in Algorithm \ref{gambletsolve}.
We define these localization/truncation steps as follows. For $k\in \{1,\ldots,q\}$ and $i\in \I^{(k)}$  define $i^{\rho}$ as in Subsection \ref{sechierarloc} (i.e. as the subset of indices $j\in \I^{(k)}$ whose corresponding subdomains $\tau_j^{(k)}$ are at distance at most $H_k \rho$ from $\tau_i^{(k)}$).

\begin{Definition}\label{deflocinvm}
  For $i\in \I^{(q)}$, let $M^{(i,\rho_q)}$ be the $i^{\rho_q} \times i^{\rho_q}$ matrix defined by $M^{(i,\rho_q)}_{l,j}=M_{l,j}$ for $l,j\in i^{\rho_q}$. Let $e^{(i,\rho_q)}$ be the $|i^{\rho_q}|$-dimensional vector defined by $e_j^{(i,\rho_q)}=\delta_{j,i}$ for $j\in i^{\rho_q}$. Let $y^{(i,\rho_q)}$ be the $|i^{\rho_q}|$-dimensional vector solution of $M^{(i,\rho_q)} y^{(i,\rho_q)}=e^{(i,\rho_q)}$.
We define the solution $M^{-1,\rho_q}_{\cdot,i}$ of the localized linear system $M^{i,\rho_q} M^{-1,\rho_q}_{\cdot,i}=\delta_{\cdot,i}$ (Line \ref{line2a} of Algorithm \ref{fastgambletsolve}) as the $i^{\rho_q}$-vector given by  $M^{-1,\loc}_{j,i}= y_j^{(i,\rho_q)}$ for $j\in i^{\rho_q}$.
\end{Definition}
Note that the associated  gamblet  $\psi^{(q),\loc}_i$ (Line \ref{line3} of Algorithm \ref{fastgambletsolve}) is also the solution of the problem of finding $\psi\in \Span\{\varphi_j \mid j\in i^{\rho_q}\}$ such that $\int_{\Omega}\psi \varphi_j=\delta_{i,j}$ for $j\in i^{\rho_q}$ (i.e. localizing the computation of the gamblet $\psi_i^{(q)}$ to a subdomain of size $H_q \rho_q$). Line \ref{line4} can  be replaced by  $g^{(q),\loc}_i=g_i$ without loss of accuracy ($g^{(q),\loc}_i=\int_{\Omega} \psi^{(q),\loc}_i g$ simplifies the presentation of the analysis).
Line  \ref{line11} of Algorithm \ref{fastgambletsolve} is defined in a similar way as follows.
\begin{Definition}\label{defb}
Let $k\in \{2,\ldots,q\}$ and $B$ be the positive definite $\J^{(k)}\times \J^{(k)}$ matrix $B^{(k),\loc}$ computed in Line \ref{line7} of Algorithm \ref{fastgambletsolve}.
For $i\in \I^{(k-1)}$, let $\rho=\rho_{k-1}$ and let $i^\chi$ be the subset of indices $j\in \J^{(k)}$ such that $j^{(k-1)}\in i^{\rho}$   (recall that if $j=(i_1,\ldots,i_k)\in \J^{(k)}$ then $j^{(k-1)}:=(j_1,\ldots,j_{k-1})\in \I^{(k-1)}$).
$B^{(i,\rho)}$ be the $i^{\chi}\times i^{\chi}$ matrix defined by $B^{(i,\rho)}_{l,j}=B_{l,j}$ for $l,j \in i^{\chi}$.
 Let $b^{(i,\rho)}$ be the $|i^{\chi}|$-dimensional vector defined by $b_j^{(i,\rho)}=-(W^{(k)}A^{(k),\loc}\bar{\pi}^{(k,k-1)})_{j,i}$ for $j\in i^{\chi}$. Let $y^{(i,\rho)}$ be the $|i^{\chi}|$-dimensional vector solution of
$B^{(i,\rho)} y^{(i,\rho)}=b^{(i,\rho)}$.
We define  the solution $D^{(k,k-1),\loc}$ of the localized linear system $\Inv(B^{(k),\loc} D^{(k,k-1),\loc}=   -W^{(k)}A^{(k),\loc}\bar{\pi}^{(k,k-1)},\rho_{k-1})$
 as the $\J^{(k)}\times \I^{(k-1)}$ sparse matrix  given by $D^{(k,k-1),\loc}_{j,i}=0$ for $j\not \in i^{\chi}$ and $D^{(k,k-1),\loc}_{j,i}= y_j^{(i,\rho)}$ for $j\in i^{\chi}$. $D^{(k-1,k),\loc}$ (Line \ref{line12} of Algorithm \ref{fastgambletsolve}) is then defined as the transpose of $D^{(k,k-1),\loc}$.
\end{Definition}

\begin{Remark}\label{rmklocalgfast} Definition \ref{defb} (Line \ref{line7} of Algorithm \ref{fastgambletsolve}) is equivalent to localizing the computation of each gamblet $\psi_i^{(k-1)}$ to a subdomain of size $H_{k-1} \rho_{k-1}$, i.e., the   gamblet  $\psi^{(k-1),\loc}_i$ computed in Line \ref{line14} of Algorithm \ref{fastgambletsolve} is  the solution of (1) the problem of finding $\psi$ in the affine space $\sum_{j\in \I^{(k)}}\bar{\pi}_{i,j}^{(k-1,k)}\psi_j^{(k),\loc}+\Span\{\chi_j^{(k),\loc} \mid j^{(k-1)}\in i^{\rho_{k-1}}\}$  such that $\psi$ is $\<\cdot,\cdot\>_a$ orthogonal to $\Span\{\chi_j^{(k),\loc} \mid j^{(k-1)}\in i^{\rho_{k-1}}\}$, and (2)
 the problem of minimizing $\|\psi\|_a$ in $\Span\{\psi_l^{(k),\loc} \mid l^{(k-1)}\in i^{\rho_{k-1}}\}$ subject to constraints $\int_{\Omega} \phi_j^{(k-1)} \psi=\delta_{i,j}$ for $j\in i^{\rho_{k-1}}$.
\end{Remark}

\subsection{Complexity vs accuracy of Algorithm \ref{fastgambletsolve} and choice of the localization radii $\rho_k$}\label{subseccomplexity}
The sizes of the localization radii $\rho_k$ (and therefore the complexity of Algorithm \ref{fastgambletsolve}) depend on whether Algorithm \ref{fastgambletsolve} is used as a pre-conditioner (as done with AMG) or as a direct solver.
Although it is natural to expect the complexity of Algorithm \ref{fastgambletsolve} to be significantly smaller if used as pre-conditioner (since pre-conditioning requires lower accuracy and therefore smaller localization radii) we will restrict our analysis and presentation to using Algorithm \ref{fastgambletsolve} as a direct solver.  Note that, when used as a direct solver,  Algorithm \ref{fastgambletsolve} is  parallel both in space (via localization) and in bandwith/subscale (subscales can be computed independently from each other and  $\psi^{(k-1),\loc}_i$ and  $u^{(k),\loc}-u^{(k-1),\loc}$ can be resolved in parallel).
We will base our analysis on the results of Subsection \ref{sechierarloc} and in particular Theorem \ref{tmshjgeydg}.
Although obtained in a continuous
setting, these results  can be generalized to the discrete setting without difficulty. Two small differences are worth mentioning.
(1) In this discrete setting, an alternative approach for obtaining localization error bounds in  the first step of the algorithm (the computation of the localized gamblets $\psi_i^{(q),\loc}$) is to use the exponential decay property of the inverse of symmetric well-conditioned banded matrices \cite{Demko1984}: since $M$ is banded and of uniformly bounded condition number \cite{Demko1984} (see also \cite[Thm~4.10]{Bebendorf:2008}) implies that $M_{i,j}^{-1}$ decays like $\exp\big(-\dist(\tau_i^{(q)},\tau_j^{(q)})/C\big)$ which guarantees that the bound $\er(q)\leq C H^{-d/2-q(2+d/2)}e^{-\rho_q/C}$ (used in Theorem \ref{thmdjjuud}) remains valid in the discrete setting. (2) Since the basis functions $\varphi_i$ are not exact set functions, neither are the resulting aggregates $\phi_i^{(k)}$. This implies that, in the discrete setting, $\int_{\Omega}\psi_i^{(k),\loc} \phi_j^{(k)}$ is not necessarily equal to zero if $\tau_j^{(k)}$ is adjacent to $S^i_{\rho_k}$ (with $j\not\in i^{\rho_k}$, using the notation  of Subsection \ref{sechierarloc}). This, however does not prevent the generalization of the proof because the  value of $\int_{\Omega}\psi_i^{(k),\loc} \phi_j^{(k)}$  (when $\tau_j^{(k)}$ is adjacent to $S^i_{\rho_k}$) can be controlled via the exponential decay of the basis functions (e.g. as done in the proof of Theorem \ref{thm:hieuhdds}). We will summarize this generalization in the following theorem (where the constant $C$ depends on the constants $C_1, C_0, \bar{\gamma}$ and $\ubar{\gamma}$ associated with the finite elements $(\varphi_i)$ in \eqref{eqhhgfff65f}, in addition to $d, \Omega, \lambda_{\min}(a), \lambda_{\max}(a), \delta$).

\begin{Theorem}\label{tmdiscrete}
Let $u$ be the solution of the discrete system \eqref{eqhuiuhiuhuiu}. Let $u^{(1),\loc}$, $u^{(k),\loc}-u^{(k-1),\loc}$, $u^{\loc}$, $A^{(k),\loc}$ and $B^{(k),\loc}$ be the outputs of Algorithm \ref{fastgambletsolve}. Let $u^{(1)}$ and $u^{(k)}-u^{(k-1)}$  be the outputs of Algorithm \ref{gambletsolve}.
For $k\in \{2,\ldots,q\}$, write $u^{(k),\loc}:=u^{(1),\loc}+\sum_{j=2}^k (u^{(j),\loc}-u^{(j-1),\loc})$.
Let $\epsilon \in (0,1)$. It holds true that if $\rho_k\geq C \big((1+\frac{1}{\ln(1/H)})\ln \frac{1}{H^k}+\ln \frac{1}{\epsilon}\big)$ for $k\in \{1,\ldots,q\}$ then (1) for $k\in \{1,\ldots,q-1\}$ we have
$\|u^{(k)} - u^{(k),\loc}\|_{H^1_0(\Omega)} \leq   \epsilon \|g\|_{H^{-1}(\Omega)}$ and $\|u^{(k)} - u^{(k),\loc}\|_{H^1_0(\Omega)} \leq  C (H^k+\epsilon) \|g\|_{L^2(\Omega)}$ (2) $\Cond(A^{(1),\loc})\leq C H^{-2}$, and for $k\in \{2,\ldots,q\}$ we have (3)
  $\Cond(B^{(k),\loc})\leq C H^{-2-2d}$ and (4) $\|u^{(k)}-u^{(k-1)}-(u^{(k),\loc}-u^{(k-1),\loc})\|_{H^1_0(\Omega)} \leq \frac{\epsilon}{2 k^2} \|g\|_{H^{-1}(\Omega)}$. Finally, (6) $\|u - u^{\loc}\|_{H^1_0(\Omega)} \leq  \epsilon \|g\|_{H^{-1}(\Omega)}$.
\end{Theorem}

Therefore, according to Theorem \ref{tmdiscrete} if the localization radii $\rho_k$ are chosen so that $\rho_k=\mathcal{O}\big(\ln \max(1/\epsilon, 1/H_k)\big)$ for $k\in \{1,\ldots,q\}$ then the condition numbers of the matrices $B^{(k),\loc}$ and $A^{(1),\loc}$ remain uniformly bounded and the algorithm achieves accuracy $\epsilon$ in a direct solve. The following theorem shows that the linear systems appearing in lines  \ref{line2a}, \ref{line8} and \ref{line11} of Algorithm \ref{fastgambletsolve} do not need to be solved exactly and provide bounds on the accuracy requirements (to simplify notations, we will from now on drop the superscripts of the vectors $y$ and $b$ appearing in definitions \ref{deflocinvm} and \ref{defb}).
\begin{Theorem}\label{tmdiscreteaccuracy}
The results of Theorem \ref{tmdiscrete} remain true if (1) $\rho_k\geq C \big((1+\frac{1}{\ln(1/H)})\ln \frac{1}{H^k}+\ln \frac{1}{\epsilon}\big)$ for $k\in \{1,\ldots,q\}$ (2) For each $i\in \I^{(q)}$ the localized linear system $M^{i,\rho_q} y=\delta_{\cdot,i}$ of Definition \ref{deflocinvm} and Line \ref{line2a} of Algorithm \ref{fastgambletsolve} is solved up to  accuracy $|y-y^{\app}|_{M^{i,\rho_q}}\leq C^{-1} H^{7d/2+3} \epsilon/q^2$ (using the notations of Subsection  \ref{subseccg}, i.e.  $|e|_A^2:=e^T A e$, and writing $y^{\app}$ the approximation of $y$) (3) For $k\in \{2,\ldots,q\}$ and each $i\in \I^{(k-1)}$, the localized linear system $B^{(i,\rho)} y=b$ of Definition \ref{defb} and Line \ref{line11} of Algorithm \ref{fastgambletsolve} is solved up to accuracy
$|y-y^{\app}|_{B^{(i,\rho)}} \leq C^{-1} H^{-k+7d/2+4}\epsilon/(k-1)^2$. (4) For $k\in \{2,\ldots,q\}$ the linear system $ B^{(k),\loc}y=  W^{(k)} g^{(k),\loc}$ of Line \ref{line8} of  Algorithm \ref{fastgambletsolve} is solved up to accuracy $|y-y^{\app}|_{B^{(k),\loc}}\leq  \epsilon \|g\|_{H^{-1}(\Omega)}/(2q)$.
\end{Theorem}
\begin{proof}
From the proof of Theorem \eqref{tmshjgeydg} we need $\er(k)  \leq C^{-1}  H^{-k (1+d/2)+7d/2+3} \epsilon/k^2$ for $k\in \{1,\ldots,q\}$. By the inverse Poincar\'{e} inequality \eqref{eqinvpoincdiscrete} this inequality  is satisfied for $k=q$ for
$\|\psi_i^{(q)}-\psi_i^{(q),\loc}\|_{L^2(\Omega)} \leq C^{-1} H^{7d/2+3} \epsilon/q^2 $ for each $i\in \I^{(q)}$, which by the definition of  $M^{i,\rho_q}$ and Line \ref{line3} of Algorithm \ref{fastgambletsolve} leads to (2).
For $k\in \{2,\ldots,q\}$ the inequality $\er(k-1)  \leq C^{-1}  H^{-(k-1) (1+d/2)+7d/2+3} \epsilon/(k-1)^2$ is satisfied if for $i\in \I^{(k-1)}$,
$\|\psi_i^{(k-1)}-\psi_i^{(k-1),\loc}\|_{a} \leq C^{-1} H^{-(k-1) +7d/2+3}\epsilon/(k-1)^2$. Using the notations of Definition \ref{defb} we have,
$\psi^{(k-1),\loc}_i=\sum_{j\in \I^{(k)}}\bar{\pi}^{(k-1,k)}_{i,j}\psi_j^{(k),\loc} + \sum_{j \in  i^\chi} D_{i,j}^{(k-1,k),\loc} \chi_j^{(k),\loc}$ with
$\<\chi_j^{(k),\loc},\chi_l^{(k),\loc}\>_a=B^{(i,\rho)}_{j,l}$ which leads to (3) by lines \ref{line14}, \ref{line12}, \ref{line9} and \ref{line7}  of Algorithm \ref{fastgambletsolve}. For (4) we simply observe that for $y\in \J^{(k)}$, $\|\sum_{i\in \J^{(k)}} (y-y^\app)_i\chi_i^{(k),\loc}\|_a=|y-y^\app|_{B^{(k),\loc}}$.
\end{proof}

Let us now describe the complexity of Algorithm \ref{fastgambletsolve}. This complexity depends on the desired accuracy $\epsilon \in (0,1)$.
Lines \ref{line1} and \ref{line2} correspond to the computation of the (sparse) mass and stiffness matrices of \eqref{eqhuiuhiuhuiu}. Note since $A$ and $M$ are sparse and banded (of bandwidth $2d=4$ in our numerical example) this computation is of  $\mathcal{O}(N)$ complexity.
Line \ref{line2a} corresponds to the resolution of the localized linear system introduced in Definition \ref{deflocinvm} using $M^{i,\rho_q}$, the $i^{\rho_q}\times i^{\rho_q}$ sub-matrix of $M$. According to Theorem \ref{tmdiscrete}, the accuracy of each solve must be $|y-y^{\app}|_{M^{i,\rho_q}}\leq C^{-1} H^{7d/2+3} \epsilon/q^2$.
Since $|i^{\rho_q}|=\mathcal{O}(\rho_q^d)$ and since $M^{i,\rho_q}$ is of condition number bounded by that of $M$, for each $i$ the linear system of Line \ref{line2a} can be solved efficiently (to accuracy $\mathcal{O}(C^{-1} H^{7d/2+3} \epsilon/q^2)$ using $\mathcal{O}(\rho_q)=\mathcal{O}\big(\ln \max( \frac{1}{\epsilon},q)\big)$ iterations of the CG method (reminded in Subsection \ref{subseccg}) with a cost of $\mathcal{O}(\rho_q^d)$ per iteration, which results in a total  cost of $\mathcal{O}\big(N \rho_q^{d} \ln \max( \frac{1}{\epsilon},q )\big)$. Lines \ref{line3} and \ref{line4} are naturally of complexity $\mathcal{O}(N \rho_q^d)$. Since $A^{(q),\loc}_{i,j}=0$ if $\tau_i^{(q)}$ and $\tau_j^{(q)}$ are at a distance larger than $2 H^q \rho_q$ the complexity of Line \ref{line5} is $\mathcal{O}(N \rho_q^{2d})$. Note that $A^{(k),\loc}$ and $B^{(k),\loc}$ are banded and of bandwidth $\mathcal{O}(N \rho_k^{d})$. It follows that Line \ref{line7} is of complexity $\mathcal{O}(|\I^{(k)}| \rho_k^{d})$. According to Theorem \ref{tmdiscreteaccuracy} the linear system of Line \ref{line8} needs to be solved up to accuracy  $|y-y^{\app}|_{B^{(k),\loc}}\leq  \epsilon \|g\|_{H^{-1}(\Omega)}/2$. Since $B^{(k),\loc}$ is of uniformly bounded condition number this can be done using $\mathcal{O}(\ln \frac{1}{\epsilon})$ iterations of the CG method with a cost of $\mathcal{O}(|\I^{(k)}| \rho_k^d)$ per iteration (using $\mathcal{O}(|\J^{(k)}|)=\mathcal{O}(|\I^{(k)}|)$), which results in a total  cost of $\mathcal{O}(|\I^{(k)}| \rho_k^{d} \ln \frac{1}{\epsilon})$ for Line \ref{line8}. Storing the fine mesh values of   $u^{(k),\loc}-u^{(k-1),\loc}$  in Line \ref{line10} costs $\mathcal{O}(N \rho_k^{d})$ (since for each node $x$ on the fine mesh only $\mathcal{O}(\rho_k^{d})$ localized basis functions contribute to the value of $u^{(k),\loc}-u^{(k-1),\loc}$).
According to Theorem \ref{tmdiscreteaccuracy}, for each $i\in \I^{(k-1)}$ the linear system $B^{(i,\rho)} y=b$ of  Line \ref{line11} needs to be solved up to accuracy $|y-y^{\app}|_{B^{(i,\rho)}} \leq C^{-1} H^{-k+7d/2+4}\epsilon/(k-1)^2$. Since the matrix $B^{(i,\rho)}$ inherits the uniformly bounded condition number from $B^{(k),\loc}$ this can be done using $\mathcal{O}(\ln \frac{1}{\epsilon})$ iterations of the CG method with a cost of $\mathcal{O}(H^{-d}\rho_{k-1}^d \rho_k^d)=\mathcal{O}(\rho_{k-1}^d \rho_k^d)$ per iteration. This results in a total cost of $\mathcal{O}(|\I^{(k)}| \rho_{k-1}^{d} \rho_k^d \ln \frac{1}{\epsilon})$ for Line \ref{line11}. We obtain, using the sparsity structures of
$D^{(k-1,k),\loc}$ and $R^{(k-1,k),\loc}$  that  the complexity of Line \ref{line12} is
$\mathcal{O}(|\I^{(k-1)}| \rho_{k-1}^{d} H^{-d})=\mathcal{O}(|\I^{(k-1)}| \rho_{k-1}^{d} )$ and that of Line \ref{line13} is $\mathcal{O}(|\I^{(k-1)}| \rho_{k-1}^{2d} \rho_{k}^{d})$. The complexity of lines \ref{line14} to \ref{line15} is summarized in the display of Algorithm \ref{fastgambletsolve} and a simple consequence of the sparsity structure of $R^{(k-1,k),\loc}$. Line \ref{line16} is complexity $\mathcal{O}(|\I^{(1)}| \rho_{1}^{d} \ln \frac{1}{\epsilon})$ (using CG as in Line \ref{line8}). As in Line \ref{line10}, storing the values of $u^{(1),\loc}$ costs $\mathcal{O}(N \rho_{1}^{d} )$. Finally, obtaining $u^{\loc}$ in Line \ref{line18} costs $\mathcal{O}(N q)$ (observe that $q=\mathcal{O}(\ln N)$).

\begin{table}[!h]
\begin{center}
\begin{tabular}{ l || c | r }
  \hline			
 Compute and store $\psi_i^{(k),\loc}$, $\chi_i^{(k),\loc}$, $A^{(k),\loc}$, $B^{(k),\loc}$   & $\epsilon\leq H^{q}$ & $\epsilon\geq H^{q}$\\
 and $u^{\loc}$ s.t. $\|u - u^{\loc}\|_{H^1_0(\Omega)} \leq  \epsilon \|g\|_{H^{-1}(\Omega)}$ & &  \\\hline
  First solve & $N \ln^{3d} \frac{1}{\epsilon} $  & $N \ln^{3d} N$ \\\hline
  Subsequence solves & $N \ln^{d+1} \frac{1}{\epsilon}$  & $N \ln^d N \ln \frac{1}{\epsilon} $ \\ \hline  \hline
  Subsequent solves to compute  $u^{(k),\loc}$ s.t. & & $N \ln^{d+1}  \frac{1}{\epsilon}$ \\
   $\|u - u^{(k),\loc}\|_{H^1_0(\Omega)} \leq  C \epsilon \|g\|_{L^2(\Omega)}$ & & \\\hline \hline
  Subsequent solves to compute the coefficients $c_i^{(k)}$  & & \\ of $u^{(k),\hom}=\sum_{i\in \I^{(k)}} c_i^{(k)} \psi_i^{(k)}$  & &$\epsilon^{-d} \ln^{d+1}  \frac{1}{\epsilon}$ \\
  s.t.  $\|u - u^{(k),\loc}\|_{H^1_0(\Omega)} \leq  C \epsilon (\|g\|_{L^2(\Omega)}+\|g\|_{\textrm{Lip}})$  & &  \\
  \hline  \hline
  Subsequent solves to compute  $u^{(k),\hom}$ s.t. & &$\epsilon^{-d} \ln^{d+1}  \frac{1}{\epsilon}$ \\
  $\|u-u^{(k),\homo}\|_{L^2(\Omega)}\leq C \epsilon  (\|g\|_{L^2(\Omega)}+\|g\|_{\textrm{Lip}})$ & & \\ \hline \hline
  $a$ periodic/ergodic with mixing length $H^p\leq \epsilon$, & & $(N (\ln^{3d} N) H^p$ \\
  first solve of $u^{(k),\hom}$ s.t. & & $+\epsilon^{-d})$   \\
  $\|u-u^{(k),\homo}\|_{L^2(\Omega)}\leq C \epsilon  (\|g\|_{L^2(\Omega)}+\|g\|_{\textrm{Lip}})$ & & $ \ln^{d+1}  \frac{1}{\epsilon}$ \\ \hline
\end{tabular}
 \end{center}
  \caption{Complexity of Algorithm \ref{fastgambletsolve}.}
    \label{tabcomplexity}
\end{table}

\paragraph{Total computational complexity, first solve.}
 Summarizing we obtain that the complexity of Algorithm \ref{fastgambletsolve}, i.e. the cost of computing  the gamblets $(\psi_i^{(k),\loc})$, $(\chi_j^{(k),\loc})$, their stiffness matrices $(A^{(k),\loc},B^{(k),\loc})$, and the approximation   $u^{\loc}$  such that
 $\|u - u^{\loc}\|_{H^1_0(\Omega)} \leq  \epsilon \|g\|_{H^{-1}(\Omega)}$ is $\mathcal{O}\big(N \big(\ln \max (\frac{1}{\epsilon},N^\frac{1}{d})\big)^{3d}\big)$ (with Line \ref{line13} being the corresponding complexity bottleneck). The complexity of storing the gamblets $(\psi_i^{(1),\loc})$, $(\chi_j^{(k),\loc})$ and their stiffness matrices $(A^{(1),\loc},B^{(k),\loc})$ is $\mathcal{O}\big(N \big(\ln \max (\frac{1}{\epsilon},N^\frac{1}{d})\big)^{d}\big)$.

 \paragraph{Computational complexity of subsequent solves with $g\in H^{-1}(\Omega)$.}
If  \eqref{eqhuiuhiuhuiu} (i.e. \eqref{eqn:scalar}) needs to be solved for more than one $g$ then the gamblets $\psi_i^{(k),\loc}, \chi^{(k),\loc}_i$ and the stiffness matrices $B^{(k),\loc}$ do not need to be recomputed. The cost of subsequent solves is therefore that of Line \ref{line8} i.e. $\mathcal{O}\big(N \big(\ln \max (\frac{1}{\epsilon},N^\frac{1}{d})\big)^{d} \ln \frac{1}{\epsilon}\big)$  to achieve the approximation accuracy $\|u - u^{\loc}\|_{H^1_0(\Omega)} \leq  \epsilon \|g\|_{H^{-1}(\Omega)}$.

 \paragraph{Computational complexity of subsequent solves with $g\in L^2(\Omega)$ and $\epsilon \geq H^q$.}
If $g\in L^2(\Omega)$ (i.e. if $\|g\|_{L^2(\Omega)}$ is used to express the accuracy of the approximation)
and $\epsilon \in [H^k,H^{k-1}]$ then, by Theorem \ref{tmdiscrete}, $u^{(k),\loc}$ achieves the approximation accuracy $\|u - u^{(k),\loc}\|_{H^1_0(\Omega)} \leq  C \epsilon \|g\|_{L^2(\Omega)}$ (i.e. $u^{(j+1),\loc}-u^{(j),\loc}$ does not need to be computed for $j\geq k$) and the corresponding complexity is
 $\mathcal{O}\big(N (\ln  \frac{1}{\epsilon})^{d+1}\big)$ (if $g\in H^{-1}(\Omega)$ then the energy of the solution can be in the fine scales and $u^{(j+1),\loc}-u^{(j),\loc}$ do  need to be computed for $j\geq k$).

 \paragraph{Computational complexity of subsequent solves with $g$ Lipschitz continuous and $\epsilon \geq H^q$.}
Note that the computational complexity bottleneck for computing the coefficients of $u^{(k),\loc}$ in the basis $(\psi_i^{(k),\loc})$ when $g\in \L^2(\Omega)$ and $\epsilon \in [H^k,H^{k-1}]$ is in the computation of the vectors $g^{(j),\loc}$ for $j>k$. If $g$ is Lipschitz continuous then $g^{(k),\loc}_i$ be approximated $g(x_i^{(k)})$ where $x_i^{(k)}$ is any point in $\tau_i^{(k)}$ without loss of accuracy. Note that this approximation requires (only)  $\mathcal{O}(H^{-kd})$ evaluations of $g$ and leads to a corresponding $u^{(k),\loc}$ satisfying $\|u - u^{(k),\loc}\|_a \leq  C \epsilon (\|g\|_{L^2(\Omega)}+\|g\|_{\textrm{Lip}})$ (with $\|g\|_{\textrm{Lip}}=\sup_{x,y\in \Omega}|g(x)-g(y)|/|x-y|$). Therefore the computational complexity of subsequent solves to obtain the coefficients $c_i^{(k)}$ in the decomposition $u^{(k),\loc}=\sum_{i\in \I^{(k)}}c_i^{(k)} \psi_i^{(k),\loc}$ is $\mathcal{O}\big(\epsilon^{-d} (\ln  \frac{1}{\epsilon})^{d+1}\big)$ (i.e. independent from $N$ if $g$ is Lipschitz continuous).
Of course, obtaining an $H^1_0(\Omega)$-norm approximation of $u$ with accuracy $H^k$ requires expressing the values of $\psi_i^{(k),\loc}$ (and therefore $u^{(k),\loc}$) on the fine mesh, which leads to a total cost of  $\mathcal{O}(N (\ln \frac{1}{\epsilon})^d )$. However if one is only interested expressing the values of $u^{(k),\loc}$ on the fine mesh in a sub-domain of diameter $\epsilon$ then the resulting complexity is  $\mathcal{O}((N \epsilon^d+\epsilon^{-d}) (\ln \frac{1}{\epsilon})^d )$

\paragraph{Computational complexity of subsequent $L^2$-approximations with $g$ Lipschitz continuous and $\epsilon \geq H^q$.}
Let $(x_i^{(k)})_{i\in \I^{(k)}}$ be points of $(\tau_i^{(k)})_{i\in \I^{(k)}}$ forming a regular coarse mesh of $\Omega$ of resolution $H^k$ and write $\varphi_i^{(k)}$ the corresponding (regular and coarse) piecewise linear nodal basis elements.
If (as in classical homogenization or  HMM) one is only interested in an $L^2$-norm approximation of $u$ with accuracy $H^k$ then the coefficients $c_i^{(k)}$ defined above are sufficient to obtain the approximation
 $u^{\homo}=\sum_{i\in \I^{(k)}} \frac{c_i^{(k)}}{\int_{\Omega}\phi_i^{(k)}} \varphi_i^{(k)}$  that satisfies $\|u^{(k),\loc}-u^{\homo}\|_{L^2(\Omega)}\leq C H^k \|u^{(k),\loc}\|_{H^1_0(\Omega)}$ ($\int_{\Omega}u^{\homo}\phi_i^{(k)}=\int_{\Omega}u^{(k),\loc}\phi_i^{(k)}$) and therefore $\|u-u^{(k),\homo}\|_{L^2(\Omega)}\leq C \epsilon  (\|g\|_{L^2(\Omega)}+\|g\|_{\textrm{Lip}})$.
Note that the computational complexity of subsequent solves to obtain $u^{\homo}$ is $\mathcal{O}\big(\epsilon^{-d} (\ln  \frac{1}{\epsilon})^{d+1}\big)$.

\paragraph{Total computational complexity if $a$ is periodic or ergodic with mixing length $H^p$ and $\epsilon \approx H^{k}$ with $k\geq p$.}
Under the assumptions of classical homogenization or HMM \cite{EEngquist:2003} (e.g. $a$ is of period $H^p$ or $a$ is ergodic with $H^p$ as mixing length), if the sets $\tau_i^{(k)}$ are chosen to match the period of $a$ and the domain is rescaled so that $1/H$ is an integer, then  the entries of $A^{(k)}$ are invariant under periodic translations (or stationary if the medium is ergodic). Therefore, under these assumptions, as in classical homogenization, it sufficient to limit the computation of gamblets to  periodicity cells (or ergodicity cells with a tight control on mixing as in \cite{GloriaNeukmanOtto2015}).
The resulting cost of obtaining   $u^{(k),\homo}$ (in a first solve) such that  $\|u^{(k),\loc}-u^{(k),\homo}\|_{L^2(\Omega)}\leq C \epsilon  (\|g\|_{L^2(\Omega)}+\|g\|_{\textrm{Lip}})$, is $\mathcal{O}\big(N \ln^{3d}N \,H^p+\epsilon^{-d}) \ln^{d+1}  \frac{1}{\epsilon}\big)$.

\paragraph{Acknowledgements.}
The author gratefully acknowledges this work supported by  the Air Force Office of Scientific Research and the DARPA EQUiPS Program under
awards number  FA9550-12-1-0389 (Scientific Computation of Optimal Statistical Estimators) and number FA9550-16-1-0054 (Computational Information Games) and the U.S. Department of Energy Office of Science, Office of Advanced Scientific Computing Research, through the Exascale Co-Design Center for Materials in Extreme Environments (ExMatEx, LANL Contract No
DE-AC52-06NA25396, Caltech Subcontract Number 273448). The author also thanks  C. Scovel, L. Zhang, P. B. Bochev, P. S. Vassilevski, J.-L. Cambier, B. Suter, G. Pavliotis and an anonymous referee for valuable comments and suggestions.

\bibliographystyle{plain}
\bibliography{RPS}

\end{document}